\documentclass[12pt,a4paper]{article}
\usepackage{mathrsfs,amsmath,amsfonts,amssymb,comment}
\usepackage{MnSymbol,mathtools,bm,xcolor,enumerate}
\usepackage{indentfirst}
\usepackage[compress,sort]{cite}
\usepackage[pdftex,bookmarks,colorlinks]{hyperref}
\hypersetup{citecolor=cyan,linkcolor=blue}
\usepackage{tabularx,multirow,array,booktabs}
\usepackage{colortbl}%%% color in tables
\usepackage{graphicx}%% 
\graphicspath{{figures/}}%
\usepackage{float}%% 
\usepackage{subcaption} 
\makeatletter\@addtoreset{equation}{section}\makeatother
\usepackage{amsthm}
\theoremstyle{plain}
\newtheorem{theorem}{Theorem}[section]%  
\newtheorem{lemma}[theorem]{Lemma}
\newtheorem*{lemma*}{Lemma}
\newtheorem{claim}[theorem]{Claim}
\newtheorem{proposition}[theorem]{Proposition}
\theoremstyle{definition}

\theoremstyle{remark}

\newcommand{\R}{\ifmmode\mathbb{R}\else$\mathbb{R}$\fi}
\newcommand{\N}{\ifmmode\mathbb{N}\else$\mathbb{N}$\fi}
\newcommand{\Z}{\ifmmode\mathbb{Z}\else$\mathbb{Z}$\fi}
\newcommand{\Q}{\ifmmode\mathbb{Q}\else$\mathbb{Q}$\fi}
\newcommand{\tn}[1]{\textnormal{#1}}
\newcommand{\NN}{\mathcal{N\hspace{-2.5pt}N}}
\let\NNF\NN

\newcommand{\NNspace}{;\ }
\newcommand{\NNinput}{\tn{\#input}}
\newcommand{\NNwidthvec}{\tn{width\hspace{0.5pt}vec}}

\newcommand{\NNdepth}{\tn{depth}}

\let\NNneuron\NNnode
\let\NNlayer\NNdepth
\newcommand{\NNparameter}{\tn{\#parameter}}
\newcommand{\NNwidth}{\tn{width}}
\let\NNmaxwidth\NNwidth
\newcommand{\NNoutput}{\tn{\#output}}
\newcommand{\bmx}{{\bm{x}}}
\newcommand{\bmz}{{\bm{z}}}
\newcommand{\bmW}{{\bm{W}}}

\newcommand{\bmh}{{\bm{h}}}
\newcommand{\bmb}{{\bm{b}}}

\newcommand{\bmbeta}{{\bm{\beta}}}
\newcommand{\bmy}{{\bm{y}}}

\newcommand{\bmA}{\bm{A}}

\newcommand{\bmPhi}{{\bm{\Phi}}}
\newcommand{\bmzero}{{\bm{0}}}
\newcommand{\calO}{{\mathcal{O}}}

\newcommand{\calH}{{\mathcal{H}}}
\newcommand{\calM}{\mathcal{M}}

\newcommand{\calA}{\mathcal{A}}
\newcommand{\calL}{\mathcal{L}}
\newcommand{\calT}{\mathcal{T}}
\newcommand{\tildephi}{{\widetilde{\phi}}}
\newcommand{\tildeg}{{\widetilde{g}}}

\newcommand{\tildef}{{\widetilde{f}}}
\newcommand{\tildepsi}{{\widetilde{\psi}}}
\newcommand{\txx}{{\widetilde{\bm{x}}}}

\newcommand{\scrF}{{\mathscr{F}}}
\newcommand{\scrB}{{\mathscr{B}}}
\newcommand{\Lip}{\tn{H\"old}}
\newcommand{\holder}[2]{{\color{black} B_{#1}(C^{#2}([0,1]^d)) }}

\newcommand{\bin}{\tn{bin}\hspace{1.2pt}}
\newcommand{\vcd}{\tn{VCDim}}

\newcommand{\Small}{{\mathcal{S\hspace{-2pt}L}}}
\newcommand{\mystep}[2]{\par \vspace{0.25cm}\noindent\textbf{\hspace{8pt}Step }$#1\colon$ #2 \vspace{0.18cm} \par }

\newcommand{\mycase}[2]{\par \vspace{0.25cm}\noindent\textbf{\hspace{8pt}Case }$#1\colon$ #2\par \vspace{0.18cm} \par}

\usepackage{tikz}
\newcommand{\myto}[2][1]{\mathop{
		\vcenter{\hbox{\scalebox{1}[#1]{\tikz{\draw[->,line width=0.72pt] (0,0.5) to (0.69*#2,0.5);}}}}
}}

\newcommand{\myMathResize}[2][0.9]{
	\scalebox{#1}[#1]{\(\displaystyle #2\)}
}

\newcommand{\mybig}[1]{\ifthenelse{#1=1}{\big}{\ifthenelse{#1=2}{\Big}{\ifthenelse{#1=3}{\bigg}{\ifthenelse{#1=4}{\Bigg}{}}}}}

\definecolor{color1}{rgb}{0.00,0.00,0.928}
\definecolor{color2}{rgb}{0.000,0.801,1}
\definecolor{color3}{rgb}{0.288,1,0.680}
\definecolor{color4}{rgb}{0.961,0.961,0.961}
\definecolor{zsjcolor}{rgb}{0.05,0.423,0.15}
\definecolor{mycolor}{HTML}{FFD700}
\newcommand{\mylegend}[1]{\,\raisebox{1.512pt}{\framebox{\colorbox{#1}{\raisebox{2.412pt}{\,}\hspace{9pt}}}}\,}
\newcommand{\sj}[1]{{\color{black}\textbf{}#1}}

\long\def\red#1{{\color{red}#1}}

\newenvironment{keywords}{\par \noindent\textbf{Key words}.}{\par}
\usepackage{makecell}
\usepackage{amsopn}

\usepackage{url}
\newcommand*{\email}[1]{\href{mailto:#1}{\nolinkurl{#1}}}
\usepackage[left,mathlines]{lineno}
\usepackage{refcount}
\definecolor{mylinenumbercolor}{HTML}{BEBEBE}

\makeatletter
\newcommand*\patchAmsMathEnvironmentForLineno[1]{%
	\expandafter\let\csname old#1\expandafter\endcsname\csname #1\endcsname
	\expandafter\let\csname oldend#1\expandafter\endcsname\csname end#1\endcsname
	\renewenvironment{#1}%
	{\linenomath\csname old#1\endcsname}%
	{\csname oldend#1\endcsname\endlinenomath}}% 
\newcommand*\patchBothAmsMathEnvironmentsForLineno[1]{%
	\patchAmsMathEnvironmentForLineno{#1}%
	\patchAmsMathEnvironmentForLineno{#1*}}%
\patchBothAmsMathEnvironmentsForLineno{equation}%
\patchBothAmsMathEnvironmentsForLineno{align}%
\patchBothAmsMathEnvironmentsForLineno{flalign}%
\patchBothAmsMathEnvironmentsForLineno{alignat}%
\patchBothAmsMathEnvironmentsForLineno{gather}%
\patchBothAmsMathEnvironmentsForLineno{multline}%
\makeatother
\title{Deep Network Approximation Characterized by Number of Neurons\thanks{Submitted to the editors DATE.}}
\author{Zuowei Shen\thanks{Department of Mathematics,  National University of Singapore
(\email{matzuows@nus.edu.sg}).}
  \and Haizhao Yang\thanks{Department of Mathematics,  Purdue University
  (\email{haizhao@purdue.edu}).}
\and Shijun Zhang\thanks{Department of Mathematics,  National University of Singapore
  (\email{zhangshijun@u.nus.edu}).}}
  \date{}
\let\tilde\widetilde
\let\epsilon\varepsilon
\let\subset\subseteq

\long\def\red#1{{\color{black}#1}}
\begin{document}
\maketitle
%\tableofcontents
%%%%%%%%%%%%%%%%%%%%%%%%%%%%%%%%%%%%%%%%%%%%%%
\begin{abstract}
This paper quantitatively characterizes the approximation power of deep feed-forward neural networks (FNNs) in terms of the number of neurons. It is shown by construction that ReLU FNNs with width $\mathcal{O}\big(\max\{d\lfloor N^{1/d}\rfloor,\, N+1\}\big)$
and depth $\mathcal{O}(L)$  can approximate an arbitrary H\"older continuous function of order $\alpha\in (0,1]$ on $[0,1]^d$ with a nearly tight approximation rate $\mathcal{O}\big(\sqrt{d} N^{-2\alpha/d}L^{-2\alpha/d}\big)$ measured in $L^p$-norm for any $N,L\in \mathbb{N}^+$ and $p\in[1,\infty]$. More generally  for an arbitrary continuous function $f$ on $[0,1]^d$ with a modulus of continuity $\omega_f(\cdot)$, the constructive approximation rate is $\mathcal{O}\big(\sqrt{d}\,\omega_f( N^{-2/d}L^{-2/d})\big)$. We also extend our analysis to $f$ on irregular domains or those localized in an $\varepsilon$-neighborhood of a $d_{\mathcal{M}}$-dimensional smooth manifold $\mathcal{M}\subseteq [0,1]^d$ with $d_{\mathcal{M}}\ll d$. Especially, in the case of an essentially low-dimensional domain, 
we show an approximation rate $\mathcal{O}\big(\omega_f(\tfrac{\varepsilon}{1-\delta}\sqrt{\tfrac{d}{d_\delta}}+\varepsilon)+\sqrt{d}\,\omega_f(\tfrac{\sqrt{d}}{(1-\delta)\sqrt{d_\delta}}N^{-2/d_\delta}L^{-2/d_\delta})\big)$ for ReLU FNNs to approximate $f$ in the $\varepsilon$-neighborhood, where
$d_\delta=\mathcal{O}\big(d_{\mathcal{M}}\tfrac{\ln (d/\delta)}{\delta^2}\big)$ for any 
$\delta\in(0,1)$ as a relative error for a projection to approximate an isometry when projecting $\mathcal{M}$ to a $d_{\delta}$-dimensional domain. 
\end{abstract}

\section{Introduction}
\label{sec:intro}
The approximation theory of neural networks has been an active research topic in the past few decades. Previously, as a special kind of ridge function approximation, shallow neural networks with one hidden layer and various activation functions (e.g., wavelets pursuits \cite{258082,471413}, adaptive splines \cite{devore_1998,PETRUSHEV2003158}, radial basis functions \cite{citeulike:3408165,6797088,radiusbase,HANGELBROEK2010203,Xie2013}, sigmoid functions \cite{HORNIK1991251,MAIOROV199981,KURKOVA1992501,BLUM1991511,Sig1,Sig2,Sig3,Sig4,Sig5}) were widely discussed and admit good approximation properties, e.g., the universal approximation property \cite{Cybenko1989ApproximationBS,HORNIK1989359,HORNIK1991251}, lessening the curse of dimensionality \cite{barron1993,Weinan2019,Weinan2019APE}, and providing attractive approximation rate in nonlinear approximation \cite{devore_1998,258082,471413,PETRUSHEV2003158,radiusbase,HANGELBROEK2010203,Xie2013}. %Particularly for the approximation rate, For functions in Besov spaces with smoothness $s$, \cite{radiusbase,HANGELBROEK2010203} constructed an $\calO(N^{-s/d})$ approximation that is almost optimal \cite{Lin2014} and the smoothness cannot be reduced generally \cite{HANGELBROEK2010203}. 
%For a general class of continuous functions, the approximation rate of one-hidden-layer FNNs can achieve $\calO(N^{-2/d})$ as shown in \cite{Xie2013}. 

The introduction of deep networks with more than one hidden layers has made significant impacts in many fields in computer science and engineering including computer vision \cite{NIPS2012_4824} and natural language processing \cite{6857341}. New scientific computing tools based on deep networks have also emerged and facilitated large-scale and high-dimensional problems that were impractical previously \cite{Han8505,E2017}. The design of deep ReLU FNNs is the key of such a revolution. These breakthroughs have stimulated broad research topics from different points of views to study the power of deep ReLU FNNs, e.g. in terms of combinatorics \cite{NIPS2014_5422}, topology \cite{ 6697897}, Vapnik-Chervonenkis (VC) dimension \cite{Bartlett98almostlinear,Sakurai,pmlr-v65-harvey17a}, fat-shattering dimension \cite{Kearns,Anthony:2009}, information theory \cite{PETERSEN2018296}, classical approximation theory \cite{Cybenko1989ApproximationBS,HORNIK1989359,barron1993,yarotsky18a,2019arXiv190210170S}, optimization \cite{NIPS2016_6112,DBLP:journals/corr/NguyenH17,opt} etc.

Particularly in approximation theory, \textbf{non-quantitative and asymptotic} approximation rates of ReLU FNNs have been proposed for various types of functions. For example, smooth functions \cite{NIPS2017_7203,DBLP:journals/corr/LiangS16,yarotsky2017,DBLP:journals/corr/abs-1807-00297}, piecewise smooth functions \cite{PETERSEN2018296}, band-limited functions \cite{bandlimit}, continuous functions \cite{yarotsky18a}, solutions to partial differential equations \cite{Martin2020}. However, to the best of our knowledge, existing  theories \cite{NIPS2017_7203,bandlimit,yarotsky2017,DBLP:journals/corr/LiangS16,Hadrien,suzuki2018adaptivity,PETERSEN2018296,yarotsky18a,DBLP:journals/corr/abs-1807-00297,Daubechies2019} can only provide implicit formulas in the sense that the approximation error contains an unknown prefactor, or work only for sufficiently large $N$ and $L$ larger than some unknown numbers. For example, \cite{yarotsky18a} estimated an approximation rate $c(d) L^{-2\alpha/d}$ via a narrow and deep ReLU FNN, where $c(d)$ is an unknown number depending on $d$, and $L$ is required to be larger than a sufficiently large unknown number $\mathscr{L}$. For another example, given an approximation error $\epsilon$, \cite{PETERSEN2018296} proved the existence of a ReLU FNN with a constant but still unknown number of layers approximating a $C^\beta$ function within the target error. These works can be divided into two cases: 1) FNNs with varying width and only one hidden layer \cite{radiusbase,HANGELBROEK2010203,Lin2014,Xie2013} (visualized by the region in \mylegend{color1} in Figure \ref{fig:comp}); 2) FNNs with a fixed width of $\calO(d)$ and a varying depth larger than an unknown number $\mathscr{L}$  \cite{NIPS2017_7203,yarotsky18a} (represented by the region in \mylegend{color3} in Figure \ref{fig:comp}). %For the class of functions in $C([0,1]^d)$, the best asymptotic approximation rate obtained in the first case was $\calO(N^{-1/d})$, while the best asymptotic approximation rate in the second case is $\calO(L^{-2/d})$, if we omit the modulus of continuity for simplicity. 

As far as we know, the first \textbf{quantitative and non-asymptotic} approximation rate
of deep ReLU FNNs was obtained in \cite{2019arXiv190210170S}. Specifically, \cite{2019arXiv190210170S} identified an explicit formulas of the approximation rate  
\begin{equation}
\label{eqn:exf}
 \begin{cases}
2 \lambda  N^{-2\alpha}, &\text{when  $L\geq 2$ and $d=1$,}\\
2(2\sqrt{d})^\alpha \lambda  N^{-2\alpha/d}, &\text{when $L\geq 3$  and  $d\ge2$},
\end{cases}
\end{equation}
for ReLU FNNs with an arbitrary width $N\in\mathbb{N}^+$ and a fixed depth $L\in\mathbb{N}^+$ to approximate a H{\"o}lder continuous function $f$ of order $\alpha$ with a H\"older constant $\lambda  $ (visualized in the region shown by \mylegend{color2} in Figure \ref{fig:comp}). The approximation rate $\calO(N^{-2\alpha/d})$ is tight in terms of $N$ and increasing $L$ cannot improve the approximation rate in $N$. The success of deep FNNs in a broad range of applications has motivated a well-known conjecture that the depth $L$ has an important role in improving the approximation power of deep FNNs. In particular, a very important question in practice would be, given an arbitrary $L$ and $N$, what is the explicit formula to characterize the approximation error so as to see whether the network is large enough to meet the accuracy requirement. Due to the highly nonlinear structure of deep FNNs, it is still a challenging open problem to characterize $N$ and $L$ simultaneously in the approximation rate. 

To answer the question just above, we establish the first framework that is able to quantify the approximation power of deep ReLU FNNs essentially with arbitrary width $N$ and depth $L$, achieving a nearly optimal approximation rate, $19\sqrt{d}\,\omega_f( N^{-2/d}L^{-2/d})$, for continuous functions $f\in C([0,1]^d)$. Our result is based on new analysis techniques merely based on the structure of FNNs and a modified bit extraction technique inspired by  \cite{Bartlett98almostlinear}, instead of designing FNNs to approximate traditional approximation basis like polynomials and splines as in the existing literature \cite{pmlr-v70-safran17a,DBLP:journals/corr/LiangS16,DBLP:journals/corr/RolnickT17,yarotsky2017,yarotsky18a,Hadrien,Boris,NIPS2017_7203,PETERSEN2018296,Johannes,suzuki2018adaptivity,MO}. The approximation rate obtained here admits an explicit formula to compute the prefactor when $\omega_f(\cdot)$ is known. For example, in the case of H{\"o}lder continuous functions of order $\alpha$ with a H\"older constant $\lambda  $ (denoted as the class $\holder{\lambda}{\alpha}$), $\omega_f(r)\le \lambda  r^\alpha$ for $r\ge 0$, resulting in the approximation rate $19\sqrt{d}\,\lambda  N^{-2\alpha/d}L^{-2\alpha/d}$ as mentioned previously. As a consequence, existing works for the function class $C([0,1]^d)$ are special cases of our result (see Figure \ref{fig:comp} for a comparison).

\begin{figure}[!htp]
    \centering
    \includegraphics[width=0.66\textwidth]{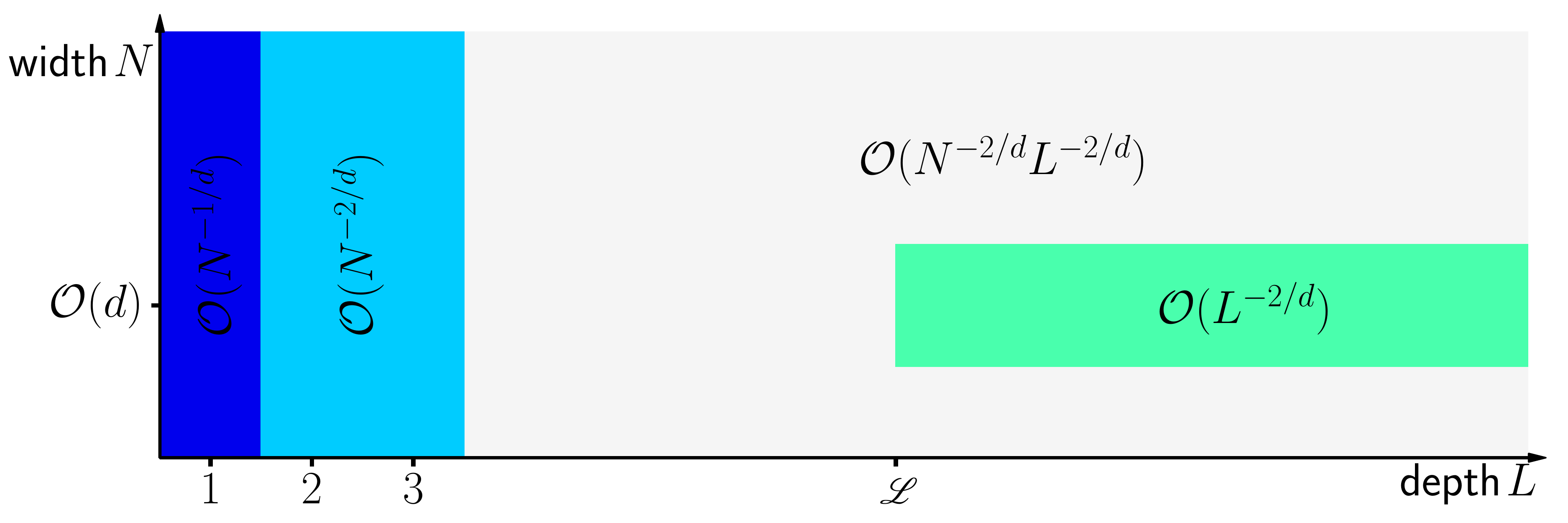}
    \caption{A summary of existing and our new results on the approximation rate of ReLU FNNs for continuous functions. Existing results \cite{radiusbase,HANGELBROEK2010203,Lin2014,Xie2013,NIPS2017_7203,yarotsky18a,2019arXiv190210170S} are applicable in the areas in \mylegend{color1}, \mylegend{color2}, and \mylegend{color3}; our new result is  suitable for almost all areas when $L\geq 2$. %To simplify notations we omit the modulus of continuity. Most existing analysis is asymptotic and can be applied to two cases: 1) FNNs with varying width and only one hidden layer \cite{radiusbase,HANGELBROEK2010203,Lin2014,Xie2013} (visualized by the region in \mylegend{color1}); 2) FNNs with a fixed width of $\calO(d)$ and a varying depth larger than an unknown number $\mathscr{L}$  \cite{NIPS2017_7203,yarotsky18a} (represented by the region in \mylegend{color3}). For the class of functions in $C([0,1]^d)$, the best approximation rate obtained in the first case was $\calO(N^{-1/d})$, while the best approximation rate in the second case is $\calO(L^{-2/d})$. The first quantitative analysis was proposed by the authors in this paper in  \cite{2019arXiv190210170S}  and it is capable of analyzing ReLU FNNs with $\calO(1)$ hidden layers in the region shown by \mylegend{color2}. This paper extends the result in \cite{2019arXiv190210170S} to the whole domain with arbitrary depth and width  visualized in white, and provides an optimal approximation rate, essentially $\calO(N^{-2/d}L^{-2/d})$, of ReLU FNNs for $C([0,1]^d)$. 
    }
\label{fig:comp}
\end{figure}

Our key contributions can be summarized as follows.
\begin{enumerate}
\item Upper bound: We provide a quantitative and non-asymptotic approximation rate $19\sqrt{d}\,\omega_f( N^{-2/d}L^{-2/d})$ in terms of width $\calO(N)$ and depth $\calO(L)$ for functions in $C([0,1]^d)$ in Theorem \ref{thm:main}. 

\item Lower bound: Through the nearly tight VC-dimension bounds of ReLU FNNs \cite{pmlr-v65-harvey17a}, we show that the approximation rate $19\sqrt{d}\,\omega_f( N^{-2\alpha/d}L^{-2\alpha/d})$ in terms of $N$ and $L$ is nearly optimal for $\holder{\lambda}{\alpha}$ in Theorem \ref{thm:lowInfty}.

\item The approximation rate in terms of the width and depth in this paper is more generic and useful than the one characterized by the number of nonzero parameters denoted as $W$ in the literature. First, the characterization in terms of width and depth implies the one in terms of $W$, while it is not true the other way around. Second, our theory can provide practical guidance for choosing network sizes in realistic applications while theories in terms of $W$ cannot tell how large a network should be to guarantee a target accuracy, since there are too many networks of different sizes sharing the same number of parameters but with different accuracies.

\item Finally, three aspects of neural networks in practice are discussed: 1) neural network approximation in a high-dimensional irregular domain; 2) neural network approximation in the case of a low-dimensional data structure; 3) the optimal ReLU FNN in parallel computation.
\end{enumerate}

 Our main result, Theorem \ref{thm:main} below, shows that ReLU FNNs with width $\calO(N)$  and depth $\calO(L)$ can approximate $f$ with an approximation rate $19\sqrt{d}\,\omega_f( N^{-2/d}L^{-2/d})$, where $\omega_f(\cdot)$ is the modulus of continuity of $f$ defined via
\begin{equation*}
\omega_f(r)\coloneq \sup\big\{|f(\bm{x})-f(\bm{y})|:\bm{x},\bm{y}\in [0,1]^d,\ \|\bm{x}-\bm{y}\|_2\le r\big\},\quad \tn{for any $r\ge 0$}.
\end{equation*}

\begin{theorem}
	\label{thm:main}
	Given $f\in C([0,1]^d)$, for any $L\in \N^+$, $N\in \N^+$, and $p\in [1,\infty]$, 
	there exists a function $\phi$ implemented by a ReLU FNN with width $C_1\max\big\{d\lfloor N^{1/d}\rfloor,\, N+1\big\}$
	and depth $12L+C_2$
	such that 
	\begin{equation*}
	\|f-\phi\|_{L^p([0,1]^d)}\le 19\sqrt{d}\,\omega_f(N^{-2/d}L^{-2/d}),
	\end{equation*}
	where $C_1=12$ and $C_2=14$ if $p\in [1,\infty)$;  $C_1=3^{d+3}$ and $C_2=14+2d$ if $p=\infty$.
	%        where $C_1=\left\{\begin{aligned}
	%        &14,&p&\in [1,\infty), \\
	%        &3^{d+3},&p&=\infty,
	%        \end{aligned}\right.$ and $C_2=\left\{\begin{aligned}
	%        &13,&p&\in [1,\infty), \\
	%        &13+2d,&p&=\infty.
	%        \end{aligned}\right.$
\end{theorem}

When Theorem \ref{thm:main} is applied to $f\in \holder{\lambda}{\alpha}$, the approximation rate is $19\sqrt{d}\,\lambda  N^{-2\alpha/d}L^{-2\alpha/d}$, because $\omega_f(r)\le \lambda  r^\alpha$ for any $r\ge 0$. An immediate question following the constructive approximation is how much we can improve the approximation rate.
In fact, the approximation rate of $f\in \holder{\lambda}{\alpha}$ is asymptotically tight based on VC-dimension as we shall see later. 

In most real applications of neural networks, though the target function $f$ is defined in a high-dimensional domain, e.g., $[0,1]^d$, where $d$ could be tens of thousands or even millions, only the approximation error of $f$ in a neighborhood of a $d_{\mathcal{M}}$-dimensional manifold $\mathcal{M}$ with $d_{\mathcal{M}}\ll d$ is concerned. Hence, we extend Theorem \ref{thm:main} to the case when the domain of $f$ is localized in an $\epsilon$-neighborhood of a compact $d_\calM$-dimensional Riemannian submanifold $\calM\subseteq [0,1]^d$ having condition number $1/\tau$, volume $V$, and geodesic covering regularity $\mathcal{R}$. The $\epsilon$-neighborhood is defined as 
\begin{equation}\label{eqn:Me}
\calM_\epsilon  \coloneqq \big\{\bmx\in[0,1]^d:\inf \{\|\bmx-\bmy\|_2:\bmy\in \calM\}\le \epsilon\big\},\quad \tn{for $\epsilon\in(0,1)$.}
\end{equation}

 Let $d_\delta=\calO\left(\tfrac{d_\calM \ln (dV\mathcal{R}\tau^{-1}\delta^{-1})}{\delta^2}    \right)=\calO\big(d_\calM\tfrac{\ln (d/\delta)}{\delta^2}\big)$ be an integer for any $\delta\in (0,1)$ such that $d_\calM\le d_\delta\le d$. 
 We show an approximation rate \[2\omega_f\big(\tfrac{2\epsilon}{1-\delta}\sqrt{\tfrac{d}{d_\delta}}+2\epsilon\big)+19\sqrt{d}\,\omega_f(\tfrac{2\sqrt{d}}{(1-\delta)\sqrt{d_\delta}}N^{-2/d_\delta}L^{-2/d_\delta})\] for ReLU FNNs to pointwisely approximate $f$ on $\calM_\varepsilon$. 
 The key ideas of the proof is the application of Theorem $3.1$ in \cite{Baraniuk2009}, which provides a nearly isometric projection $\bmA\in\mathbb{R}^{d_\delta\times d}$ that maps points in $\mathcal{M}\subseteq[0,1]^d$ to a $d_\delta$-dimensional domain with
  \[{(1-\delta)}|\bmx_1-\bmx_2|\le |\bmA\bmx_1-\bmA\bmx_2|\le {(1+\delta)}|\bmx_1-\bmx_2|,\quad \text{for any }\bmx_1,\bmx_2\in \calM,\] and the application of Theorem \ref{thm:main} in this paper, which constructs the desired ReLU FNN with a size depending on $d_\delta$ instead of $d$ to lessen the curse of dimensionality. When $\delta$ is closer to $1$, $d_\delta$ is closer to $d_{\mathcal{M}}$ but the isometric property of the projection is weakened; when $\delta$ is closer to $0$, the isometric property becomes better but $d_\delta$ could be larger than $d$, in which case we can simply enforce $d_\delta=d$ and choose the identity map as the projection. Hence, $\delta\in(0,1)$ is a parameter to make a balance between isometry and dimension reduction.
 
\begin{theorem}
    \label{thm:upDimReduction}
    Let $f$ be a continuous function on $[0,1]^d$ and $\calM\subseteq [0,1]^d$ be a compact $d_\calM$-dimensional Riemannian submanifold. For any  $N\in \N^+$,  $L\in\N^+$, $\epsilon\in(0,1)$, and $\delta\in(0,1)$, there exists a function $\phi$ implemented by a ReLU FNN  with width $3^{d_\delta+3}\max\big\{d_\delta\lfloor N^{1/d_\delta}\rfloor,\, N+1\big\}$ and depth $12L+14+2d_\delta$ such that
    \begin{equation}
    \label{eqn:errLp}
    \myMathResize[0.9325]{|f(\bmx)-\phi(\bmx)|
    \le 2\omega_f\big(\tfrac{2\epsilon}{1-\delta}\sqrt{\tfrac{d}{d_\delta}}+2\epsilon\big)+19\sqrt{d}\,\omega_f(\tfrac{2\sqrt{d}}{(1-\delta)\sqrt{d_\delta}}N^{-2/d_\delta}L^{-2/d_\delta})},
    \end{equation}   
    for any $\bmx\in \calM_\varepsilon$, where $\calM_\varepsilon$ is defined in Equation \eqref{eqn:Me} 
\end{theorem}

The approximation rate of deep neural networks for functions defined precisely on low-dimensional smooth manifolds has been studied in \cite{SHAHAM2018537} for $C^2$ functions and in \cite{DBLP:journals/corr/abs-1811-09054,10.3389/fams.2018.00014} for Lipschitz continuous functions. Considering that it might be more reasonable to assume data located in a small neighborhood of low-dimensional smooth manifold in real applications, we introduce the $\epsilon$-neighborhood of the manifold $\mathcal{M}$ in Theorem \ref{thm:upDimReduction}. In general, existing results are again asymptotic and they cannot be applied to estimate the approximation accuracy of a ReLU FNN with arbitrarily given width $N$ and depth $L$, since there is no explicit formula without unknown constants to specify the exact error bound. For example, \cite{DBLP:journals/corr/abs-1811-09054} provides an approximation rate $c_1 \left(NL\right)^{-c_2/d_\delta}$ with unknown constants (e.g., $c_1$ and $c_2$) and requires $NL$ greater than an unknown large number. The demand of an explicit error estimation motivates Theorem \ref{thm:upDimReduction} in this paper. When data are concentrating around $\mathcal{M}$, $\epsilon$ is very small and the dominant term of the approximation error in \eqref{eqn:errLp} is $19\sqrt{d}\,\omega_f(\tfrac{2\sqrt{d}}{(1-\delta)\sqrt{d_\delta}}N^{-2/d_\delta}L^{-2/d_\delta})$ implying that the approximation via deep ReLU FNNs can lessen the curse of dimensionality.

The analysis above provides a general guide for selecting the width and depth of ReLU FNNs to approximate continuous functions, especially when the computation is conducted with parallel computing, which is usually the case in real applications \cite{10.1007/978-3-642-15825-4_10,Ciresan:2011:FHP:2283516.2283603}. As we shall see later, when the approximation accuracy and the parallel computing efficiency are considered together, very deep FNNs become less attractive than those with $\calO(1)$ depth. 

The approximation theories in this paper assume that the target function $f$ is fully accessible, making it possible to estimate the approximation error and identify an asymptotically optimal ReLU FNN with a given budget of neurons to minimize the approximation error. In real applications, usually only a limited number of possibly noisy observations of $f$ is available, resulting in a regression problem in statistics. In the latter case, the problem is usually formulated in a stochastic setting with randomly generated noisy observations and the regression error contains mainly two components: bias and variance. The bias is the difference of the expectation of an estimated function and its ground truth $f$. The approximation theories in this paper play an important role in characterizing the power of neural networks when they are applied to solve regression problems by providing a lower bound of the regression bias.

The rest of this paper is organized as follows. We first prove Theorem \ref{thm:main} and show its optimality in Section \ref{sec:contFunc} when assuming Theorem \ref{thm:mainGap} is true. Next, Theorem \ref{thm:mainGap} is proved in Section \ref{sec:mainTriflingRegion}. In Section \ref{sec:NNP}, three aspects of neural networks in practice will be discussed: 1) neural network approximation in a high-dimensional irregular domain; 2) neural network approximation in the case of a low-dimensional data structure; 3) the optimal ReLU FNN in parallel computation. Finally, Section \ref{sec:conclusion} concludes this paper with a short discussion.

\section{Approximation of continuous functions}
\label{sec:contFunc}
In this section, we prove Theorem \ref{thm:main} and discuss its optimality when assume Theorem \ref{thm:mainGap} is true. Notations throughout the proof will be summarized in Section \ref{sec:notation}. 

\subsection{Notations}
\label{sec:notation}

    Let us summarize all basic notations used in this paper as follows.
\begin{itemize}

     \item Matrices are denoted by bold uppercase letters. For instance,  $\bm{A}\in\mathbb{R}^{m\times n}$ is a real matrix of size $m\times n$, and $\bm{A}^T$ denotes the transpose of $\bm{A}$.  %Correspondingly, $\bm{A}(i,j)$ is the $(i,j)$-th entry of $\bm{A}$; $\bm{A}(:,j)$ is the $j$-th column of $\bm{A}$; $\bm{A}(i,:)$ is the $i$-th row of $\bm{A}$. 
     Vectors are denoted as bold lowercase letters. For example, $\bm{v}=\left[\def\arraystretch{0.748}\begin{array}{c}
          v_1  \\
          \vdots \\
          v_d
     \end{array}\right]=[v_1,\cdots,v_d]^T\in \R^d$ is a column vector %consisting of numbers $\{v_i\}_i$
     with $\bm{v}(i)=v_i$ being the $i$-th element. Besides, ``['' and ``]''  are used to  partition matrices (vectors) into blocks, e.g., $\bmA=\left[\begin{smallmatrix}\bmA_{11}&\bmA_{12}\\ \bmA_{21}&\bmA_{22}\end{smallmatrix}\right]$.
     
      \item For any $p\in [1,\infty)$, the $p$-norm of a vector $\bmx=[x_1,x_2,\cdots,x_d]^T\in\R^d$ is defined by 
    \begin{equation*}
        \|\bmx\|_p\coloneqq \big(|x_1|^p+|x_2|^p+\cdots+|x_d|^p\big)^{1/p}.
    \end{equation*}
     
     \item  Let $\mu(\cdot)$ be the Lebesgue measure. 
     
         \item Let $1_{S}$ be the characteristic function on a set $S$, i.e., $1_{S}$ is equal to $1$ on $S$ and $0$ outside of $S$.

     \item The set difference of two sets $A$ and $B$ is denoted by $A\backslash B:=\{x:x\in A,\ x\notin B\}$. 
     
     %\item Assume $\mathcal{N}$ is a subset of $\R$, then $\max \mathcal{N}$, $\min\mathcal{N}$, $\sup \mathcal{N}$, and $\inf \mathcal{N}$ mean the maximum, the minimum, the supremum, and the infimum of $\mathcal{N}$, respectively.
     
     \item For any $\xi\in \R$, let $\lfloor \xi\rfloor:=\max \{i: i\le \xi,\ i\in \Z\}$ and $\lceil \xi\rceil:=\min \{i: i\ge \xi,\ i\in \Z\}$.
     \item Assume $\bm{n}\in \N^d$, then $f(\bm{n})=\mathcal{O}(g(\bm{n}))$ means that there exists positive $C$ independent of $\bm{n}$, $f$, and $g$ such that $ f(\bm{n})\le Cg(\bm{n})$ when all entries of $\bm{n}$ go to $+\infty$.
     
     \item Let $\sigma:\R\to \R$ denote the rectified linear unit (ReLU), i.e. $\sigma(x)=\max\{0,x\}$. With the abuse of notations, we define $\sigma:\R^d\to \R^d$ as $\sigma(\bmx)=\left[\begin{array}{c}
          \max\{0,x_1\}  \\
          \vdots \\
          \max\{0,x_d\}
     \end{array}\right]$ for any $\bmx=[x_1,\cdots,x_d]^T\in \R^d$.
     
     \item 
     Given $K\in N^+$ and $\delta\in (0, \tfrac{1}{K})$, define a trifling region  $\Omega([0,1]^d,K,\delta)$ of $[0,1]^d$ as 
     \begin{equation}
     \label{eq:triflingRegionDef}
     \Omega([0,1]^d,K,\delta)\coloneqq\bigcup_{i=1}^{d} \Big\{\bmx=[x_1,x_2,\cdots,x_d]^T\red{\in [0,1]^d}: x_i\in \cup_{k=1}^{K-1}(\tfrac{k}{K}-\delta,\tfrac{k}{K})\Big\}.
     \end{equation}
     In particular, $\Omega([0,1]^d,K,\delta)=\emptyset$ if $K=1$. See Figure \ref{fig:region} for two examples of trifling regions.
     
     \begin{figure}[!htp]        
     	\centering
     \begin{minipage}{0.85\textwidth}
     	\centering
     	\begin{subfigure}[b]{0.47\textwidth}
     		\centering            \includegraphics[width=0.999\textwidth]{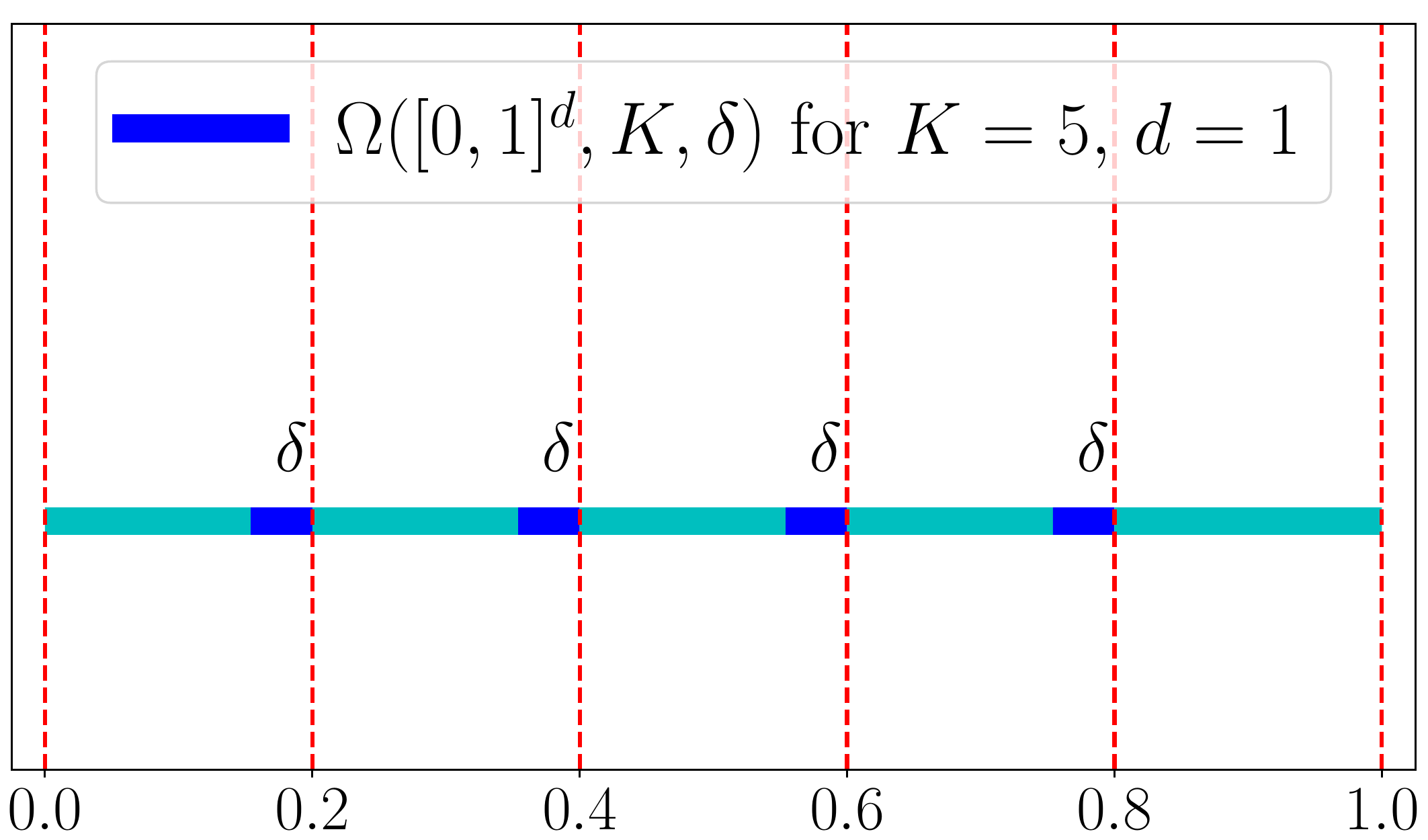}
     		\subcaption{}
     	\end{subfigure}
     \begin{minipage}{0.02\textwidth}
     	\,
     \end{minipage}
     	\begin{subfigure}[b]{0.30641\textwidth}
     		\centering            \includegraphics[width=0.99\textwidth]{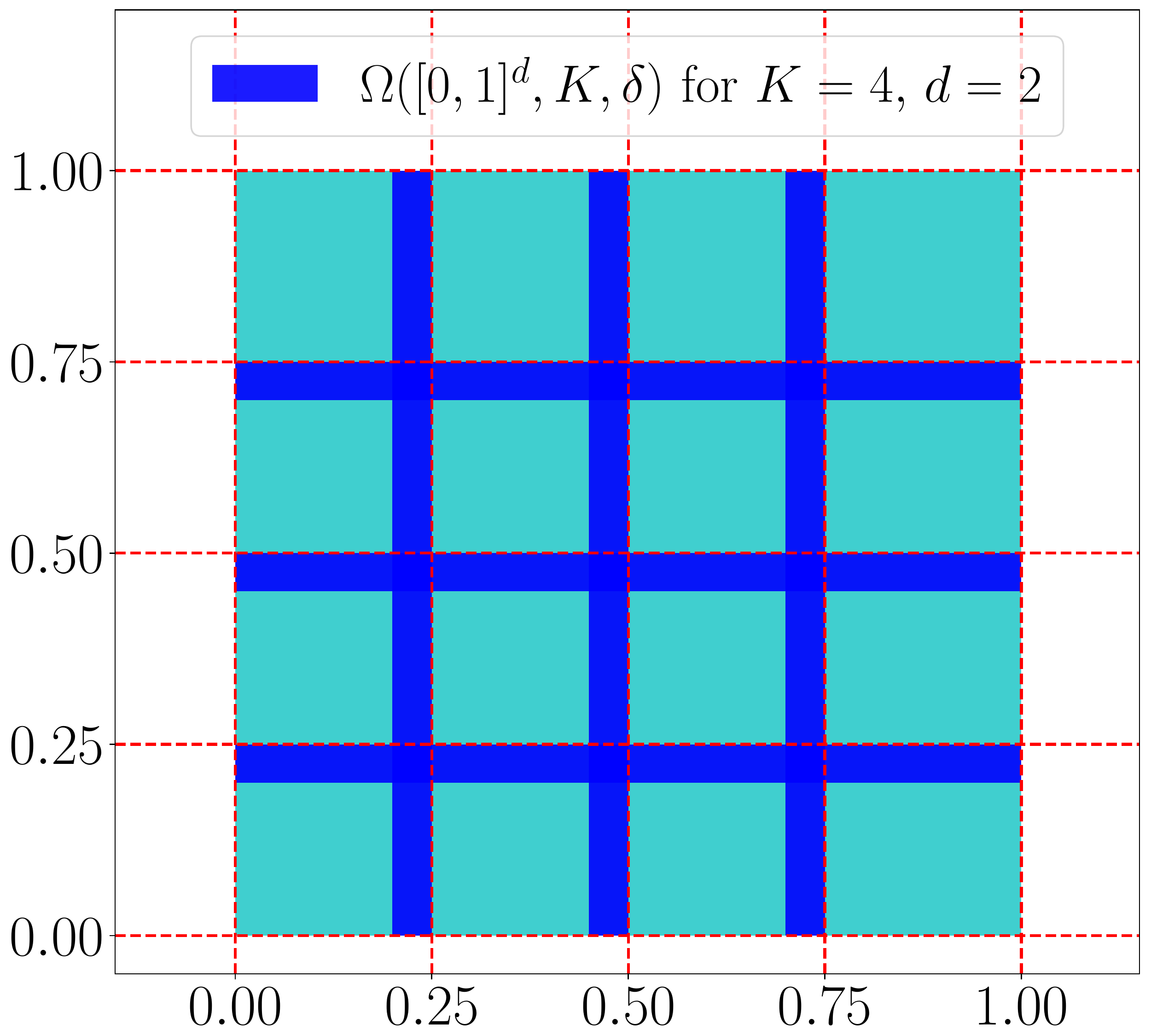}
     		\subcaption{}
     	\end{subfigure}
 	\end{minipage}
     	\caption{Two examples of trifling regions. (a)  $K=5,d=1$. (b) $K=4,d=2$.}
     	\label{fig:region}
     \end{figure}
 
     \item {Let $C^\alpha([0,1]^d)$ be the set containing all H{\"o}lder continuous functions on $[0,1]^d$ of order $\alpha\in (0,1]$. In particular,  the $\lambda$-ball in $C^{\alpha}([0, 1]^d)$ is denoted by $\holder{\lambda}{\alpha}$ for any $\lambda>0$.}
     
     \item We will use $\NNF$ to denote a function implemented by a ReLU FNN for short and use Python-type notations to specify a class of functions implemented by ReLU FNNs with several conditions, e.g., $\NNF(\tn{c}_1;\ \tn{c}_2;\ \cdots;\ \tn{c}_m)$ is a set of functions implemented by  ReLU FNNs satisfying $m$ conditions given by $\{\tn{c}_i\}_{1\leq i\leq m}$, each of which may specify the number of inputs ($\NNinput$), the number of outputs ($\NNoutput$), the total number of neurons in all hidden layers ($\NNneuron$), the number of hidden layers ($\NNdepth$), the total  number of parameters ($\NNparameter$), and the width in each hidden layer ($\NNwidthvec$), the maximum width of all hidden layers ($\NNwidth$), etc. For example, if $\phi\in \NNF(\NNinput=2\NNspace \NNwidthvec=[100,100]\NNspace\NNoutput=1)$,  then $\phi$ is a functions satisfies
     \begin{itemize}
         \item $\phi$ maps from $\R^2$ to $\R$.
         \item $\phi$ can be implemented by a ReLU FNN with two hidden layers and the number of nodes in each hidden layer is $100$.
     \end{itemize}
 
     \item $[n]^L$ is short for $[n,n,\cdots,n]\in \N^{L}$. 
     For example, \[\NNF(\NNinput=d\NNspace\NNwidthvec=[100,100])=\NNF(\NNinput=d\NNspace\NNwidthvec=[100]^2).\]

     \item For a function $\phi\in \NNF(\NNinput=d\NNspace\NNwidthvec=[N_1,N_2,\cdots,N_L]\NNspace\NNoutput=1)$, if we set $N_0=d$ and $N_{L+1}=1$, then the architecture of the network implementing $\phi$ can be briefly described as follows:
    \begin{equation*}
    \begin{aligned}
    \bm{x}=\widetilde{\bm{h}}_0 \myto{2.2}^{\bm{W}_0,\ \bm{b}_0} \bm{h}_1\mathop{\longrightarrow}^{\sigma} \tilde{\bm{h}}_1 \ \cdots\ \myto{2.7}^{\bm{W}_{L-1},\ \bm{b}_{L-1}} \bm{h}_L\mathop{\longrightarrow}^{\sigma} \tilde{\bm{h}}_L \mathop{\myto{2.2}}^{\bm{W}_{L},\ \bm{b}_{L}} \bm{h}_{L+1}=\phi(\bm{x}),
    \end{aligned}
    \end{equation*}
    where $\bm{W}_i\in \R^{N_{i+1}\times N_{i}}$ and $\bm{b}_i\in \R^{N_{i+1}}$ are the weight matrix and the bias vector in the $i$-th \red{(affine)} linear transform $\calL_i$ in $\phi$, respectively, i.e., 
    \[\bm{h}_{i+1} =\bm{W}_i\cdot \tilde{\bm{h}}_{i} + \bm{b}_i\eqqcolon \calL_i(\tilde{\bm{h}}_{i}),\quad \tn{for $i=0,1,\cdots,L$,}\]  
    and
    \[
       \tilde{\bm{h}}_i=\sigma(\bm{h}_i),\quad \tn{for $i=1$, $\dots$, $L$.}
    \]
    In particular, $\phi$ can be represented in a form of function compositions as follows
    \begin{equation*}
        \phi =\calL_L\circ\sigma\circ\calL_{L-1}\circ \sigma\circ \ \cdots \  \circ \sigma\circ\calL_1\circ\sigma\circ\calL_0,
    \end{equation*}
    which has been illustrated in Figure \ref{fig:ReLUeg}.
    \begin{figure}[!htp]        
     	\centering
     		\centering            \includegraphics[width=0.7\textwidth]{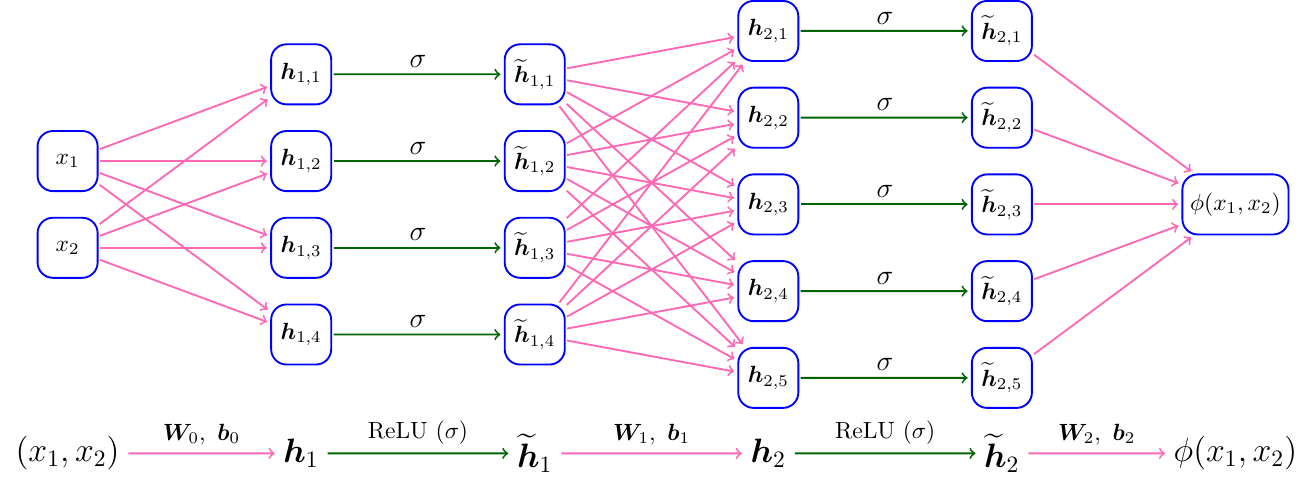}
     	\caption{An example of a ReLU network with width $5$ and depth $2$. }
     	\label{fig:ReLUeg}
     \end{figure}

     \item The expression ``an FNN with width $N$ and depth $L$'' means
     \begin{itemize}
         \item The maximum width of this FNN for all \textbf{hidden} layers is no more than $N$.
         \item The number of \textbf{hidden} layers of this FNN is no more than $L$.
     \end{itemize}
     
    \item For $\theta\in[0,1)$, suppose its binary representation is $\theta=\sum_{\ell=1}^{\infty}\theta_\ell2^{-\ell}$ with $\theta_\ell\in \{0,1\}$, we introduce a special notation $\bin 0.\theta_1\theta_2\cdots \theta_L$ to denote the $L$-term binary representation of $\theta$, i.e., $\bin 0.\theta_1\theta_2\cdots \theta_L\coloneqq\sum_{\ell=1}^{L}\theta_\ell2^{-\ell}$.  
\end{itemize}

\subsection{Proof of Theorem \ref{thm:main}}
\label{sec:main}

We essentially construct piecewise constant functions to approximate continuous functions in the proof. However, it is impossible to construct a piecewise constant function via ReLU FNNs due to the continuity of ReLU FNNs. Thus, we introduce the trifling region  $\Omega([0,1]^d,K,\delta)$, defined in Equation \eqref{eq:triflingRegionDef}, and use ReLU FNNs to implement piecewise constant functions outside of the trifling region.
To prove Theorem \ref{thm:main}, we first establish a theorem showing how to construct ReLU FNNs to pointwisely approximate continuous functions except for the trifling region. 
%\begin{theorem}
%	\label{thm:mainGap}
%	Let $f$ be a given function in $C([0,1]^d)$. For arbitrary $L\in \N^+$ and $N\in \N^+$,
%	there exists a ReLU FNN $\phi$ with width $\max\big\{8d\lfloor N^{1/d}\rfloor+4d,\, 12N+14\big\}$
%	and depth $9L+11$
%	such that $ \|\phi\|_{L^\infty(\R^d)}\le |f(0)|+ \omega_f(\sqrt{d})$ and 
%	\begin{equation*}
%	|f(\bmx)-\phi(\bmx)|\le 2\omega_f(8\sqrt{d}N^{-2/d}L^{-2/d}),\quad \tn{for any $\bmx\in [0,1]^d\backslash\Omega([0,1]^d,K,\delta)$},
%	\end{equation*}
%	where $K=\lfloor N^{1/d}\rfloor^2\lfloor L^{2/d}\rfloor$ and $\delta$ is an arbitrary number in $(0,\tfrac{1}{3K}]$.
%\end{theorem}
\begin{theorem}
	\label{thm:mainGap}
	Given $f\in C([0,1]^d)$, for any $L\in \N^+$ and $N\in \N^+$,
	there exists a function $\phi$ implemented by a  ReLU FNN with width $\max\big\{4d\lfloor N^{1/d}\rfloor+3d,\, 12N+8\big\}$
	and depth $12L+14$
	such that $ \|\phi\|_{L^\infty(\R^d)}\le |f(\bmzero)|+ \omega_f(\sqrt{d})$ and 
	\begin{equation*}
	|f(\bmx)-\phi(\bmx)|\le 18\sqrt{d}\,\omega_f(N^{-2/d}L^{-2/d}),\quad \tn{for any $\bmx\in [0,1]^d\backslash\Omega([0,1]^d,K,\delta)$},
	\end{equation*}
	where $K=\lfloor N^{1/d}\rfloor^2\lfloor L^{2/d}\rfloor$ and $\delta$ is an arbitrary number in $(0,\tfrac{1}{3K}]$.
\end{theorem}

With Theorem \ref{thm:mainGap} that will be proved in Section \ref{sec:mainTriflingRegion}, we can easily prove Theorem \ref{thm:main} for the case $p\in[1,\infty)$. In the early version of this paper, which focuses on continuous functions as target functions, we only considered the case $p\in[1,\infty)$ since it was challenging to control the approximation error in the trifling region. Later in \cite{2020arXiv200103040L} when we considered smooth functions as target functions, we invented a technique that can handle the error in the trifling region as in the lemma below. Therefore, we are now able to control the approximation error for $p=\infty$. The results 
in this paper are for continuous functions, to which the results in \cite{2020arXiv200103040L} are not applicable; the results in \cite{2020arXiv200103040L} characterize how the smoothness of target functions helps to enhance the approximation capacity of ReLU FNNs, which is not addressed in this paper. It is interesting to point out that the approximation rate $\calO(N^{-2/d}L^{-2/d})$ for continuous functions in this paper is even better than the rate $\calO((\tfrac{N}{\ln N})^{-2/d}(\tfrac{L}{\ln L})^{-2/d})$ for functions in $C^1([0,1]^d)$ in \cite{2020arXiv200103040L}.

\begin{lemma}[Theorem $2.1$ of \cite{2020arXiv200103040L}]
	\label{thm:Gap}
	Given $\varepsilon>0$, $N,L,K\in \N^+$, and $\delta\in (0, \tfrac{1}{3K}]$,
	assume $f\in C([0,1]^d)$ and $\tildephi$ can be implemented by a ReLU FNN with width $N$ and depth $L$. If 
	\begin{equation*}
	|f(\bmx)-\tildephi(\bmx)|\le \varepsilon,\quad \tn{for any $\bmx\in [0,1]^d\backslash \Omega([0,1]^d,K,\delta)$,}
	\end{equation*}
	then there exists a function $\phi$  implemented by a new ReLU FNN with width $3^d(N+4)$ and depth $L+2d$ such that 
	\begin{equation*}
	|f(\bmx)-\phi(\bmx)|\le \varepsilon+d\cdot\omega_f(\delta),\quad \tn{for any $\bmx\in [0,1]^d$.}
	\end{equation*}
\end{lemma}
Now we are ready to prove Theorem \ref{thm:main} by assuming Theorem \ref{thm:mainGap} is true, which  will be proved later in Section \ref{sec:proofMainGap}.
\begin{proof}[Proof of Theorem \ref{thm:main}]	
	Let us first consider the case $p\in [1,\infty)$. 
	We may assume $f$ is not a constant function since it is  a trivial case. Then $\omega_f(r)>0$ for any $r>0$. Set $K=\lfloor N^{1/d}\rfloor^2\lfloor L^{2/d}\rfloor$ and choose a small $\delta\in (0,\tfrac{1}{3K}]$ such that
	\begin{equation*}
	\begin{split}
	Kd\delta\big( \red{2}|f(\bmzero)|+\red{2}\omega_f(\sqrt{d})\big)^p
	&=\lfloor N^{1/d}\rfloor^2\lfloor L^{2/d}\rfloor d\delta \big( 2|f(\bmzero)|+2\omega_f(\sqrt{d})\big)^p\\
	&\le \big(\omega_f(N^{-2/d}L^{-2/d})\big)^p. 
	\end{split}
	\end{equation*}
	By Theorem \ref{thm:mainGap}, there exists a function $\phi$ implemented by a ReLU FNN with width \[\max\big\{4d\lfloor N^{1/d}\rfloor+3d,\, 12N+8\big\}\le 12\max \big\{d\lfloor N^{1/d}\rfloor,\, N+1\big\}\]
	and depth $12L+14$ such that \red{$\|\phi\|_{L^\infty(\R^d)}\le |f(\bmzero)|+\omega_f(\sqrt{d})$} and 
	\begin{equation*}
	|f(\bmx)-\phi(\bmx)|\le 18\sqrt{d}\,\omega_f(N^{-2/d}L^{-2/d}),\quad \tn{for any $\bmx\in [0,1]^d\backslash\Omega([0,1]^d,K,\delta)$},
	\end{equation*}
	It follows from $\mu(\Omega([0,1]^d,K,\delta))\le Kd\delta$ and \red{$\|f\|_{L^\infty([0,1]^d)}\le |f(\bmzero)|+\omega_f(\sqrt{d})$} that
	\begin{equation*}
	\begin{split}
	\|f-\phi\|_{L^p([0,1]^d)}^p
	&=\int_{\Omega([0,1]^d,K,\delta)}|f(\bmx)-\phi(\bmx)|^p\tn{d}\bmx+\int_{[0,1]^d\backslash\Omega([0,1]^d,K,\delta)}|f(\bmx)-\phi(\bmx)|^p\tn{d}\bmx\\
	&\le Kd\delta\big( \red{2}|f(\bmzero)|+\red{2}\omega_f(\sqrt{d})\big)^p+ \big(18\sqrt{d}\,\omega_f(N^{-2/d}L^{-2/d})\big)^p\\
	&\le \big(\omega_f(N^{-2/d}L^{-2/d})\big)^p+ \big(18\sqrt{d}\,\omega_f(N^{-2/d}L^{-2/d})\big)^p\\
	&\le \big(19\sqrt{d}\,\omega_f(N^{-2/d}L^{-2/d})\big)^p.
	\end{split}
	\end{equation*}
	Hence, $\|f-\phi\|_{L^p([0,1]^d)}\le 19\sqrt{d}\,\omega_f(N^{-2/d}L^{-2/d})$.
	
	Next, let us discuss the case $p=\infty$. Set $K=\lfloor N^{1/d}\rfloor^2\lfloor L^{2/d}\rfloor$ and choose a small $\delta\in(0,\tfrac{1}{3K}]$ such that 
	\begin{equation*}
	d\cdot \omega_f(\delta)\le \omega_f(N^{-2/d}L^{-2/d}).
	\end{equation*}	
	By Theorem \ref{thm:mainGap}, there exists a function $\tildephi$ implemented   by a  ReLU FNN with width $\max\big\{4d\lfloor N^{1/d}\rfloor+3d,\, 12N+8\big\}$
	and depth $12L+14$ such that
	\begin{equation*}
	|f(\bmx)-\tildephi(\bmx)|\le 18\sqrt{d}\,\omega_f(N^{-2/d}L^{-2/d})\coloneqq \varepsilon,\quad \tn{for $\bmx\in [0,1]^d\backslash\Omega([0,1]^d,K,\delta)$},
	\end{equation*}
	By Lemma \ref{thm:Gap}, there exists a  function $\phi$ implemented by a ReLU FNN  with width 
	\begin{equation*}
	3^d\Big(\max\big\{4d\lfloor N^{1/d}\rfloor+3d,\, 12N+8\big\}+4\Big)\le 3^{d+3}\max\big\{d\lfloor N^{1/d}\rfloor,\, N+1\big\}
	\end{equation*}
	and depth $12L+14+2d$ such that 
	\begin{equation*}
	|f(\bmx)-\phi(\bmx)|\le \varepsilon+d\cdot \omega_f(\delta)\le  19\sqrt{d}\,\omega_f(N^{-2/d}L^{-2/d}),\quad \tn{for any $\bmx\in [0,1]^d$}.
	\end{equation*}
	So we finish the proof.
\end{proof}

%%%%%%%%%%%%%%%%%%%%%%%%%%%%%%%%%%%%%%%
%%%%%%%%%%%%%%%%%%%%%%%%%%%%%%%%%%%%%%%%%%%
\subsection{Optimality of Theorem \ref{thm:main}}
This section will show that the approximation rate in Theorem \ref{thm:main} is nearly tight and there is no room to improve for the function class $\holder{\lambda}{\alpha}$. 
Theorem \ref{thm:lowInfty} below shows that the approximation rate  $\calO(\omega_f(N^{-(2/d+\rho)}L^{-(2/d+\rho)}))$ for any $\rho>0$ is unachievable, implying the approximation rate in Theorem \ref{thm:main} is nearly tight for the function class $\holder{\lambda}{\alpha}$.

\begin{theorem}
    \label{thm:lowInfty}
    Given any $ \rho>0$ and $ C>0$, there exists $f\in \holder{\lambda}{\alpha}$ such that,  for any $J_0>0$, there exist $N,L\in \N$ with $NL\ge J_0$ satisfying 
    \begin{equation*}
    \inf_{\phi\in \NN(\NNinput=d\NNspace\NNmaxwidth\le N\NNspace\NNlayer\le L)} \|\phi-f\|_{L^\infty([0,1]^d)}
    \ge C \lambda  N^{-(2\alpha/d+\rho)}L^{-(2\alpha/d+\rho)}.
    \end{equation*}
\end{theorem}

\vspace{6pt}
In fact, we can show a stronger result than Theorem \ref{thm:lowInfty}. Under the same conditions as in Theorem \ref{thm:lowInfty}, for any    
$\calH\in [0,1]^d$ with $\mu(\calH)\le 2^{-(d+K^d+1)}K^{-d}$, where $K=\lfloor (NL)^{2/d+\rho/(2\alpha)}\rfloor$, it can be proved that
\begin{equation}
 \label{eq:negativeResultH}  
  \inf_{\phi\in \NN(\NNinput=d\NNspace\NNmaxwidth\le N\NNspace\NNlayer\le L)} \|\phi-f\|_{L^\infty([0,1]^d\backslash\calH)}
    \ge C \lambda  N^{-(2\alpha/d+\rho)}L^{-(2\alpha/d+\rho)}.
\end{equation}
 We will prove \eqref{eq:negativeResultH}  by contradiction, then Theorem \ref{thm:lowInfty} holds as a consequence. 
 %The result of \eqref{eq:negativeResultH} will be used later in Section \ref{sec:NormDef}. 
 Assuming Equation \eqref{eq:negativeResultH} is false, we have the following claim.
\begin{claim}
    \label{thm:lowClaim}
    There exist $\rho>0$ and $C>0$ such that given any $f\in \holder{\lambda}{\alpha}$, there exists $J_0=J_0(\rho,C,f)>0$ such that, for any $N,L\in \N$ with $NL\ge J_0$, there exist $\phi\in \NN(\NNinput=d\NNspace\NNmaxwidth\le N\NNspace\NNlayer\le L)$ and $\calH\in [0,1]^d$ with $\mu(\calH)\le 2^{-(d+K^d+1)}K^{-d}$, where $K=\lfloor (NL)^{2/d+\rho/(2\alpha)}\rfloor$, satisfying
    \[
    \|f-\phi\|_{L^\infty([0,1]^d\backslash\calH)} \le C \lambda  N^{-(2\alpha/d+\rho)}L^{-(2\alpha/d+\rho)}.
    \]
\end{claim}

Now let us disprove this claim to show Theorem \ref{thm:lowInfty} and Equation \eqref{eq:negativeResultH}  are true.
\begin{proof}[Disproof of Claim \ref{thm:lowClaim}]
Without the loss of generality, we assume $\lambda  =1$; in the case of $\lambda  \neq 1$, the proof is similar.
    We will disprove Claim \ref{thm:lowClaim} using the VC dimension.
    Recall that the VC dimension of a class of functions is defined as the cardinality of the largest set of points that this class of functions can shatter.
    Denote the VC dimension of a function set $\mathscr{F}$ by $\tn{VCDim} (\mathscr{F})$. By \cite{pmlr-v65-harvey17a} and the fact
    \begin{equation*}
        \NN(\NNmaxwidth\le N\NNspace\NNdepth\le L)\subseteq \NN\big(\NNparameter\le (LN+d+2)(N+1)\big),
    \end{equation*}
    there exists $C_1>0$ such that
    \begin{equation}
    \label{eq:NLVcdimUpperBound}
    \begin{split}
    &\quad \tn{VCDim} \big(\NN(\NNinput=d\NNspace\NNmaxwidth\le N\NNspace\NNlayer\le L)\big)\\
    & \le C_1(LN+d+2)(N+1)L\ln \big( (LN+d+2)(N+1)\big)\\
    &\eqqcolon b_u(N,L).
    \end{split}
    \end{equation}

    Then we will use Claim \ref{thm:lowClaim} to estimate a lower bound of 
    \begin{equation}\label{eqn:vcd}
    \tn{VCDim} \big(\NN(\NNinput=d\NNspace\NNmaxwidth\le N\NNspace\NNlayer\le L)\big),
    \end{equation}
     and this lower bound is asymptotically larger than $b_u(N,L)$, which leads to a contradiction. 
     
     More precisely, we will construct $\{f_\chi:\chi\in \mathscr{B}\}\subseteq \holder{1}{\alpha}$, which can shatter $b_\ell(N,L)\coloneqq K^d$ points, where $\mathscr{B}$ is a set defined later. Then by Claim \ref{thm:lowClaim}, there exists $\{\phi_\chi:\chi\in\mathscr{B}\}$ such that this set can shatter $b_\ell(N,L)$ points. Finally, $b_\ell(N,L)=K^d=\lfloor (NL)^{2/d+\rho/(2\alpha)}\rfloor^d$ is asymptotically larger than $b_u(N,L)=C_1(LN+d+2)(N+1)L\ln \big( (LN+d+2)(N+1)\big)$, which leads to a contradiction. More details can be found below.
    
    \mystep{1}{Construct $\{f_\chi:\chi\in \mathscr{B}\}\subseteq\holder{1}{\alpha}$ that scatters $b_\ell(N,L)$ points.}

    Divide $[0,1]^d$ into $K^d$ non-overlapping sub-cubes $\{Q_{\bm{\beta}}\}_{\bm{\beta}}$ as follows: 
     \[Q_{\bm{\beta}}\coloneqq \big\{{\bm{x}}=[x_1,x_2,\cdots,x_d]^T\in[0,1]^d: x_i\in [\tfrac{\beta_i-1}{K},\tfrac{\beta_i}{K}],\ i=1,2,\cdots,d\big\},\] for any index vector  ${\bm{\beta}}= [\beta_1,\beta_2,\cdots,\beta_d]^T\in \{1,2,\cdots,K\}^d$.
     
    Let $Q (\bm{x}_0,\eta)\subseteq [0,1]^d$ be a hypercube, whose center and sidelength are ${\bm{x}}_0$ and $\eta$, respectively. Then we define a function $\zeta_Q$ on $[0,1]^d$ corresponding to $Q=Q (\bm{x}_0,\eta)\subseteq [0,1]^d$  such that:
    \begin{itemize}
        \item $\zeta_{Q} (\bm{x}_0)= (\eta/2)^\alpha/2$;
        \item $\zeta_{Q} (\bm{x})=0$ for any $\bmx\notin Q \backslash\partial Q $, where $\partial Q $ is the boundary of $Q$;
        \item $\zeta_{Q }$ is linear on the line that connects ${\bm{x}}_0$ and ${\bm{x}}$, for any ${\bm{x}}\in \partial Q$.
    \end{itemize}

     Define
     \begin{equation*}
         \mathscr{B}\coloneq \big\{\chi: \chi \tn{ is  a map from } \{1,2,\cdots,K\}^d \tn{   to } \{-1,1\}\big\}.
     \end{equation*}
     For each $\chi\in \mathscr{B}$, we define
    \begin{equation*}
    f_\chi (\bm{x})\coloneq \sum_{{\bm{\beta}}\in \{1,2,\cdots,K\}^d} \chi ({\bm{\beta}})\zeta_{Q_{\bm{\beta}}} (\bm{x}),
    \end{equation*}
    where $\zeta_{Q_{\bm{\beta}}} (\bm{x})$ is the associated function introduced just above. 
    It is easy to check that $ \{f_\chi:\chi\in \mathscr{B}\}\subseteq \holder{1}{\alpha}$  can shatter 
    $b_\ell(N,L)= K^d$
    points.
    
    \mystep{2}{Construct $\{\phi_\chi:\chi\in\mathscr{B}\}$ that scatters $b_\ell(N,L)$ points.}

    By Claim \ref{thm:lowClaim}, there exist $\rho>0$ and $C_2>0$ such that, for any $f_\chi\in \{f_\chi:\chi\in \mathscr{B}\}$ there exists $J_\chi>0$ such that for all $N,L\in \N$ with $NL\ge J_\chi$, there exist $\phi_\chi\in \NN(\NNinput=d\NNspace\NNmaxwidth\le N\NNspace\NNlayer\le L)$ and $\calH_\chi$ with $\mu(\calH_\chi)\le 2^{-(d+K^d+1)}K^{-d}$ such that 
    \begin{equation*}
    \label{eq:fMinusPhi1}
       |f_\chi(\bmx)-\phi_\chi(\bmx)|\le   C_2(NL)^{-\alpha (2/d+\rho/\alpha)}, \quad \tn{for any } \bmx \in [0,1]^d\backslash\calH_\chi . 
    \end{equation*}
    
    Set $\calH=\cup_{\chi\in \mathscr{B}} \calH_\chi$ and $J_1=\max_{\chi\in \mathscr{B}} J_\chi$. Then it holds that
    \begin{equation}
    \label{eq:HUpperBound}
       \mu(\calH)\le 2^{K^d}2^{-(d+K^d+1)}K^{-d}=(2K)^{-d}/2. 
    \end{equation}
     It follows that for all $\chi\in \mathscr{B}$ and $N,L\in \N$ with $NL\ge J_1$, we have
     \begin{equation}
    \label{eq:fMinusPhi}
       |f_\chi(\bmx)-\phi_\chi(\bmx)|\le   C_2(NL)^{-\alpha (2/d+\rho/\alpha)}, \quad \tn{for any } \bmx \in [0,1]^d\backslash\calH. 
    \end{equation}   

     For each index vector ${\bm{\beta}}\in \{1,2,\cdots,K\}^d$ and any $\bm{x} \in \tfrac{1}{2}Q_{\bm{\beta}}$, where $\tfrac12Q_\bmbeta$ denotes the cube whose sidelength is half of that of $Q_\bmbeta$ sharing the same center of $Q_\bmbeta$, since $Q_{\bm{\beta}}$ has a sidelength $\tfrac{1}{K}=\lfloor (NL)^{2/d+\rho/ (2\alpha)}\rfloor^{-1}$, we have
        \begin{equation}
        \label{eq:fBetaLowerBound}
        |f_\chi (\bm{x})|=|\zeta_{Q_{\bm{\beta}}}(\bm{x})|\ge |\zeta_{Q_{\bm{\beta}}} (\bm{x}_{Q_{\bm{\beta}}})|/2= \left (\tfrac{1}{2K}\right)^\alpha/4=\tfrac{1}{2^{2+\alpha}}\lfloor (NL)^{2/d+\rho/ (2\alpha)}\rfloor^{-\alpha},
        \end{equation}
        where $\bm{x}_{Q_{\bm{\beta}}}$ is the center of $Q_{\bm{\beta}}$.
         For fixed $d$, $\alpha$, and $\rho$, there exists $J_2>0$ large enough such that, for any $N,L\in \N$ with $NL\ge J_2$, we have
        \begin{equation}
        \label{eq:J2}
        \tfrac{1}{2^{2+\alpha}}\lfloor (NL)^{2/d+\rho/ (2\alpha)}\rfloor^{-\alpha}> C_2(NL)^{-\alpha (2/d+\rho/\alpha)}.
        \end{equation}
    
    By Equation \eqref{eq:HUpperBound}, for any $\bm{\beta}\in \{1,2,\cdots,K\}^d$, we have
    \begin{equation*}
    \mu(\calH) \le (2K)^{-d}/2< (2K)^{-d}=\mu (\tfrac{1}{2}Q_{\bm{\beta}}),
    \end{equation*}
    which means
     $(\tfrac12Q_{\bm{\beta}})\cap ([0,1]^d\backslash\calH)$ is not empty.
    Therefore, there exists $\bmx_{\bm{\beta}}\in (\tfrac12Q_{\bm{\beta}})\cap ([0,1]^d\backslash\calH)$ for each $\bm{\beta}\in \{1,2,\cdots,K\}^d$
    such that
    \begin{equation*}
    |f_\chi (\bmx_{\bm{\beta}})|\ge  \tfrac{1}{2^{2+\alpha}}\lfloor (NL)^{2/d+\rho/ (2\alpha)}\rfloor^{-\alpha}
    >  C_2(NL)^{-\alpha (2/d+\rho/\alpha)}
    \ge  |f_\chi (\bmx_{\bm{\beta}})-\phi_\chi (\bmx_{\bm{\beta}})|,\label{inq3}
    \end{equation*}
    for any $N,L\in\N$ with $NL\ge J_0=\max\{J_1,J_2\}$, where the first, the second, and the last inequalities come from \eqref{eq:fBetaLowerBound}, \eqref{eq:J2}, and \eqref{eq:fMinusPhi}, respectively. In other words, for any $\chi\in\mathscr{B}$ and ${\bm{\beta}}\in \{1,2,\cdots,K\}^d$, $f_\chi (\bmx_{\bm{\beta}})$ and $\phi_\chi (\bmx_{\bm{\beta}})$ have the same sign. Then  $ \{\phi_\chi:\chi\in\mathscr{B}\}$ shatters $\big\{\bm{x_\beta}:\bm{\beta}\in \{1,2,\cdots,K\}^d\big\}$ since $\{f_\chi:\chi\in\mathscr{B}\}
    $ shatters $\big\{\bm{x_\beta}:\bm{\beta}\in \{1,2,\cdots,K\}^d\big\}$ as discussed in Step $1$. Hence, 
    \begin{equation}
    \label{eq:NLVcdimLowerBound}
    \tn{VCDim}\big( \{\phi_\chi:\chi\in\mathscr{B}\} \big)\geq K^d=b_\ell(N,L),
    \end{equation}
    for any $N,L\in\N$ with $NL\ge J_0$,
    
    \mystep{3}{Contradiction.}
    By Equation \eqref{eq:NLVcdimUpperBound} and \eqref{eq:NLVcdimLowerBound}, for any $N,L\in\N$ with $NL\ge J_0$, we have
    \begin{equation*}
    \begin{split}
    b_\ell(N,L)&\le \tn{VCDim}\big(\{\phi_\chi:\chi\in\mathscr{B}\}\big)
            \\&\le \tn{VCDim}\big(\NN(\NNinput=d\NNspace\NNmaxwidth\le N\NNspace\NNlayer\le L)\big)
            \le b_u(N,L),
    \end{split}
    \end{equation*}
    implying that
        \begin{equation*}
            \lfloor (NL)^{2/d+\rho/(2\alpha)}\rfloor ^d
            \le 
            C_1(LN+d+2)(N+1)L\ln \big( (LN+d+2)(N+1)\big),
    \end{equation*}
    which is a contradiction for sufficiently large $N,L\in\N$.
     So we finish the proof.    
\end{proof}

\vspace{8pt}
By Theorem \ref{thm:lowInfty}, for any $\rho>0$, the approximation rate cannot be better than $\calO(N^{-(2\alpha/d+\rho)}L^{-(2/\alpha+\rho)})$, if we use FNNs in $\NN(\NNinput=d\NNspace\NNmaxwidth\le N\NNspace\NNlayer\le L)$ to approximate functions in $\holder{\lambda}{\alpha}$. By a similar argument, we can show that the approximation rate cannot be $\calO(N^{-2\alpha/d}L^{-(2/\alpha+\rho)})$ nor $\calO(N^{-(2\alpha/d+\rho)}L^{-2\alpha/d})$. Hence, the approximation rate in Theorem \ref{thm:main} is nearly tight. 

\section{Proof of Theorem \ref{thm:mainGap}}
\label{sec:mainTriflingRegion}
In this section, we will prove Theorem \ref{thm:mainGap}. We first present the key ideas in Section \ref{sec:keyIdea}. Based on two propositions in Section \ref{sec:keyIdea}, the detailed proof is presented in Section \ref{sec:proofMainGap}. 
Finally, the proofs of two propositions in Section \ref{sec:keyIdea} can be found in Section \ref{sec:proofProp1} and \ref{sec:proofProp2}.

\subsection{Key ideas of proving Theorem \ref{thm:mainGap}}
\label{sec:keyIdea}
\begin{figure}[!htp]
	\centering
	\includegraphics[width=0.872\textwidth]{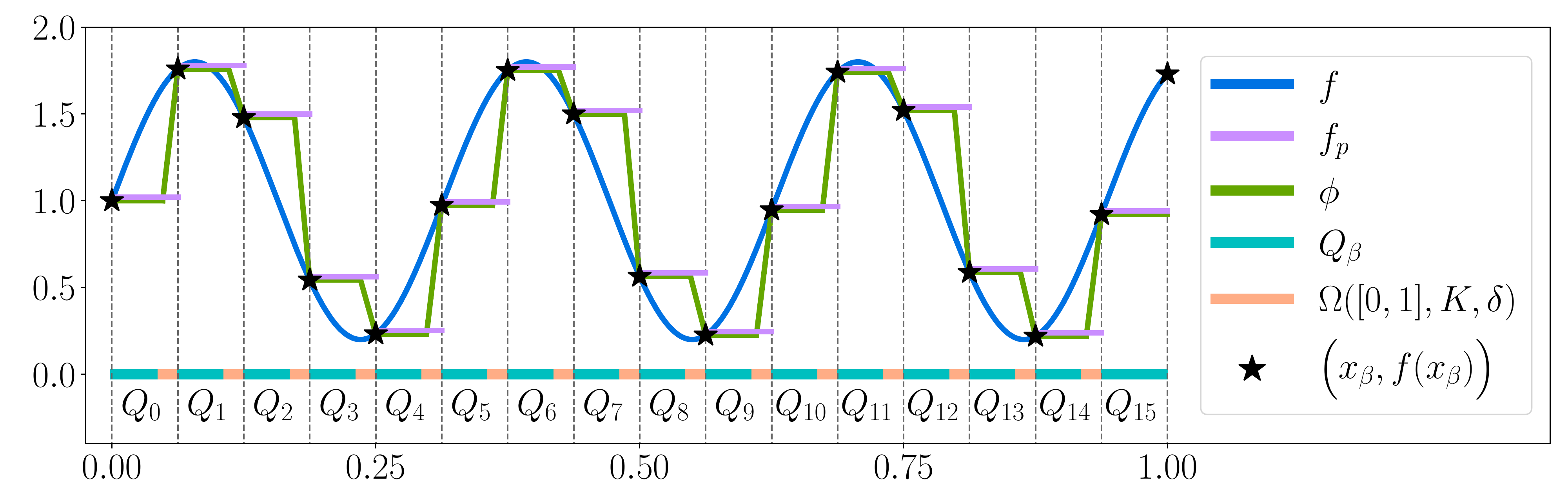}
	\caption{An illustration of $f$, $f_p$, $\phi$, $x_\beta$, $Q_{\beta}$, and the trifling  region $\Omega([0,1]^d,K,\delta)$ in the one-dimensional case for $\beta\in \{0,1,\cdots,K-1\}^d$, where $K=N^2L^2$ and  $d=1$ with $N=2$ and $L=2$. $f$ is the target function; $f_p$ is the piecewise constant function approximating $f$; $\phi$ is a function, implemented by a ReLU FNN,  approximating $f$; and $x_\beta$ is a representative of $Q_\beta$. The measure of the trifling region $\Omega([0,1]^d,K,\delta)$ can be arbitrarily small as we shall see in the proof of Theorem \ref{thm:main}.}
	\label{fig:Q}
\end{figure}

We will show that an almost piecewise constant function $\phi$ implemented by a ReLU FNN is enough to achieve the desired approximation rate in Theorem \ref{thm:main}. Given an arbitrary $f\in C([0,1]^d)$, we introduce a piecewise constant function $f_p\approx f$ serving as an intermediate approximant in our construction in the sense that
\[
f\approx f_p \tn{ on $[0,1]^d$,} \quad \tn{and}\quad f_p\approx \phi \tn{ on $[0,1]^d\backslash\Omega([0,1]^d,K,\delta)$} .
\]
The approximation in $f\approx f_p$ is a simple and standard technique in constructive approximation. For example, given arbitrary $N$ and $L$, uniformly partition $[0,1]^d$ into $\calO(N^2L^2)$ pieces and define $f_p$ using this partition. Then the approximation error of $f_p\approx f$ scales like $\calO(N^{-2/d}L^{-2/d})$.  We will address the approximation in $f_p\approx \phi$ with the same error scaling and a limited budget of the FNN size, e.g., $\calO(NL)$ neurons, based on the fact that $f_p$ can be approximately implemented by a ReLU FNN in $[0,1]^d\backslash\Omega([0,1]^d,K,\delta)$, where $\Omega([0,1]^d,K,\delta)$ is the trifling  region near the discontinuous locations of $f_p$ with an arbitrarily small Lebesgue measure (see Figure \ref{fig:Q} for an illustration). The introduction of the trifling  region is to ease the construction of a deep ReLU FNN to implement the desired $\phi$, which is a piecewise linear and continuous function, to approximate the discontinuous function $f_p$ by removing the difficulty near discontinuous points, essentially smoothing $f_p$ by restricting the approximation domain in $[0,1]^d\backslash\Omega([0,1]^d,K,\delta)$.

Now let us discuss the detailed steps of construction.
First, divide $[0,1]^d$  into a union of important regions $\{Q_\bmbeta\}_\bmbeta$ and the trifling  region $\Omega([0,1]^d,K,\delta)$, where each $Q_\bmbeta$ is associated with a representative $\bm{x}_\bmbeta\in Q_\bmbeta$ such that $f(\bm{x}_\bmbeta)=f_p(\bm{x}_\bmbeta)$ for each index vector $\bmbeta\in \{0,1,\dots,K-1\}^d$, where $K=\calO(N^{2/d}L^{2/d})$ is the partition number per dimension (see Figure \ref{fig:Q+TR} for examples for $d=1$ and $d=2$). 
Next, we design a vector function $\bmPhi_1(\bmx)$ constructed via $\bmPhi_1(\bmx)=\big[\phi_1(x_1),\phi_1(x_2),\cdots,\phi_1(x_d)\big]^T$ to project the whole cube $Q_\bmbeta$ to a $d$-dimensional index $\bmbeta$ for each $\bmbeta$, where each one-dimensional function $\phi_1$  is a step function implemented by a ReLU FNN. The final step is to solve a point fitting problem. To be precise, we construct a function $\phi_2$ implemented by a ReLU FNN to map $\bmbeta$ approximately to $f_p(\bmx_\bmbeta)=f(\bmx_\bmbeta)$. Then $\phi_2\circ\bmPhi_1(\bmx)=\phi_2(\bmbeta)\approx f_p(\bmx_\bmbeta)=f(\bmx_\bmbeta)$ for any $\bmx\in Q_\bmbeta$ and each $\bmbeta$, implying 
$\phi\coloneqq \phi_2\circ\bmPhi_1 \approx f_p\approx f $ on $ [0,1]^d\backslash\Omega([0,1]^d,K,\delta)$. We would like to point out that we only need to care about the values of $\phi_2$ at a set of points $\{0,1,\cdots,K-1\}^d$ in the construction of $\phi_2$ according to our design $\phi=\phi_2\circ\bmPhi_1$ as illustrated in Figure \ref{fig:idea}. Therefore, it is unnecessary to care about the values of $\phi_2$ sampled outside the set $\{0,1,\cdots,K-1\}^d$, which is a key point to ease the design of a ReLU FNN to implement $\phi_2$ as we shall see later. 

\begin{figure}[!htp]        
	\centering
	\includegraphics[width=0.999\textwidth]{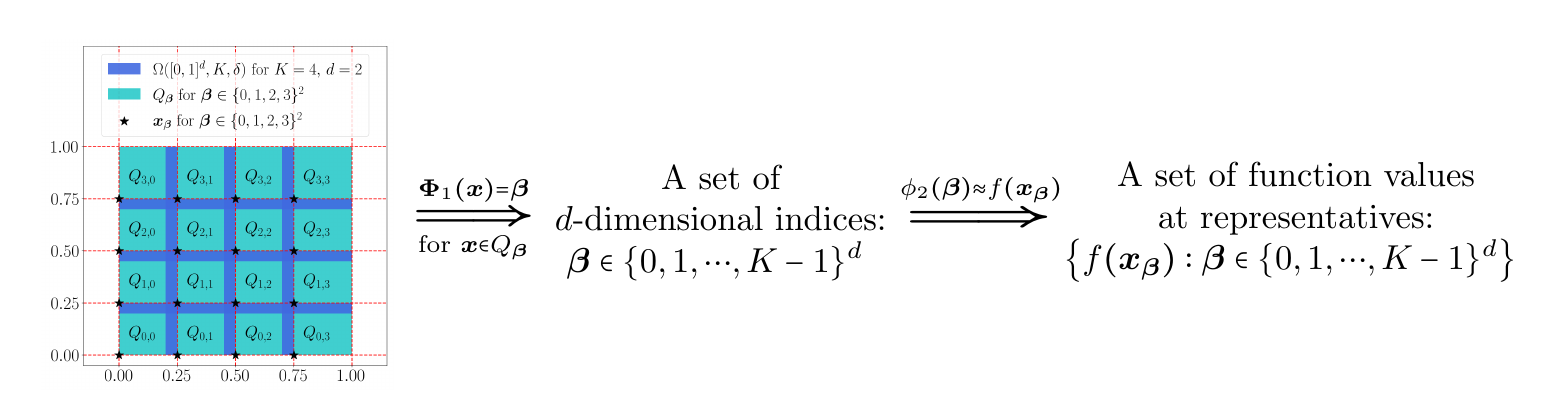}
	\caption{An illustration of the desired function $\phi=\phi_2\circ\bmPhi_1$. Note that $\phi\approx f$ on $[0,1]^d\backslash \Omega([0,1]^d,K,\delta)$, since $\phi(\bmx)=\phi_2\circ\bmPhi_1(\bmx)=\phi_2(\bmbeta)\approx f(\bmx_\bmbeta)$ for any $\bmx\in Q_\bmbeta$ and each $\bmbeta\in \{0,1,\cdots,K-1\}^d$.}
	\label{fig:idea}
\end{figure}

Finally, we discuss how to implement $\bmPhi_1$ and $\phi_2$ by deep ReLU FNNs with width $\calO(N)$ and depth $\calO(L)$ using two propositions as we shall prove in Section \ref{sec:proofProp1} and \ref{sec:proofProp2} later. %The construction of $\bmPhi_1$ is based on Proposition \ref{prop:stepFunc}. 
We first construct a ReLU FNN with desired width and depth  by Proposition \ref{prop:stepFunc} to implement a one-dimensional step function $\phi_1$. Then $\bmPhi_1$ can be attained via defining \[\bmPhi_1(\bmx)=\big[\phi_1(x_1),\phi_1(x_2),\cdots,\phi_1(x_d)\big]^T,\quad \tn{for any $\bmx=[x_1,x_2,\cdots,x_d]^T\in \R^d$.}\]

\begin{proposition}
	\label{prop:stepFunc}
	For any $N,L,d\in \N^+$ and $\delta\in(0, \tfrac{1}{3K}]$ with $K=\lfloor N^{1/d}\rfloor^2 \lfloor L^{2/d}\rfloor$, there exists a one-dimensional function $\phi$ implemented by a ReLU FNN  with width $4\lfloor N^{1/d}\rfloor +3$ and depth $4L+5$ such that
	\begin{equation*}
	\phi(x)=k,\quad \tn{if $x\in [\tfrac{k}{K},\tfrac{k+1}{K}-\delta\cdot 1_{\{k\le K-2\}}]$ for $k=0,1,\cdots,K-1$.}
	\end{equation*} 
\end{proposition}

The construction of $\phi_2$ is a direct result of Proposition \ref{prop:pointFitting} below, the proof of which relies on 
%both the idea illustrated in Figure \ref{fig:wideToDeepIdea} and
the bit extraction technique in   \cite{Bartlett98almostlinear}. 

\begin{proposition}
	\label{prop:pointFitting}
	Given any $\epsilon>0$ and  arbitrary $N,L,J\in \N^+$ with $J\le N^2L^2$, assume   $\{{y}_j\ge 0:j=0,1,\cdots,J-1\}$ is a sample set with $|y_{j}-y_{j-1}|\le \epsilon$ for $j=1,2,\cdots,J-1$. Then there exists $\phi\in \NN(\NNinput=1\NNspace\NNwidth\le 12N+8\NNspace\NNdepth\le 4L+9\NNspace\NNoutput=1)$ such that
	\begin{enumerate}[(i)]
		\item $|\phi(j)-{y}_j|\le \epsilon$ for $j=0,1,\cdots,J-1$;
		\item $0\le \phi(x)\le  \max\{y_{j}:j=0,1,\cdots,J-1\}$ for any $x\in\R$.
	\end{enumerate}
\end{proposition}

With the above propositions ready, let us prove Theorem \ref{thm:mainGap} in Section \ref{sec:proofMainGap}. We further assume that $\omega_f(r)>0$ for any $r>0$, excluding a simple case when $f$ is a constant function.

%%%%%%%%%%%%%%%%%%%%%%%%%%%%%%%%%%%%%%%%%%%
%%%%%%%%%%%%%%%%%%%%%%%%%%%%%%%%%%%%%%%%
\subsection{Proof of Theorem \ref{thm:mainGap}}
\label{sec:proofMainGap}

We essentially construct an almost piecewise constant function implemented by a ReLU FNN with $\calO(NL)$ neurons to approximate $f$. \red{We may $f$ is not a constant since it is a trivial case. Then $\omega_f(r)>0$ for any $r>0$.}
It is clear that $|f(\bm{x})-f(\bmzero)|\le \omega_f(\sqrt{d})$ for any ${\bm{x}}\in [0,1]^d$. Define $\tildef=f-f(\bmzero)+\omega_f(\sqrt{d})$, then $0\le \tildef (\bm{x}) \le 2\omega_f(\sqrt{d})$ for any ${\bm{x}}\in [0,1]^d$. Let $M=N^2L$,  $K=\lfloor N^{1/d}\rfloor^2 \lfloor L^{2/d}\rfloor$, and $\delta$ be an arbitrary number in $(0,\tfrac{1}{3K}]$.

The proof can be divided into four steps as follows:
\begin{enumerate}
	\item Divide $[0,1]^d$ into a union of sub-cubes $\{Q_{\bm{\beta}}\}_{\bm{\beta}\in \{0,1,\cdots,K-1\}^d}$ and the trifling region $\Omega([0,1]^d,K,\delta)$, and denote $\bmx_\bmbeta$ as the vertex of $Q_\bmbeta$ with minimum $\|\cdot\|_1$ norm;
	
	\item Construct a sub-network to implement a vector function $\bmPhi_1$ projecting the whole cube $ Q_\bmbeta$ to the $d$-dimensional index $\bmbeta$ for each $\bmbeta$, i.e., $\bmPhi_1(\bmx)=\bmbeta$ for all $\bmx\in Q_\bmbeta$;
	
	\item Construct a sub-network to implement a  function $\phi_2$  mapping the index $\bmbeta$ approximately to $\tildef(\bmx_\bmbeta)$. This core step can be further divided into three sub-steps:	
	\begin{enumerate}
		\item[3.1.] Construct a sub-network to implement $\psi_1$ bijectively mapping the index set $\{0,1,\cdots,K-1\}^d$ to an auxiliary set $\mathcal{A}_1\subseteq \big\{\tfrac{j}{2K^d}:j=0,1,\cdots,2K^d\big\}$ defined later (see Figure \ref{fig:g+A12} for an illustration);
		
		\item[3.2.] Determine a continuous piecewise linear function $g$ with  a set of breakpoints $\calA_1\cup\calA_2\cup\{1\}$  satisfying: 1) assign the values of $g$ at breakpoints in $\calA_1$ based on $\{\tildef(\bmx_\bmbeta)\}_\bmbeta$, i.e., $g\circ \psi_1(\bmbeta)=\tildef(\bmx_\bmbeta)$; 2) assign the values of $g$ at breakpoints in $\calA_2\cup\{1\}$ to reduce the variation of $g$  for applying Proposition \ref{prop:pointFitting};
		%		such that $\tildef$ and $g\circ \phi_2 \circ \bmPhi_1$ have the same value at the elements of $\{\tfrac{k}{K}:k=0,1,\cdots,K-1\}^d$, i.e. $\tildef\approx g\circ \phi_2\circ \bmPhi_1$;
		
		\item[3.3.] Apply Proposition \ref{prop:pointFitting} to construct a sub-network to implement a function $\psi_2$ approximating $g$ well on $\calA_1\cup\calA_2\cup\{1\}$. Then the desired function $\phi_2$ is given by  $\phi_2=\psi_2\circ\psi_1$ satisfying $\phi_2(\bmbeta)=\psi_2\circ\psi_1(\bmbeta)\approx g\circ\psi_1(\bmbeta)=\tildef(\bmx_\bmbeta)$;
		%		based on Proposition \ref{prop:pointFitting} to implement a function  $\psi_3$ mapping $\mathcal{A}_1$ approximately to $\big\{\tildef(\bmx):\bmx\in \{\tfrac{k}{K}:k=0,1,\cdots,K-1\}^d\big\}$ such that $g\approx \phi_3$ on $\mathcal{A}_1$;
	\end{enumerate}
	
	\item Construct the final target network to implement the desired function $\phi$ such that $\phi(\bmx)= \phi_2 \circ \bmPhi_1(\bmx) +f(\bmzero)-\omega_f(\sqrt{d})\approx \tildef(\bmx_\bmbeta) +f(\bmzero)-\omega_f(\sqrt{d}) = f(\bmx_\bmbeta)$ for $\bmx\in Q_\bmbeta$.
\end{enumerate}    

The details of these steps can be found below.
\mystep{1}{Divide $[0,1]^d$ into  $\{Q_{\bm{\beta}}\}_{\bm{\beta}\in \{0,1,\cdots,K-1\}^d}$ and  $\Omega([0,1]^d,K,\delta)$.}

Define  $\bmx_\bmbeta \coloneqq \bmbeta/K$ and 
\[
Q_{\bm{\beta}}\coloneqq\Big\{{\bm{x}}= [x_1,\cdots,x_d]^T\in [0,1]^d:x_i\in[\tfrac{\beta_i}{K},\tfrac{\beta_i+1}{K}-\delta\cdot 1_{\{\beta_i\le K-2\}}], \ i=1,\cdots,d\Big\}
\]
for each $d$-dimensional index  ${\bm{\beta}}= [\beta_1,\cdots,\beta_d]^T\in \{0,1,\cdots,K-1\}^d$. Recall that $\Omega([0,1]^d,K,\delta)$ is the trifling region defined in Equation \eqref{eq:triflingRegionDef}. Apparently, $\bmx_\bmbeta$ is the vertex of $Q_\bmbeta$ with minimum $\|\cdot\|_1$ norm and 
\[[0,1]^d= \big(\cup_{\bm{\beta}\in \{0,1,\cdots,K-1\}^d}Q_{\bm{\beta}}\big)\cup \Omega([0,1]^d,K,\delta),\]
see Figure \ref{fig:Q+TR} for illustrations.

\begin{figure}[!htp]
	\centering
	\begin{minipage}{0.905\textwidth}
		\centering
		\begin{subfigure}[b]{0.522\textwidth}
			\centering
			\includegraphics[width=0.999\textwidth]{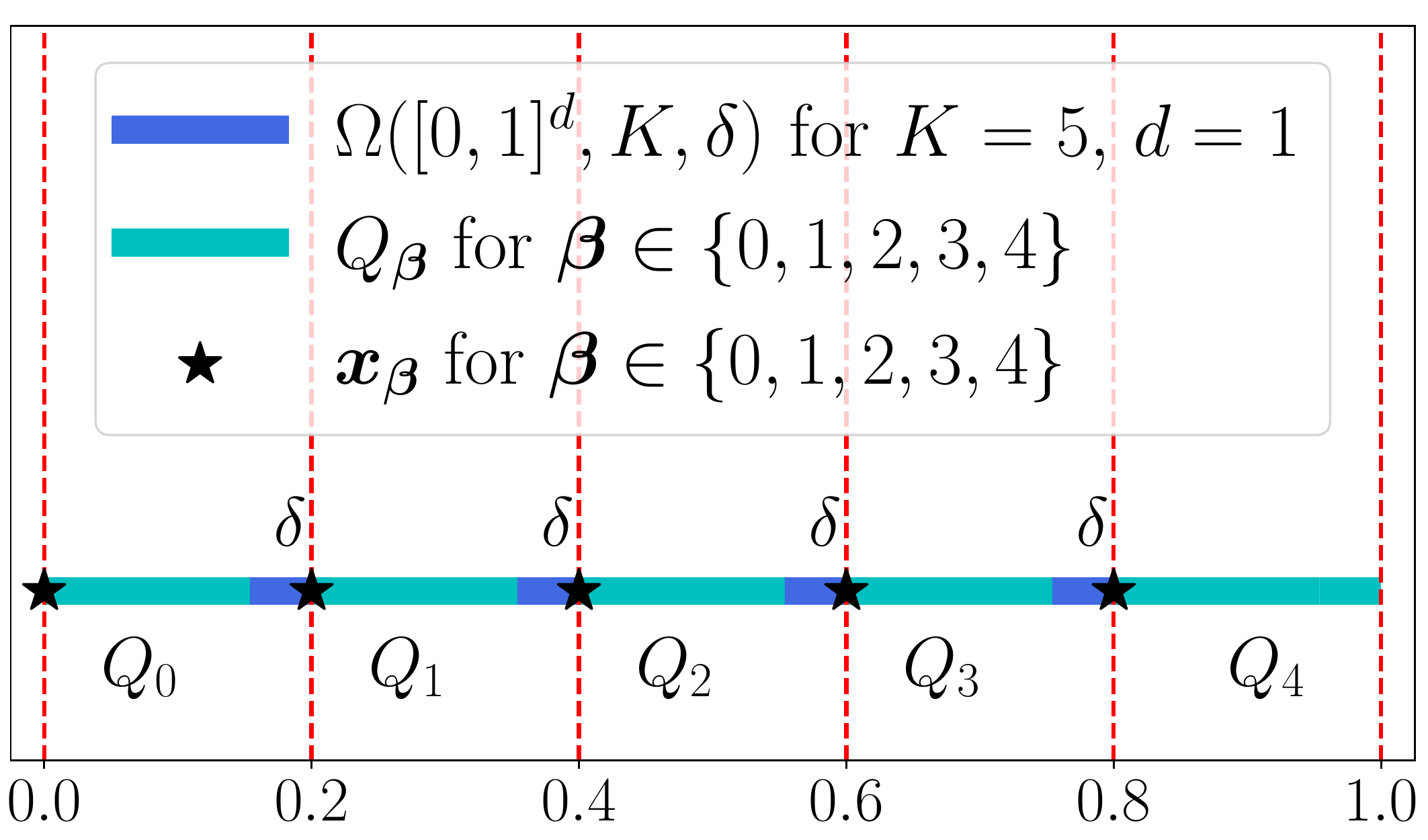}
			\subcaption{}
		\end{subfigure}
			\begin{minipage}{0.021\textwidth}
				\
			\end{minipage}
		\begin{subfigure}[b]{0.303\textwidth}
			\centering
			\includegraphics[width=0.999\textwidth]{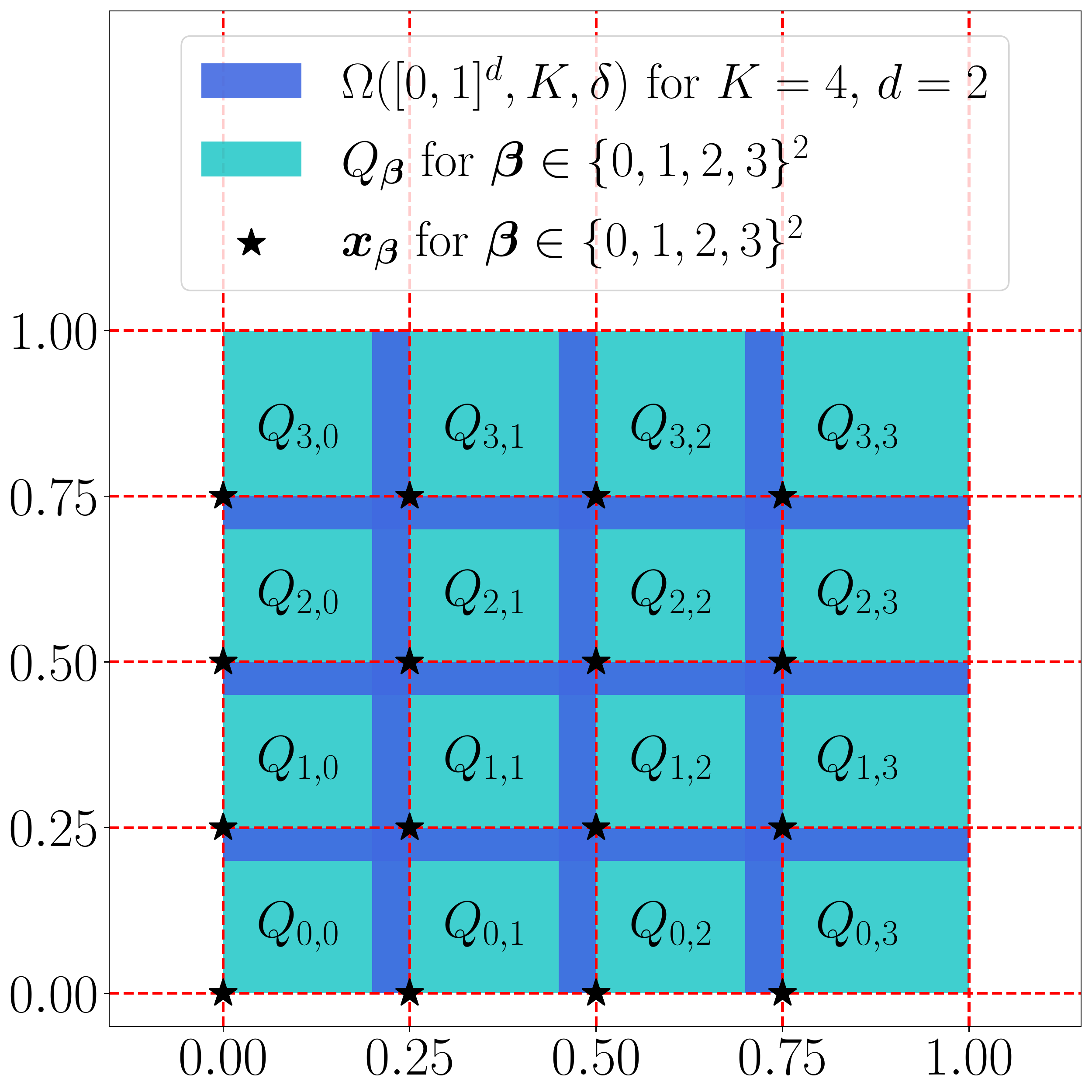}
			\subcaption{}
		\end{subfigure}
	\end{minipage}
	\caption{Illustrations of  $\Omega([0,1]^d,K,\delta)$,  $Q_\bmbeta$, and $\bmx_\bmbeta$ for $\bmbeta\in \{0,1,\cdots,K-1\}^d$. (a) $K=5$ and $d=1$. (b) $K=4$ and $d=2$. }
	\label{fig:Q+TR}
\end{figure}

\mystep{2}{Construct $\bmPhi_1$ mapping $\bmx\in Q_\bmbeta$ to  $\bmbeta$.}

By Proposition \ref{prop:stepFunc}, there exists $\phi_1\in \NNF(
\NNwidth \le 4\lfloor N^{1/d}\rfloor +3\NNspace\NNdepth\le 4L+5)$ such that
\begin{equation*}
\phi_1(x)=\red{k},\quad \tn{if $x\in [\tfrac{k}{K},\tfrac{k+1}{K}-\delta\cdot 1_{\{k\le K-2\}}]$ for $k=0,1,\cdots,K-1$.}
\end{equation*}    
It follows that $\phi_1(x_i)=\beta_i$ if $\bmx=[x_1,x_2,\cdots,x_d]^T\in Q_\bmbeta$ for each $\bmbeta=[\beta_1,\beta_2,\cdots,\beta_d]^T$.

By defining
\begin{equation*}
\bmPhi_1(\bmx)\coloneqq \big[\phi_1(x_1),\phi_1(x_2),\cdots,\phi_1(x_d)\big]^T,\quad \tn{for any } \bmx=[x_1,x_2,\cdots,x_d]^T\in \R^d,
\end{equation*}
we have $\bmPhi_1(\bmx)=\bmbeta$ if $\bmx\in Q_\bmbeta$ for $\bmbeta\in \{0,1,\cdots,K-1\}^d$.

\mystep{3}{Construct $\phi_2$ mapping $\bmbeta$ approximately to $\tildef(\bmx_\bmbeta)$.}

The construction of the sub-network implementing $\phi_2$ is essentially based on Proposition \ref{prop:pointFitting}. 
To meet the requirements of applying Proposition \ref{prop:pointFitting}, we first define two auxiliary set $\calA_1$ and $\calA_2$ as 
\begin{equation*}
\calA_1\coloneqq \big\{\tfrac{i}{K^{d-1}}+\tfrac{k}{2K^d}:i=0,1,\cdots,K^{d-1}\red{-1}\tn{\quad and \quad} k=0,1,\cdots,K-1\big\}
\end{equation*}
and 
\begin{equation*}
\calA_2\coloneqq \big\{\tfrac{i}{K^{d-1}}+\red{\tfrac{K+k}{2K^d}}:i=0,1,\cdots,K^{d-1}\red{-1}\tn{\quad and \quad}k=0,1,\cdots,K-1\big\}.
\end{equation*}
Clearly, $\calA_1\cup\calA_2\cup\{1\}=\{\tfrac{j}{2K^d}:j=0,1,\cdots,2K^d\}$ and $\calA_1\cap\calA_2=\emptyset$. See Figure \ref{fig:Q+TR} for an illustration of $\calA_1$ and $\calA_2$. Next, we further divide this step into three sub-steps.

\mystep{3.1}{Construct $\psi_1$ bijectively mapping  $\{0,1,\cdots,K-1\}^d$ to $\mathcal{A}_1$.}

Inspired by the binary representation, we define
\begin{equation}
\psi_1(\bmx)\coloneqq \tfrac{x_d}{2K^d}+\sum_{i=1}^{d-1}\tfrac{x_i}{K^i},\quad \tn{for any $\bmx=[x_1,x_2,\cdots,x_d]^T\in \R^d$.}
\end{equation}
Then $\psi_1$ is a linear function bijectively mapping the index set $\{0,1,\cdots,K-1\}^d$ to
\begin{equation*}
\begin{split}
&\quad \Big\{\tfrac{\beta_d}{2K^d}+\sum_{i=1}^{d-1}\tfrac{\beta_i}{K^i}:\bm{\beta}\in \{0,1,\cdots,K-1\}^d\Big\}\\
&=\big\{\tfrac{i}{K^{d-1}}+\tfrac{k}{2K^d}:i=0,1,\cdots,K^{d-1}\red{-1}\tn{\quad and\quad } k=0,1,\cdots,K-1\big\}=\calA_1.
\end{split}
\end{equation*}

\mystep{3.2}{Construct $g$ to satisfy $g\circ\psi_1(\bmbeta)=\tildef(\bmx_\bmbeta)$ and to meet the requirements of applying Proposition \ref{prop:pointFitting}.}

Let $g:[0,1]\to \R$ be a continuous piecewise linear function with a set of breakpoints $\left\{\tfrac{j}{2K^d}: j=0,1,\cdots,2K^d\right\}=\calA_1\cup\calA_2\cup\{1\}$ and the values of $g$ at these breakpoints satisfy the following properties:

\begin{itemize}
	\item The values of $g$ at the breakpoints in $\mathcal{A}_1$ are set as
	\begin{equation}
	\label{eq:ftog}
	g(\psi_1(\bmbeta))=\tildef(\bmx_\bmbeta),\quad \tn{for any  $\bm{\bm{\beta}}\in \{0,1,\cdots,K-1\}^d$;}
	\end{equation} 
	
	\item At the breakpoint $1$, let $g(1)=\tildef(\bm{1})$, where $\bm{1}=[1,1,\cdots,1]^T\in \R^d$;
	
	\item 
	The values of $g$ at the breakpoints in $\mathcal{A}_2$ are assigned to reduce the variation of $g$, which is a requirement of applying Proposition \ref{prop:pointFitting}. Note that 
	\begin{equation*}
	\{\tfrac{i}{K^{d-1}}-\tfrac{K+1}{2K^d},\ \tfrac{i}{K^{d-1}}\}\subseteq\calA_1\cup\{1\},\quad \tn{for $i=1,2,\cdots,K^{d-1}$,}
	\end{equation*}
	implying the values of $g$ at $\tfrac{i}{K^{d-1}}-\tfrac{K+1}{2K^d}$ and $\tfrac{i}{K^{d-1}}$ have been assigned for $i=1,2,\cdots,K^{d-1}$. Thus, the values of $g$ at the breakpoints in $\calA_2$ can be successfully assigned by letting $g$ linear on each interval 
	$[\tfrac{i}{K^{d-1}}-\tfrac{K+1}{2K^d},\, \tfrac{i}{K^{d-1}}]$ for $i=1,2,\cdots,K^{d-1}$, since $\calA_2\subseteq \cup_{i=1}^{K^{d-1}}[\tfrac{i}{K^{d-1}}-\tfrac{K+1}{2K^d},\, \tfrac{i}{K^{d-1}}]$.
\end{itemize}
\begin{figure}
	\centering
	\includegraphics[width=0.8\textwidth]{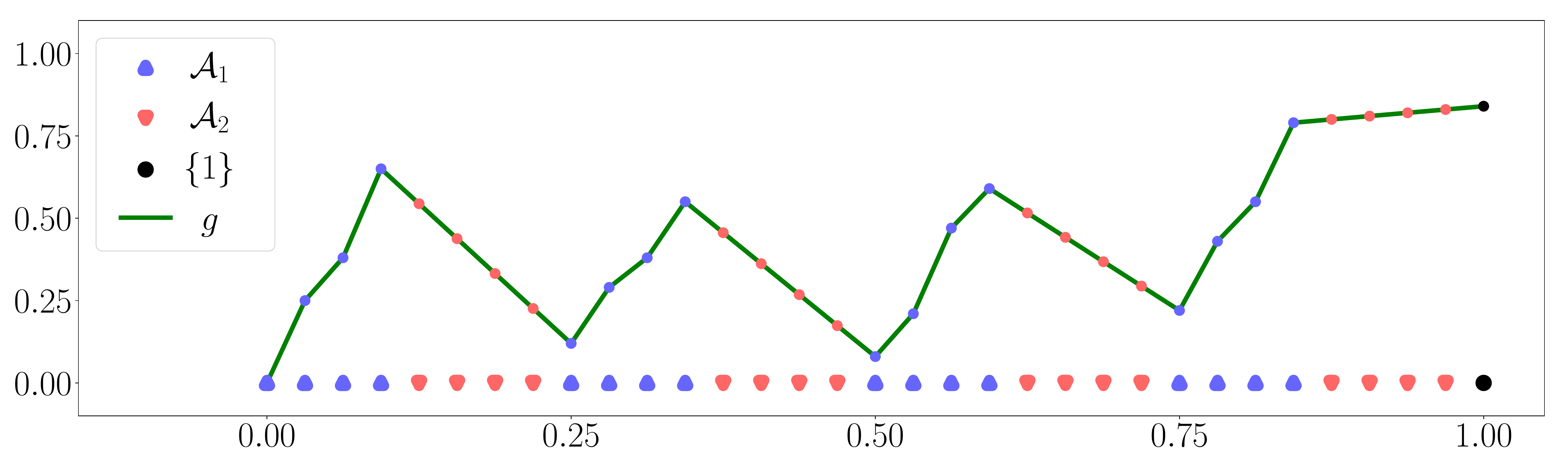}
	\caption[An illustration of $\mathcal{A}_1$, $\mathcal{A}_2$, $\{1\}$, and $g$ for $d=2$ and $K=4$]{An illustration of $\mathcal{A}_1$, $\mathcal{A}_2$, $\{1\}$, and $g$ for $d=2$ and $K=4$.}
	\label{fig:g+A12}
\end{figure}

Apparently, such a function $g$ exists (see Figure \ref{fig:g+A12} for an example) and satisfies
\begin{equation*}
\label{eq:gErrorEstimation}
\left|g(\tfrac{j}{2K^d})-g(\tfrac{j-1}{2K^d})\right|\le \max\big\{\omega_f(\tfrac{1}{K}),\omega_f({\sqrt{d}})/K\big\}\le  \omega_f(\tfrac{\sqrt{d}}{K}),\quad \tn{for } j=1,2,\cdots,2K^d,
\end{equation*} 
and
\begin{equation*}
0\le g(\tfrac{j}{2K^d})\le 2\omega_f(\sqrt{d}), \quad \tn{for}\ j=0,1,\cdots,2K^d.
\end{equation*}

\mystep{3.3}{Construct $\psi_2$ approximating $g$ well on $\calA_1\cup\calA_2\cup\{1\}$.}

Since $2K^d=2 \big(\lfloor N^{1/d}\rfloor^2 \lfloor L^{2/d}\rfloor\big)^d\le 2\big(N^2L^2\big)\le 
N^2\widetilde{L}^2$, where $\widetilde{L}=2L$, by Proposition \ref{prop:pointFitting} \red{(set $y_j=g(\tfrac{j}{2K^2})$ and $\varepsilon=\omega_f(\tfrac{\sqrt{d}}{K})>0$ therein)}, there exists $\tildepsi_2\in \NNF(\NNinput=1\NNspace\NNwidth\le 12 N+8\NNspace\NNdepth\le 4\widetilde{L}+9)=\NNF(\NNinput=1\NNspace\NNwidth\le 12N+8\NNspace\NNdepth\le 8L+9)$ such that
\begin{equation*}
%\label{eq:phi1Minusg}
|\tildepsi_2(j)-g(\tfrac{j}{2K^d})|\le \omega_f(\tfrac{\sqrt{d}}{K}),\quad  \tn{for } j=0,1,\cdots,2K^d-1,
\end{equation*}
and 
\begin{equation*}
%\label{eq:phi3tUB}
\begin{split}
0\le \tildepsi_2(x) \le  \max\{g(\tfrac{j}{2K^d}):j=0,1,\cdots,2K^d-1\}\le 2\omega_f(\sqrt{d}), \quad \tn{for any $x\in\R$.}
\end{split}
\end{equation*}

By defining $\psi_2(x)\coloneqq \tildepsi_2(2K^dx)$ for any $x\in \R$, we have $\psi_2\in \NNF(\NNinput=1\NNspace\NNwidth\le 12N+8\NNspace\NNdepth\le 8L+9)$,
\begin{equation}
\label{eq:phi3tUB}
\begin{split}
0\le \psi_2(x)=\tildepsi_2(2K^dx) \le 2\omega_f(\sqrt{d}), \quad \tn{for any } x\in\R,
\end{split}
\end{equation}
and 
\begin{equation}
\label{eq:phi1Minusg}
|\psi_2(\tfrac{j}{2K^d})-g(\tfrac{j}{2K^d})|=|\tildepsi_2(j)-g(\tfrac{j}{2K^d})|\le \omega_f(\tfrac{\sqrt{d}}{K}), \quad \tn{for $ j=0,1,\cdots,2K^d-1.$}
\end{equation}

%\mystep{3.4}{Construct $\phi_2$ mapping $\bmbeta$ approximately to $\tildef(\bmx_\bmbeta)$.}
Let us end Step $3$ by defining the desired function $\phi_2$ as 
$\phi_2\coloneqq \psi_2\circ \psi_1$. Note that $\psi_1:\R^d\to\R $ is a linear function and $\psi_2\in \NNF(\NNinput=1\NNspace\NNwidth\le 12N+8\NNspace\NNdepth\le 8L+9)$. Thus, $\phi_2\in \NNF(\NNinput=d\NNspace\NNwidth\le 12N+8\NNspace\NNdepth\le 8L+9)$.
By Equation \eqref{eq:ftog} and \eqref{eq:phi1Minusg}, we have
\begin{equation}
\label{eq:phi2-tildef}
\begin{split}
|\phi_2(\bmbeta)-\tildef(\bmx_\bmbeta)|%&=\left|\psi_2\big(\sigma(\psi_1(\bmbeta))\big)-g(\psi_1(\bmbeta))\right|\\
=\left|\psi_2(\psi_1(\bmbeta))-g(\psi_1(\bmbeta))\right|\le \omega_f(\tfrac{\sqrt{d}}{K}),
\end{split}
\end{equation}
for any $\bmbeta\in\{0,1,\cdots,K-1\}^d$.
Equation \eqref{eq:phi3tUB} and $\phi_2= \psi_2\circ \psi_1$ implies 
\begin{equation}
\label{eq:phi2tUB}
\begin{split}
0\le \phi_2(\bmx)\le 2\omega_f(\sqrt{d}), \quad \tn{for any } \bmx\in\R^d.
\end{split}
\end{equation}

\mystep{4}{Construct the final  network to implement the desired function $\phi$.}

Define $\phi\coloneqq \phi_2\circ\bmPhi_1+f(\bmzero)-\omega_f(\sqrt{d})$. 
Since $\phi_1 \in\NNF(\NNwidth\le 4\lfloor N^{1/d}\rfloor+3\NNspace\NNdepth\le 4L+5])$, we have $\bmPhi_1\in \NNF(\NNinput=d\NNspace\NNwidth\le4d\lfloor N^{1/d}\rfloor+3d\NNspace\NNdepth\le4L+5\NNspace\NNoutput=d)$. Note that $\phi_2\in \NNF(\NNinput=d\NNspace\NNwidth\le 12N+8\NNspace\NNdepth\le 8L+9)$. Thus, $\phi=\phi_2 \circ \bmPhi_1+f(\bmzero)-\omega_f(\sqrt{d})$ is in 
\begin{equation*}
\begin{split}
\NN\big(\NNwidth\le\max\{4d\lfloor N^{1/d}\rfloor+3d, 12N+8\}\NNspace\NNdepth\le (4L+5)+(8L+9)= 12L+14\big).
\end{split}
\end{equation*}

Now let us estimate the approximation error.
Note that $f=\tildef+f(\bmzero)-\omega_f(\sqrt{d})$. By Equation \eqref{eq:phi2-tildef}, for any $\bmx\in Q_\bmbeta$ and $\bmbeta\in \{0,1,\cdots,K-1\}^d$, we have
\begin{equation*}
\begin{split}
|f(\bmx)-\phi(\bmx)|&=|\tildef(\bmx)-\phi_2(\bmPhi_1(\bmx))|=|\tildef(\bmx)-\phi_2(\bmbeta)|\\
&\le |\tildef(\bmx)-\tildef(\bmx_\bmbeta)|+|\tildef(\bmx_\bmbeta)-\phi_2(\bmbeta)|\\
&\le \omega_f(\tfrac{\sqrt{d}}{K})+\omega_f(\tfrac{\sqrt{d}}{K})\le 2\omega_f(8\sqrt{d}N^{-2/d}L^{-2/d}),
\end{split}
\end{equation*}
where the last inequality comes from the fact $K=\lfloor N^{1/d}\rfloor^2\lfloor L^{2/d}\rfloor\ge \tfrac{N^{2/d}L^{2/d}}{8}$ for any $N,L\in \N^+$. Recall the fact $\omega_f(nr)\le n\omega_f(r)$ for any $n\in\N^+$ and $r\in [0,\infty)$. Therefore, for any $\bmx\in \cup_{\bm{\beta}\in \{0,1,\cdots,K-1\}^d} Q_\bmbeta\red{=} [0,1]^d\backslash \Omega([0,1]^d,K,\delta)$, we have
\begin{equation*}
\begin{split}
|f(\bmx)-\phi(\bmx)|\le 2\omega_f(8\sqrt{d}N^{-2/d}L^{-2/d})&\le 2\lceil 8\sqrt{d}\rceil\omega_f(N^{-2/d}L^{-2/d})\\
&\le 18\sqrt{d}\,\omega_f(N^{-2/d}L^{-2/d}).
\end{split}
\end{equation*}

It remains to show the upper bound of $\phi$. By Equation \eqref{eq:phi2tUB} and  $\phi= \phi_2\circ\bmPhi_1+f(\bmzero)-\omega_f(\sqrt{d})$, it holds that
$ \|\phi\|_{L^\infty(\R^d)}\le |f(\bmzero)|+ \omega_f(\sqrt{d})$. 
Thus, we finish the proof.

%%%%%%%%%%%%%%%%%%%%%%%%%%%%%%%%%%%%%%%%%%%
\subsection{Proof of Proposition \ref{prop:stepFunc}}
\label{sec:proofProp1}

\begin{lemma}
	\label{lem:widthPower}
	For any $N_1,N_2\in \N^+$, given $N_1(N_2+1)+1$ samples $(x_i,y_i)\in \R^2$ with $x_0<x_1<\cdots<x_{N_1(N_2+1)}$ and  $y_i\ge 0$ for $i=0,1,\cdots,N_1(N_2+1)$,
	there exists $\phi\in \NN(\NNinput=1\NNspace\NNwidthvec=[2N_1,2N_2+1]\NNspace\NNoutput=1)$ satisfying the following conditions.
	\begin{enumerate}[(i)]
		\item $\phi(x_i)=y_i$ for $i=0,1,\cdots,N_1(N_2+1)$;
		\item $\phi$ is linear on each interval $[x_{i-1},x_{i}]$ for $i\notin \{(N_2+1)j:j=1,2,\cdots,N_1\} $.
	\end{enumerate}
\end{lemma}
In fact, Lemma \ref{lem:widthPower} is a part of Lemma $2.2$ in \cite{2019arXiv190210170S}. For the purpose of being self-contained, we present it as follows. 
\begin{lemma*}[Lemma 2.2 of \cite{2019arXiv190210170S}]
%	\label{lem:SquarePointsLemma}
	For any $m,n\in \N^+$, given any $m(n+1)+1$ samples $(x_i,y_i)\in \R^2$ with $x_0<x_1<x_2<\cdots<x_{m(n+1)}$ and $y_i\ge 0$ for $i=0,1,\cdots,m(n+1)$, there exists $\phi\in \NNF(\NNinput=1;\NNwidthvec=[2m,2n+1]\NNspace\NNoutput=1)$ satisfying the following conditions.
	\begin{enumerate}[(i)]
		\item $\phi(x_i)=y_i$ for $i=0,1,\cdots,m(n+1)$;
		\item $\phi$ is linear on each interval $[x_{i-1},x_{i}]$ for $i\notin \{(n+1)j:j=1,2,\cdots,m\}$;
		\item $\displaystyle \sup_{x\in[x_0,\,x_{m(n+1)}]} |\phi(x)| \le 3\max_{i\in \{0,1,\cdots,m(n+1)\}}y_i \prod_{k=1}^{n}\left(1+\tfrac{\max\{x_{j(n+1)+n}-x_{j(n+1)+k-1}:j=0,1,\cdots,m-1\} }
		{\min\{x_{j(n+1)+k}-x_{j(n+1)+k-1}:j=0,1,\cdots,m-1\} }\right)$.
	\end{enumerate}
\end{lemma*}
\begin{lemma}
	\label{lem:wideToDeep}
	Given any $N,L,d\in \N^+$, it holds that \begin{equation*}
	    \begin{split}
	        &\quad \, \NN(\NNinput=d\NNspace\NNwidthvec=[N,NL]\NNspace\NNoutput=1)\\
	        &\subseteq \NN(\NNinput=d\NNspace\NNwidth\le 2N+2\NNspace \NNdepth\le L+1\NNspace\NNoutput=1).
	    \end{split}
	\end{equation*}
\end{lemma}
%This section proves Propositions \ref{lem:widthPower} to \ref{prop:pointFitting} in Section \ref{sec:main}. 
\begin{proof}%[Proof of Proposition \ref{lem:wideToDeep}]
	The key idea to prove Proposition \ref{lem:wideToDeep} is to re-assemble $\calO(L)$ sub-FNNs %with width $\calO(N)$ and depth $\calO(1)$ 
	in the shallower FNN in the left of Figure \ref{fig:wideToDeepIdea} to form a deeper one with width $\calO(N)$ and depth $\calO(L)$ on the right of Figure \ref{fig:wideToDeepIdea}. 
	\begin{figure}[!htp]        
		\centering
		\begin{minipage}{0.95\textwidth}
			\centering
			\begin{subfigure}[c]{0.25\textwidth}
				\centering            \includegraphics[width=0.8\textwidth]{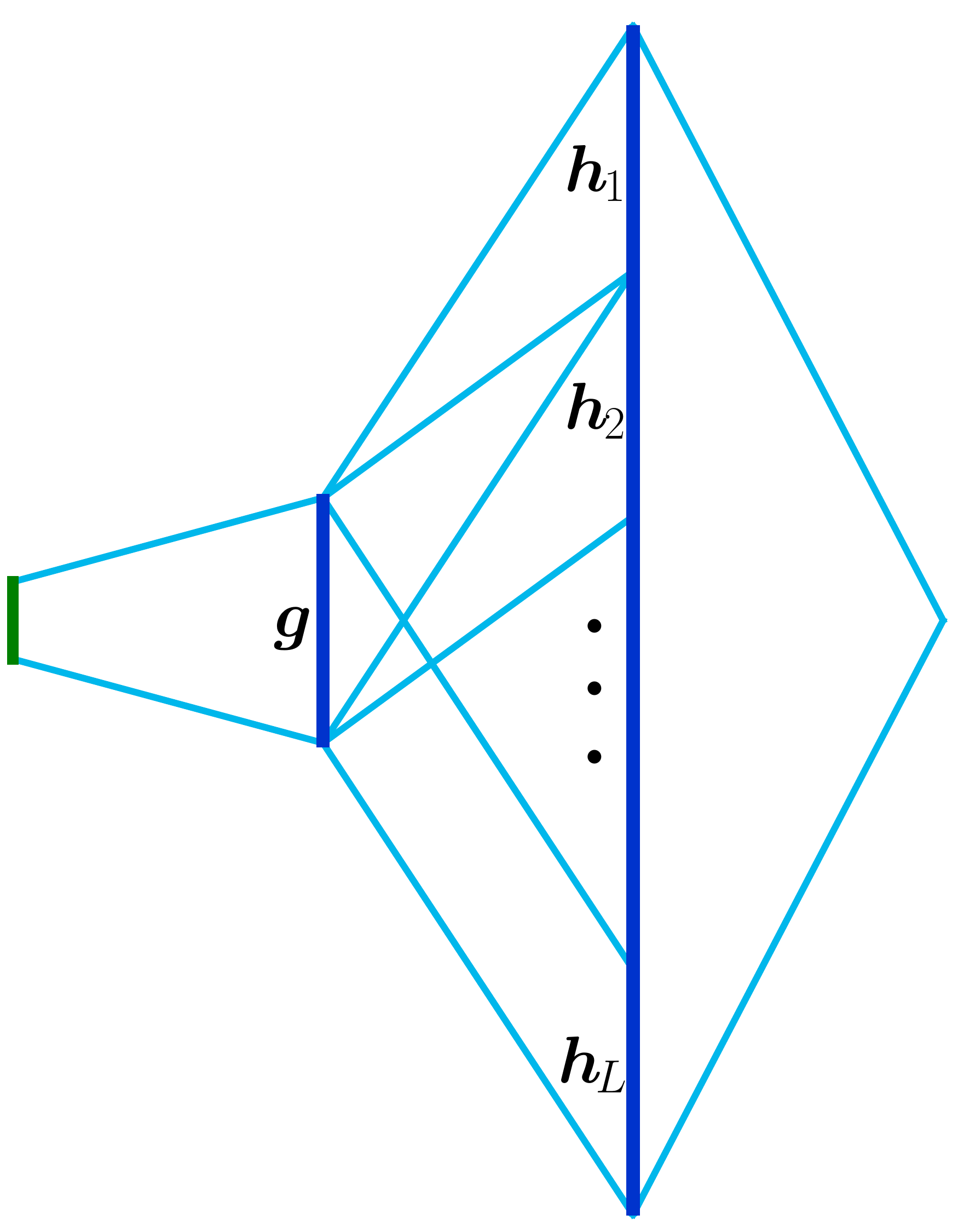}
				%\subcaption{}
			\end{subfigure}
			\begin{subfigure}[c]{0.04\textwidth}
				\centering
				{\vfill 
					$\Longrightarrow$
					\vfill}
			\end{subfigure}
			\begin{subfigure}[c]{0.5\textwidth}
				\centering            \includegraphics[width=0.9\textwidth]{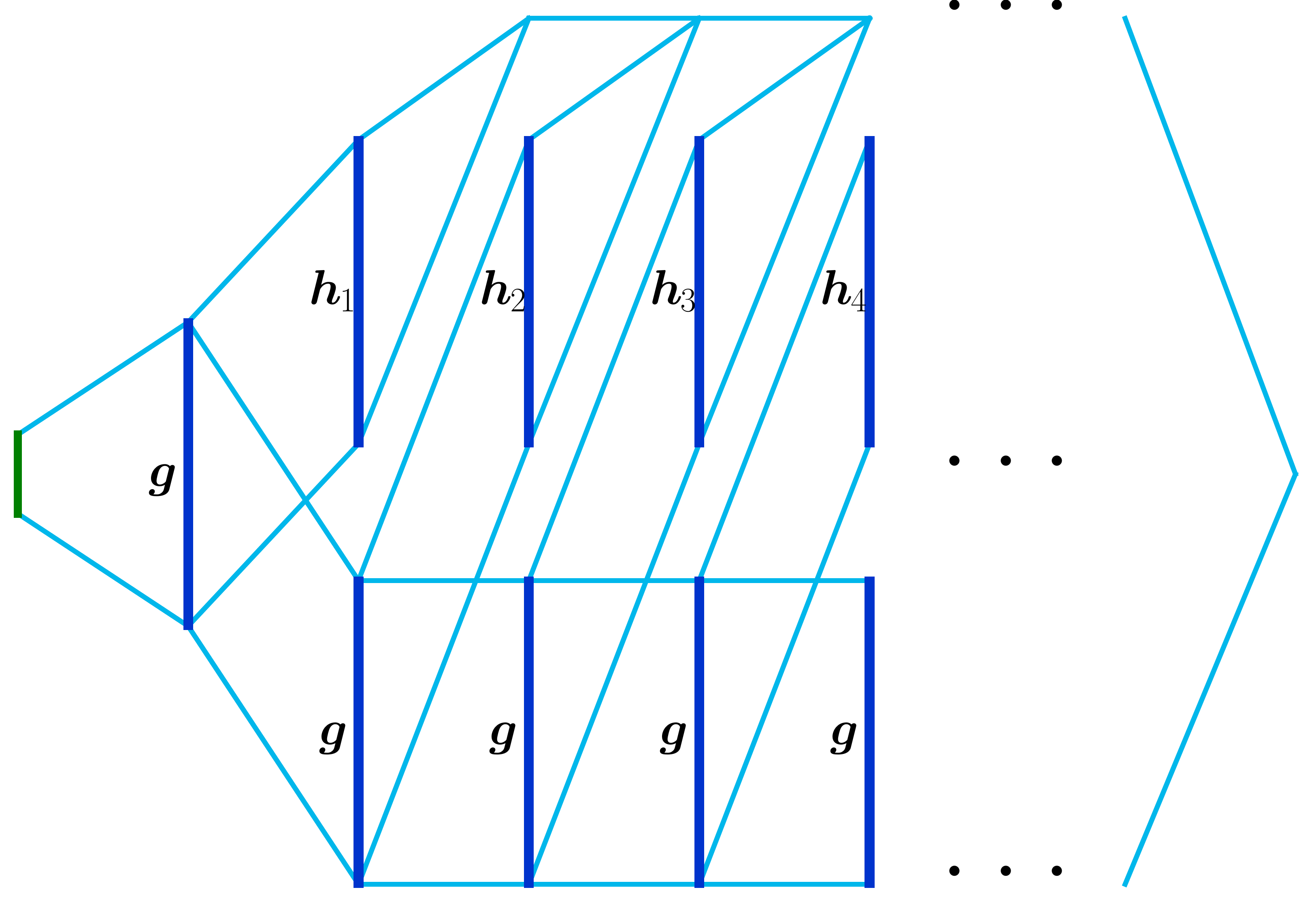}
				%\subcaption{}
			\end{subfigure}
		\end{minipage}
		\caption{An illustration of the main idea to prove Lemma \ref{lem:wideToDeep}.}
		\label{fig:wideToDeepIdea}
	\end{figure}
	%An illustration of the main idea to prove this proposition can be found in Figure \ref{fig:wideToDeep}. 

	For any $\phi\in \NNF(\NNinput=d\NNspace\NNwidthvec=[N,NL]\NNspace\NNoutput=1)$, $\phi$ can be implemented by a ReLU FNN described as 
	\begin{equation*}
	\begin{aligned}
	\bmx \myto{2.2}^{\bm{W}_0,\ \bm{b}_0}_\sigma \bm{g}\myto{2.2}^{\bm{W}_1,\ \bm{b}_1}_\sigma \bm{h} \myto{2.2}^{\bm{W}_2,\ \bm{b}_2} \phi(\bmx),
	\end{aligned}
	\end{equation*}
	where $\bm{g}$ and $\bm{h}$ are the output of the first hidden layer and the second hidden layer, respectively. Note that \begin{equation*}
	\bm{g}=\sigma(\bm{W}_0\cdot\bmx+\bm{b}_0),\quad \bm{h}=\sigma(\bm{W}_1\cdot\bm{g}+\bmb_1),\quad \text{ and }\quad \phi(\bmx)=\bm{W}_2\cdot\bm{h}+\bmb_2.
	\end{equation*}
	We can evenly divide $\bm{h}\in \R^{NL\times 1}$, $\bm{b}_1\in \R^{NL\times 1}$, $\bm{W}_1\in\R^{NL\times N}$, and $\bm{W}_2\in \R^{1\times NL}$ into $L$ parts as follows: 
	\[\bmh=\left[\begin{array}{c}
	\bm{h}_1\\ \bm{h}_2\\ \vdots \\ \bm{h}_L
	\end{array}\right],\quad 
	\bmb_1=\left[\begin{array}{c}
	\bm{b}_{1,1}\\ \bm{b}_{1,2}\\ \vdots \\ \bm{b}_{1,L}
	\end{array}\right],\quad \bmW_1=\left[\begin{array}{c}\bm{W}_{1,1} \\ \bm{W}_{1,2}\\\vdots\\ \bm{W}_{1,L}\end{array}\right],
	\] 
	and $\bmW_2=[\bm{W}_{2,1},\bm{W}_{2,2},\cdots,\bm{W}_{2,L}]$, where $\bm{h}_\ell\in \R^{N\times 1}$, $\bm{b}_{1,\ell}\in \R^{N\times 1}$, $\bm{W}_{1,\ell}\in \R^{N\times N}$, and $\bm{W}_{2,\ell}\in \R^{1\times N}$ for $\ell=1,2,\cdots,L$. 
	\red{	Then, for $\ell=1,2,\cdots,L$, we have
		\begin{equation}\label{eq:wideToDeep1}
			\bm{h}_\ell=\sigma(\bm{W}_{1,\ell}\cdot\bm{g}+\bmb_{1,\ell})\quad \text{ and }\quad \phi(\bmx)=\bm{W}_2\cdot\bm{h}+\bmb_2=\sum_{j=1}^L\bmW_{2,j}\cdot\bmh_{j}+\bmb_2.
	\end{equation}}
	
		Define 
	\begin{equation*}
	s_0\coloneqq 0 ,\quad \tn{and}\quad  s_\ell\coloneqq  \sum_{j=1}^{\ell}\bm{W}_{2,j}\cdot\bm{h}_j,\quad \tn{ for $\ell=1,2,\cdots,L$.}
	\end{equation*}
	Then $\phi(\bmx)=\bmW_2\cdot \bmh+\bmb_2=s_L+\bmb_2$ and 
	\begin{equation}
	\label{eq:wideToDeep2}
	s_\ell=s_{\ell-1}+\bm{W}_{2,\ell}\cdot\bm{h}_\ell,\quad \tn{ for $\ell=1,2,\cdots,L$.}
	\end{equation}
     %$\bm{h}=\left[\begin{array}{c}\bm{h}_1 \\\vdots\\\bm{h}_L \end{array}\right]\ge 0$since it is the output of the ReLU function. 
	Hence, it is easy to check that $\phi$ can also be implemented by the deep network  shown in Figure \ref{fig:wideToDeepNew}.
	\begin{figure}[!htp]
		\centering
		\includegraphics[width=0.95\textwidth]{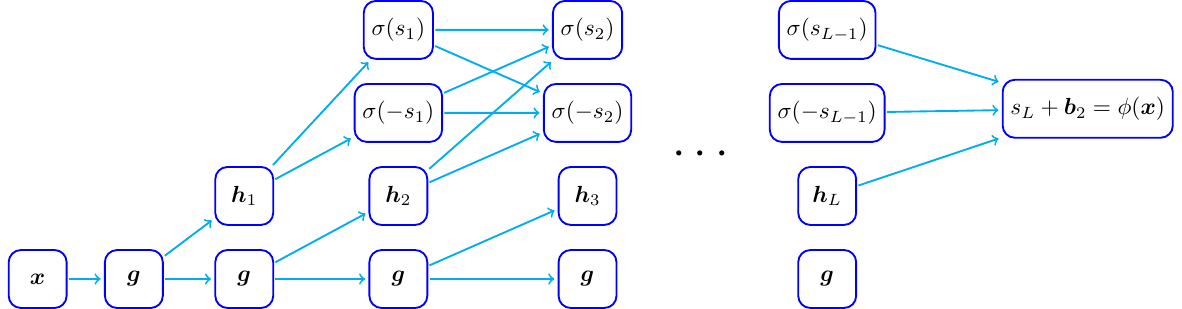}
		\caption{A illustration of the desired network based on Equation \eqref{eq:wideToDeep1} and \eqref{eq:wideToDeep2}, and  the fact $x=\sigma(x)-\sigma(-x)$ for any $x\in \R$. We omit the activation function ($\sigma$) if the input is non-negative.}
		\label{fig:wideToDeepNew}
	\end{figure}
	%	where ``$\to$" represents the composition of a ReLU activation function and an appropriate linear transform with weights and bias from the transforms in \eqref{eqn:p32} up to the change of their signs.
	It is clear that the network has the architecture of Figure \ref{fig:wideToDeepNew} is with width $2N+2$ and depth $L+1$. So, we finish the proof.
\end{proof}

With Lemma \ref{lem:widthPower} and \ref{lem:wideToDeep} in hand, we are ready to present the detailed proof of Proposition \ref{prop:stepFunc}. 
\begin{proof}[Proof of Proposition \ref{prop:stepFunc}]
	We divide the proof into two cases: $d=1$ and $d\ge 2$.
	
	\mycase{1}{$d=1$.}
	
	In this case,  $K=\lfloor N^{1/d}\rfloor^2 \lfloor L^{2/d}\rfloor=N^2L^2$. Denote $M=N^2L$ and consider the sample set \[
	\begin{split}\big\{(1,M-1),(2,0)\big\}\cup
	\big\{(\tfrac{m}{M},m):m=0,1,\cdots,M-1\big\}
	\cup \big\{(\tfrac{m+1}{M}-\delta,m):m=0,1,\cdots,M-2\big\}
	.\end{split}\]
	Its size is $2M+1=N\cdot\big((2NL-1)+1\big)+1$. By Lemma \ref{lem:widthPower} (set $N_1=N$ and $N_2=2NL-1$ therein), there exists $\phi_1\in \NNF(\NNwidthvec=[2N,2(2NL-1)+1])=\NNF(\NNwidthvec=[2N,4NL-1])$ such that
	\begin{itemize}
		\item $\phi_1(\tfrac{M-1}{M})=\phi_1(1)=M-1$ and $\phi_1(\tfrac{m}{M})=\phi_1(\tfrac{m+1}{M}-\delta)=m$ for $m=0,1,\cdots,M-2$;
		\item $\phi_1$ is linear on $[\tfrac{M-1}{M},1]$ and each interval $[\tfrac{m}{M},\tfrac{m+1}{M}-\delta]$ for $m=0,1,\cdots,M-2$. 
	\end{itemize}
	Then
	\begin{equation}
	\label{eq:returnmStepFunc}
	\phi_1(x)=m, \quad \tn{if} \ x\in [\tfrac{m}{M},\tfrac{m+1}{M}-\delta\cdot 1_{\{m\le M-2\}}],\quad \tn{for $m=0,1,\cdots,M-1$.}
	\end{equation}

	Now consider the another sample set \[\big\{(\tfrac{1}{M},L-1),(2,0)\big\}\cup\big\{(\tfrac{\ell}{ML},\ell):\ell=0,1,\cdots,L-1\big\}\cup \big\{(\tfrac{\ell+1}{ML}-\delta,\ell):\ell=0,1,\cdots,L-2\big\}.\] 
	Its size is $2L+1=1\cdot\big((2L-1)+1\big)+1$. By Lemma \ref{lem:widthPower} (set $N_1=1$ and $N_2=2L-1$ therein), there exists $\phi_2\in \NNF(\NNwidthvec=[2,2(2L-1)+1])=\NN(\NNwidthvec=[2,4L-1])$ such that
	\begin{itemize}
		\item $\phi_2(\tfrac{L-1}{ML})=\phi_2(\tfrac{1}{M})=L-1$ and $\phi_2(\tfrac{\ell}{ML})=\phi_2(\tfrac{\ell+1}{ML}-\delta)=\ell$ for $\ell=0,1,\cdots,L-2$;
		\item $\phi_2$ is linear on $[\tfrac{L-1}{ML},\tfrac{1}{M}]$ and each interval $[\tfrac{\ell}{ML},\tfrac{\ell+1}{ML}-\delta]$ for $\ell=0,1,\cdots,L-2$. 
	\end{itemize}
	It follows that, for  $m=0,1,\cdots,M-1$ and $\ell=0,1,\cdots,L-1$,
	\begin{equation}
	\label{eq:returnlStepFunc}
	\phi_2(x-\tfrac{m}{M})=\ell, \quad  \tn{for}\ x\in [\tfrac{mL+\ell}{ML},\tfrac{mL+\ell+1}{ML}-\delta\cdot 1_{\{\ell\le L-2\}}].
	\end{equation}

	The fact $K=ML$ implies each $k\in \{0,1,\cdots,K-1\}$ can be unique represented by $k=mL+\ell$ for $m=0,1,\cdots,M-1$ and $\ell=0,1,\cdots,L-1$. 
	Then the desired function $\phi$ can be implemented by a ReLU FNN shown in Figure \ref{fig:stepFunc}. Clearly, 
	\begin{equation*}
	%	\label{eq:extractIndex1D}
	\phi(x)=k,\quad \tn{if $x\in [\tfrac{k}{K},\tfrac{k}{K}-\delta\cdot 1_{\{k\le K-2\}}]$ for $k\in\{0,1,\cdots,K-1\}.$}
	\end{equation*}
	By Lemma \ref{lem:wideToDeep}, $\phi_1\in\NNF(\NNwidthvec=[2N,4NL-1])\subseteq\NNF(\NNwidth\le 4N+2\NNspace\NNdepth\le 2L+1) $ and $\phi_2 \in\NNF(\NNwidthvec=[2,4L-1])\subseteq\NNF(\NNwidth\le 6\NNspace\NNdepth\le 2L+1)$, implying $\phi\in \NNF(\NNwidth\le\max\{4N+2+1,6+1\}= 4N+3\NNspace\NNdepth\le (2L+1)+2+(2L+1)+1= 4L+5)$.
	So we finish the proof for the case $d=1$.
	\begin{figure}[!htp]
		\centering
		\includegraphics[width=0.88\textwidth]{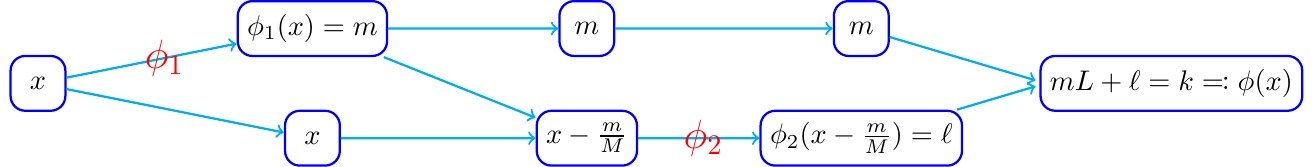}
		\caption[An illustration of the desired ReLU FNN architecture for proving Theorem \ref{prop:stepFunc}]{An illustration of the ReLU FNN implementing $\phi$ based on Equation \eqref{eq:returnmStepFunc} and \eqref{eq:returnlStepFunc} with $x\in [\tfrac{k}{K},\tfrac{k}{K}-\delta\cdot 1_{\{k\le K-2\}}]=[\tfrac{mL+\ell}{ML},\tfrac{mL+\ell+1}{ML}-\delta\cdot 1_{\{m\le M-2 \tn{ or }\ell\le L-2\}}]$, where $k=mL+\ell$ for $m=0,1,\cdots,M-1$ and $\ell=0,1,\cdots,L-1$.
		``\textcolor{red}{$\phi_1$}'' and ``\textcolor{red}{$\phi_2$}'' near ``\textcolor{cyan}{$\longrightarrow$}'' represent the respective ReLU FNN implementing itself. We omit the activation function ReLU if the input of a neuron is non-negative.
		}
		\label{fig:stepFunc}
	\end{figure}
	
	\mycase{2}{$d\ge2$.}
	
	Now we consider the case when $d\ge2$. Consider the sample set 
	\begin{equation*}
	\begin{split}
	\big\{(1,{K-1}),(2,0)\big\}\cup\big\{(\tfrac{k}{K},\red{k}):k=0,1,\cdots,K-1\big\}
	\cup \big\{(\tfrac{k+1}{K}-\delta,k):k=0,1,\cdots,K-2\big\},
	\end{split}
	\end{equation*} 
	whose size is $2K+1=\lfloor N^{1/d}\rfloor \big((2\lfloor N^{1/d}\rfloor \lfloor L^{2/d}\rfloor-1)+1\big)+1$.
	By Lemma \ref{lem:widthPower} (set $N_1=\lfloor N^{1/d}\rfloor$ and $N_2=2\lfloor N^{1/d}\rfloor \lfloor L^{2/d}\rfloor-1$ therein), there exists $\phi$ in 
	\begin{equation*}
	\begin{split}
	  &\quad\, \NNF(\NNwidthvec =[2\lfloor N^{1/d}\rfloor,2(2\lfloor N^{1/d}\rfloor \lfloor L^{2/d}\rfloor-1)+1])\\
	&\subseteq \NNF(\NNwidthvec =[2\lfloor N^{1/d}\rfloor,4\lfloor N^{1/d}\rfloor \lfloor L^{2/d}\rfloor-1])
	\end{split}
	\end{equation*} 
	such that
	\begin{itemize}
		\item $\phi(\tfrac{K-1}{K})=\phi(1)=K-1$, and $\phi(\tfrac{k}{K})=\phi(\tfrac{k+1}{K}-\delta)=k$ for $k=0,1,\cdots,K-2$;
		\item $\phi$ is linear on $[\tfrac{K-1}{K},1]$ and each interval $[\tfrac{k}{K},\tfrac{k+1}{K}-\delta]$ for $k=0,1,\cdots,K-2$. 
	\end{itemize}
	Then
	\begin{equation*}
	\phi(x)=k, \quad \tn{if} \ x\in [\tfrac{k}{K},\tfrac{k+1}{K}-\delta\cdot 1_{\{k\le K-2\}}]  \tn{ for $k=0,1,\cdots,K-1$.}
	\end{equation*}
	
	By Lemma \ref{lem:wideToDeep}, 
	\begin{equation*}
	\begin{split}
	\phi &\in \NNF(\NNwidthvec =[2\lfloor N^{1/d}\rfloor,4\lfloor N^{1/d}\rfloor \lfloor L^{2/d}\rfloor-1])\\
	&\subseteq
	\NNF(\NNwidth\le 4\lfloor N^{1/d}\rfloor+2\NNspace\NNdepth \le 2\lfloor L^{2/d}\rfloor+1)\\
	&\subseteq 
	\NNF(\NNwidth\le 4\lfloor N^{1/d}\rfloor+3\NNspace \NNdepth \le 4L+5).
	\end{split}
	\end{equation*} 
	which means we finish the proof for the case $d\ge 2$.
\end{proof}

%%%%%%%%%%%%%%%%%%%%%%%%%%%%%%%%%%%%%%%%%%%%%%%%%%%%
\subsection{Proof of Proposition \ref{prop:pointFitting}}
\label{sec:proofProp2}

The proof of Proposition \ref{prop:pointFitting} is based on the bit extraction technique in \cite{Bartlett98almostlinear,pmlr-v65-harvey17a}. In fact, we modify this technique to extract the sum of many bits rather than one bit and this modification can be summarized in Lemma \ref{lem:BitExtraction} and \ref{lem:BitExtractionMulti} below.

\begin{lemma}%[Bit Extraction]
	\label{lem:BitExtraction}
	For any $L\in \N^+$, there exists a function $\phi$ in
	\[\NNF(\NNinput=2\NNspace \NNwidth\le 7\NNspace\NNdepth\le 2L+1\NNspace\NNoutput=1)\]
	such that, for any $\theta_1,\theta_2,\cdots,\theta_L\in \{0,1\}$,  we have 
	\begin{equation*}
	\phi(\bin    0.\theta_1\theta_2\cdots \theta_L,\,\ell)=\sum_{j=1}^{\ell}\theta_j,\quad \tn{for $\ell=1,2,\cdots,L$.}
	\end{equation*}
\end{lemma}
\begin{proof}%[Proof of Lemma \ref{lem:BitExtraction}]
	Given $\theta_1,\theta_2,\cdots,\theta_L\in\{0,1\}$,  define  \[\xi_j\coloneqq \bin    0.\theta_j \theta_{j+1}\cdots \theta_L,\quad \tn{for $j=1,2,\cdots,L$}\]  and
	\begin{equation*}
	\mathcal{T}(x)\coloneqq \left\{{1,\ x\ge 0, \atop 0,\ x<0.}\right.
	\end{equation*}
	Then we have \[\theta_j=\mathcal{T}(\xi_j-1/2),\quad \tn{ for $j=1,2,\cdots,L$,} \] and 
	\[\xi_{j+1}=2\xi_j-\theta_j, \quad \tn{ for  $j=1,2,\cdots,L-1$.}\]
	I would like to point out that, by  above two iteration equations,  we can iteratively get $\xi_1,\theta_1,\xi_2,\theta_2,\cdots,\xi_L,\theta_L$ when  $\xi_1$ is given. Based on this iteration idea, the rest proof can be divided into three steps.
	
	\mystep{1}{Simplify the iteration equations.}
	
	Note that $\mathcal{T}(x)=\sigma(x/\delta+1)-\sigma(x/\delta)$ for any $x\notin (-\delta,0)$.
	By setting  $\delta=1/2-\sum_{j=2}^{L}2^{-j}=2^{-L}$, we have $\xi_j-1/2\notin (-\delta,0)$ for all $j$, implying	
	\begin{equation}
	\label{eq:thetaIteration}
	\begin{split}
	\theta_j=\mathcal{T}(\xi_j-1/2)
	&=\sigma\big((\xi_j-1/2)/\delta+1\big)-\sigma\big((\xi_j-1/2)/\delta\big) \\
	&=\sigma\big(\calL(\xi_j)+1\big)-\sigma\big(\calL(\xi_j)\big),   \end{split}
	\end{equation}
	for $j=1,2,\cdots,L$,  where $\calL$ is the linear map given by $\calL(x)=(x-1/2)/\delta$. It follows that, for $j=1,2,\cdots,L-1$,	
	\begin{equation}
	\label{eq:xiIteration}
	\begin{split}
	\xi_{j+1}=2\xi_j-\theta_j=2\xi_j-\sigma\big(\calL(\xi_j)+1\big)+\sigma\big(\calL(\xi_j)\big).
	\end{split}
	\end{equation}
	
	\mystep{2}{Design a ReLU FNN to output $\sum_{j=1}^{\ell}\theta_j$.}
	
	It is easy to design a ReLU FNN to  output $\theta_1,\theta_2,\cdots,\theta_L$ by Equation \eqref{eq:thetaIteration} and \eqref{eq:xiIteration} when using $\xi_1=\bin  0.\theta_1\theta_2\cdots\theta_L$ as the input. However, it is highly non-trivial to construct a ReLU FNN to output $\sum_{j=1}^{\ell}\theta_j$ with another input $\ell$, since many operations like multiplication and comparison are not allowed in designing ReLU FNNs.
	
	\begin{figure}[!htp]
		\centering
		\includegraphics[width=0.999\textwidth]{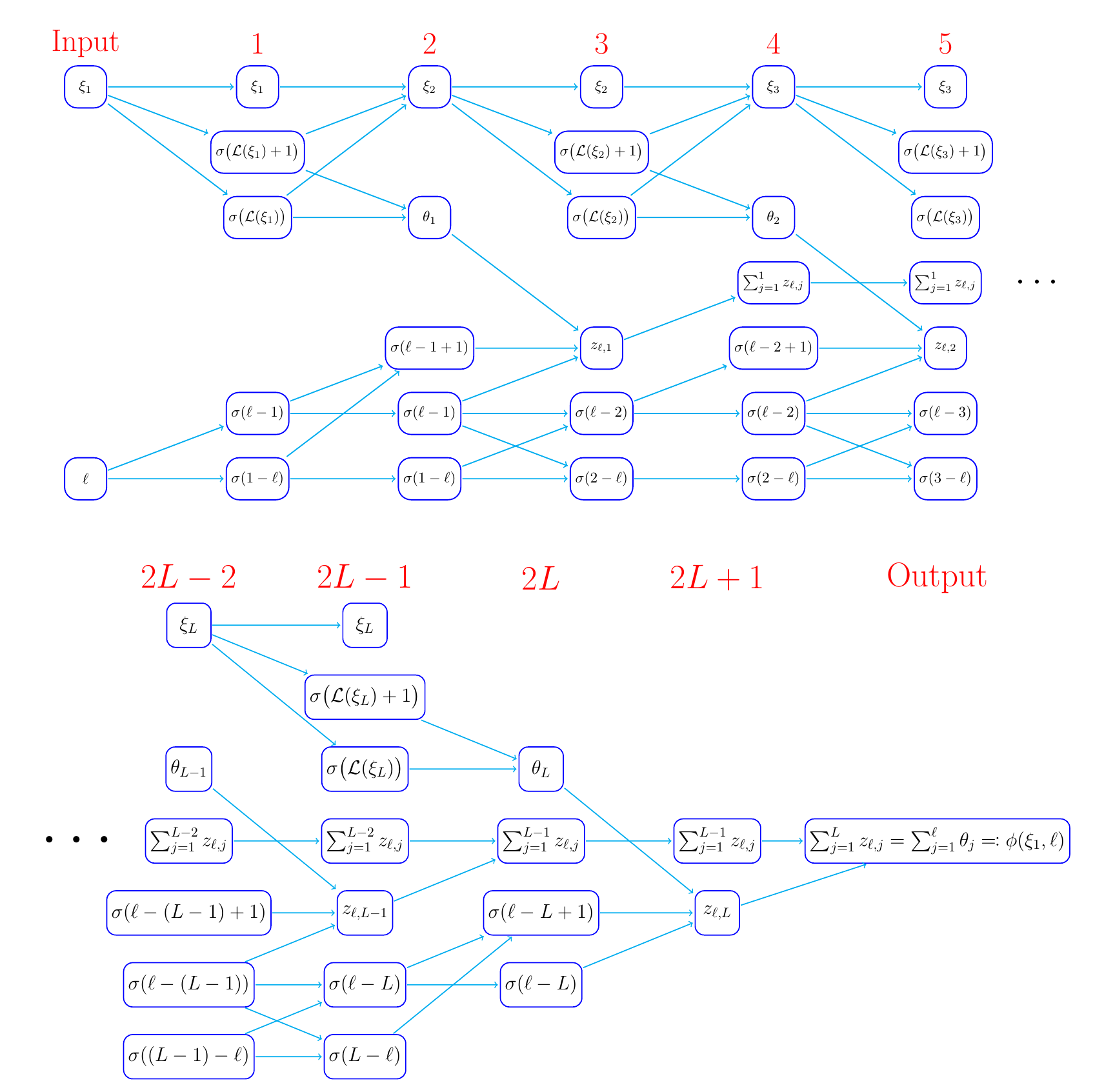}
		\caption{A illustration of the target ReLU FNN implementing $\phi$ to output $\sum_{j=1}^{L} z_{j,\ell}=\sum_{j=1}^{\ell}\theta_j=\phi(\xi_1,\ell)$ given the input $(\xi_1,\ell)=(\bin   0.\theta_1\theta_2\cdots \theta_L,\ell)$  for  $\ell\in \{1,2,\cdots,L\}$ and $\theta_1,\theta_2,\cdots,\theta_L\in\{0,1\}$.  The construction is mainly based on Equation \eqref{eq:thetaIteration},  \eqref{eq:xiIteration},  \eqref{eq:zDef}, and  \eqref{eq:outputSum}.
			The numbers above the architecture indicate the order of hidden layers.
			It builds a whole iteration step for every two layers. We output both $\sigma(\ell-j)$ and $\sigma(j-\ell)$ in the hidden layers for $j=1,2,\cdots,L$  because of the fact  $x=\sigma(x)-\sigma(-x)$ for any $x\in\R$. We omit the activation function ($\sigma$) if the input of a neuron is non-negative.  Note that all parameters of this network are essentially determined by Equation \eqref{eq:thetaIteration} and \eqref{eq:xiIteration}, which are valid no matter what $\theta_1,\theta_2,\cdots,\theta_L\in \{0,1\}$ are. Thus,  the desired function $\phi$ implemented by this network is independent of $\theta_1,\theta_2,\cdots,\theta_L\in \{0,1\}$.}
		\label{fig:bitExtration}
	\end{figure}
	
	Now let us establish a formula to represent $\sum_{j=1}^{\ell}\theta_j$ in a form of a ReLU FNN as follows:
	
	The fact that $x_1x_2=\sigma(x_1+x_2-1)$ for any  $x_1,x_2\in \{0,1\}$ implies
	\begin{equation*}
	%	\label{eq:outputSumxFormula}
	\begin{split}
	\sum_{j=1}^{\ell}\theta_j=\sum_{j=1}^{L}\theta_j\mathcal{T}(\ell-j)
	&=\sum_{j=1}^{L}\sigma\big(\theta_j+\calT(\ell-j)-1\big)\\
	&=\sum_{j=1}^{L}\sigma\big(\theta_j+\sigma(\ell-j+1)-\sigma(\ell-j)-1\big),
	\end{split}
	\end{equation*} 
	for $\ell=1,2,\cdots,L$,
	where the last equality comes from the fact $\mathcal{T}(n)=\sigma(n+1)-\sigma(n)$ for any integer $n$.
	
	To simplify the notations, we define
	\begin{equation}
	\label{eq:zDef}
	z_{\ell,j}\coloneqq  \sigma\big(\theta_j+\sigma(\ell-j+1)-\sigma(\ell-j)-1\big),
	\end{equation}
	for  $\ell=1,2,\cdots,L$ and $j=1,2,\cdots,L$.
	Then, 
	\begin{equation}
	\label{eq:outputSum}
	\sum_{j=1}^{\ell}\theta_j=\sum_{j=1}^Lz_{\ell,j},\quad \tn{for $\ell=1,2,\cdots,L$.}
	\end{equation}

	With  Equation 
	\eqref{eq:thetaIteration},  \eqref{eq:xiIteration}, \eqref{eq:zDef}, and \eqref{eq:outputSum}  in hand, it is easy to construct a function $\phi$ implemented by a ReLU FNN with the desired width and depth outputting $    \sum_{j=1}^{\ell}\theta_j=\sum_{j=1}^{L}z_{\ell,j}$ given the input $(\xi_1,\ell)=(\bin   0.\theta_1\theta_2\cdots\theta_L,\ell)$ for  $\ell\in \{1,2,\cdots,L\}$ and $\theta_1,\theta_2,\cdots,\theta_L\in\{0,1\}$. The details of construction 
are shown in Figure \ref{fig:bitExtration}. Clearly,   the  network in Figure \ref{fig:bitExtration} is with width $7$ and depth $2L+1$, which implies 
	\[\phi\in \NNF(\NNinput=2\NNspace\NNwidth\le 7\NNspace \NNdepth\le 2L+1\NNspace\NNoutput=1).\]
	So we finish the proof.
\end{proof}

Next, we introduce Lemma \ref{lem:BitExtractionMulti} as an advanced version of Lemma \ref{lem:BitExtraction}.
\begin{lemma}
	\label{lem:BitExtractionMulti}
		For any $N,L\in \N^+$, any $\theta_{m,\ell}\in \{0,1\}$ for $m=0,1,\cdots,M-1$ and  $\ell=0,1,\cdots,L-1$, where $M=N^2L$, there exists a function $\phi$ implemented by a ReLU FNN with width $4N+3$ and depth $3L+3$ such that
	\[\phi(m,\ell)=\sum_{j=0}^{\ell}\theta_{m,j}, \quad \tn{
	for $m=0,1,\cdots,M-1$ and $\ell=0,1,\cdots,L-1$.}\]
\end{lemma}
\begin{proof}
	Define \[y_m\coloneqq \bin    0.\theta_{m,0}\theta_{m,1}\cdots \theta_{m,L-1},\quad\tn{for $m=0,1,\cdots,M-1$.}\]  
	Consider the sample set $\{(m,y_m):m=0,1,\cdots,M\}$, whose cardinality is $M+1=N\big((NL-1)+1\big)+1$. By Lemma \ref{lem:widthPower} (set $N_1=N$ and $N_2=NL-1$ therein), there exists 
	\[\begin{split}\phi_1&\in\NNF(\NNinput=1\NNspace\NNwidthvec=[2N,2(NL-1)+1])\\
	&= \NN(\NNinput=1\NNspace\NNwidthvec=[2N,2NL-1])\end{split}\]  such that 
	\[\phi_1(m)=y_m,\quad   \tn{for  $m=0,1,\cdots,M-1$.} \]
	
	By Lemma \ref{lem:BitExtraction}, there exists \[\phi_2\in \NNF(\NNinput=2\NNspace\NNwidth\le 7\NNspace\NNdepth\le 2L+1)\] such that, for any  $\xi_1,\xi_2,\cdots,\xi_L\in \{0,1\}$, we have
	\[\phi_2(\bin   0.\xi_1\xi_2\cdots \xi_L,\,\ell)=\sum_{j=1}^{\ell}\xi_j,\quad \tn{for  $\ell=1,2,\cdots,L$.}\]
	It follows that,
	for any  $\xi_0,\xi_1,\cdots,\xi_{L-1}\in \{0,1\}$, we have
	\[\phi_2(\bin   0.\xi_0\xi_1\cdots \xi_{L-1},\,\ell+1)=\sum_{j=0}^{\ell}\xi_j,\quad \tn{for  $\ell=0,1,\cdots,L-1$.}\]
	
	Thus, for $m=0,1,\cdots,M-1$ and $\ell=0,1,\cdots,L-1$, we have
	\[\phi_2(\phi_1(m),\ell+1)=\phi_2(y_m,\ell+1)=\phi_2(0.\theta_{m,0}\theta_{m,1}\cdots \theta_{m,L-1},\, \ell+1)=\sum_{j=0}^{\ell}\theta_{m,j}.\]
	
	Hence, the desired function function $\phi$ can be implemented by the network shown in Figure \ref{fig:bitExtractionMulti}. By Lemma \ref{lem:wideToDeep}, $\phi_1\in \NNF(\NNwidthvec=[2N,2NL-1])\subseteq \NNF(\NNwidth\le 4N+2\NNspace\NNdepth\le L+1)$, implying the network in Figure \ref{fig:bitExtractionMulti} is with width $\max\{(4N+2)+1,7\}=4N+3$ and depth $(2L+1)+1+(L+1)=3L+3$.
	\begin{figure}
		\centering
		\includegraphics[width=0.85\textwidth]{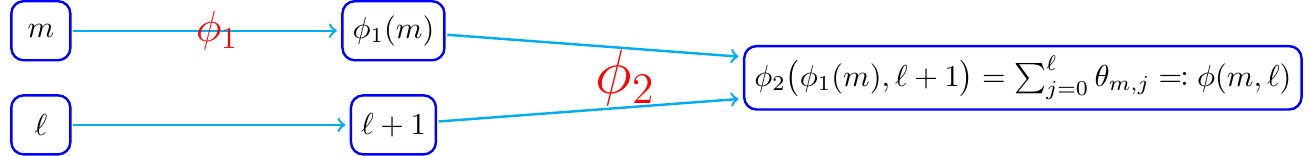}	
		\caption[A illustration of the desired network architecture for proving Theorem \ref{thm:bitExtraction}]{A illustration of the  network implementing the desired function $\phi$.  ``\textcolor{red}{$\phi_1$}'' and ``\textcolor{red}{$\phi_2$}'' near ``\textcolor{cyan}{$\longrightarrow$}'' represent the respective ReLU FNN implementing itself. We omit the activation function ReLU if the input of a neuron is non-negative.
		}
		\label{fig:bitExtractionMulti}
	\end{figure}
	So we finish the proof.
\end{proof}

\vspace{5pt}
Next, we apply Lemma \ref{lem:BitExtractionMulti} to prove Lemma \ref{lem:vcdimPoints1} below, which is a key intermediate conclusion to prove Proposition \ref{prop:pointFitting}.

\vspace{5pt}
\begin{lemma}
	\label{lem:vcdimPoints1}
	For any $\epsilon>0$, $L,N\in \N^+$, denote $M=N^2L$ and assume   $\{y_{m,\ell}\ge 0:m=0,1,\cdots,M-1 \tn{ and } \ell=0,1,\cdots,L-1\}$ is a sample set with   
	\[|y_{m,\ell}-y_{m,\ell-1}|\le \epsilon,\quad  \tn{for $m=0,1,\cdots,M-1\tn{ and } \ell=1,2,\cdots,L-1$.}\]
	Then there exists
	$\phi\in \NN(\NNinput=2\NNspace\NNwidth\le 12N+8\NNspace\NNdepth\le 3L+6)$ such that 
	\begin{enumerate}[(i)]
		\item $|\phi(m,\ell)-y_{m,\ell}|\le \epsilon$, for $m=0,1,\cdots,M-1 \tn{ and } \ell=0,1,\cdots,L-1$;
		
		\item $0\le \phi(x_1,x_2)\le  \max\{y_{m,\ell}:m=0,1,\cdots,M-1 \tn{ and } \ell=0,1,\cdots,L-1\}$, for any $x_1,x_2\in \R$.
	\end{enumerate}    
\end{lemma}

\begin{proof}%[Proof of Lemma \ref{lem:vcdimPoints1}]
	Define 
\begin{equation*}
a_{m,\ell}\coloneqq \lfloor y_{m,\ell}/\varepsilon\rfloor 
,\quad \tn{for $m=0,1,\cdots,M-1 \tn{ and } \ell=0,1,\cdots,L-1$.}
\end{equation*}
We will construct a function implemented by a ReLU FNN to map the index $(m,\ell)$ to $a_{m,\ell}\varepsilon$ for $m=0,1,\cdots,M-1 \tn{ and } \ell=0,1,\cdots,L-1$. 

Define $b_{m,0}\coloneqq 0$ and 
$b_{m,\ell}\coloneqq a_{m,\ell}-a_{m,\ell-1}$ for $m=0,1,\cdots,M-1 \tn{ and } \ell=1,\cdots,L-1$. 
Since $|y_{m,\ell}-y_{m,\ell-1}|\le \varepsilon$ for all $m$ and $\ell$, we have $b_{m,\ell}\in \{-1,0,1\}$. Hence, there exist $c_{m,\ell}$ and $d_{m,\ell}\in\{0,1\}$ such that 
$b_{m,\ell}=c_{m,\ell}-d_{m,\ell}$, which implies
\begin{equation*}
\begin{split}
a_{m,\ell}=a_{m,0}+\sum_{j=1}^{\ell}(a_{m,j}-a_{m,j-1})
&=a_{m,0}+\sum_{j=1}^{\ell}b_{m,j}
=a_{m,0}+\sum_{j=0}^{\ell}b_{m,j}\\
&=a_{m,0}+\sum_{j=0}^{\ell}c_{m,j}-\sum_{j=0}^{\ell}d_{m,j}.
\end{split}
\end{equation*}
for $m=0,1,\cdots,M-1 \tn{ and } \ell=1,\cdots,L-1$.

For the sample set $\{(m,a_{m,0}):m=0,1,\cdots,M-1\}\cup \{(M,0)\}$, whose size is $M+1=N\cdot\big((NL-1)+1\big)+1$, by Lemma \ref{lem:widthPower} (set $N_1=N$ and $N_2=NL-1$ therein), there exists $\psi_{1}\in \NNF(\NNwidthvec=[2N,2(NL-1)+1])=\NNF(\NNwidthvec=[2N,2NL-1])$ such that 
\[\psi_{1}(m)=a_{m,0},\quad \tn{for $m=0,1,\cdots,M-1$.}\]

By Lemma \ref{lem:BitExtractionMulti}, there exist $\psi_{2}, \psi_{3}\in \NNF(\NNwidth\le 4N+3\NNspace\NNdepth\le 3L+3)$ such that \[\psi_{2}(m,\ell)=\sum\limits_{j=0}^{\ell}c_{m,j}\quad\tn{and}\quad  \psi_{3}(m,\ell)=\sum\limits_{j=0}^{\ell}d_{m,j},\]
\tn{for $m=0,1,\cdots,M-1 \tn{ and } \ell=0,1,\cdots,L-1$.}
Hence, it holds that 
\begin{equation}
\label{eq:returnaml}
a_{m,\ell}=a_{m,0}+\sum_{j=0}^{\ell}c_{m,j}-\sum_{j=0}^{\ell}d_{m,j}=\psi_{1}(m)+\psi_{2}(m,\ell)-\psi_{3}(m,\ell),
\end{equation}
for $m=0,1,\cdots,M-1 \tn{ and } \ell=0,1,\cdots,L-1$.

Define 
\begin{equation*}
y_{\tn{max}}\coloneqq\max\{ y_{m,\ell}: m=0,1,\cdots,M-1 \tn{ and } \ell=0,1,\cdots,L-1\}.
\end{equation*}
Then the desired function can be implemented by two sub-networks shown in Figure \ref{fig:contFuncPointFitting}.

\begin{figure}[!htp]
	\centering
	\begin{subfigure}[c]{0.37\textwidth}
		\centering
		\includegraphics[width=0.98\textwidth]{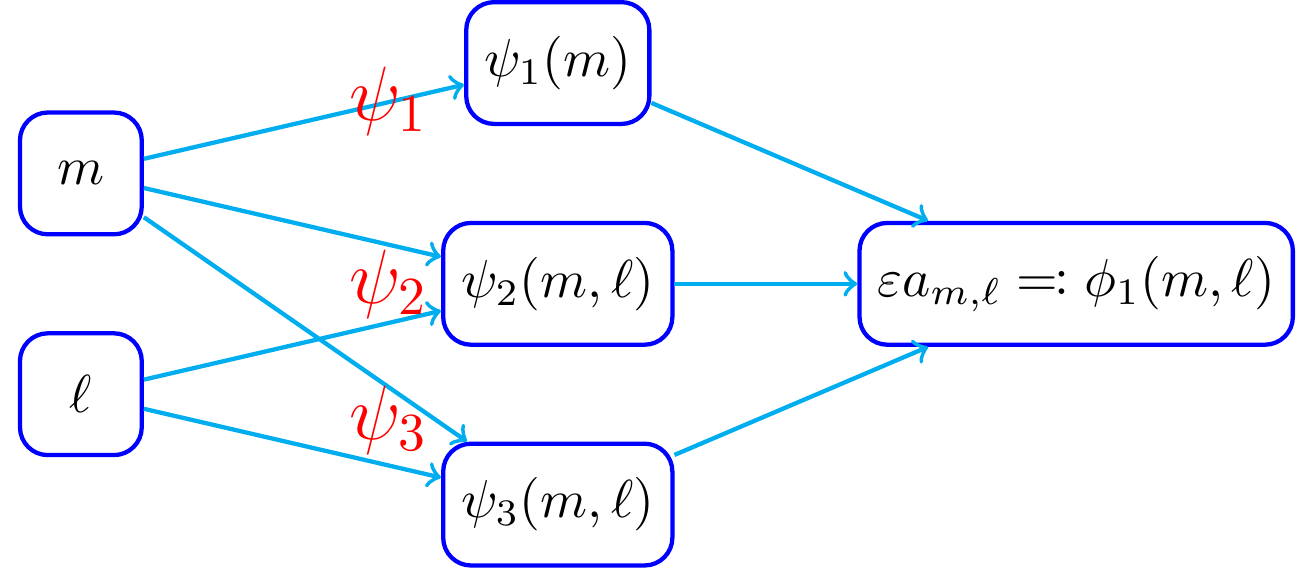}
		\subcaption{$\phi_1$}
	\end{subfigure}
	\begin{subfigure}[c]{0.53\textwidth}
		\centering
		\includegraphics[width=0.98\textwidth]{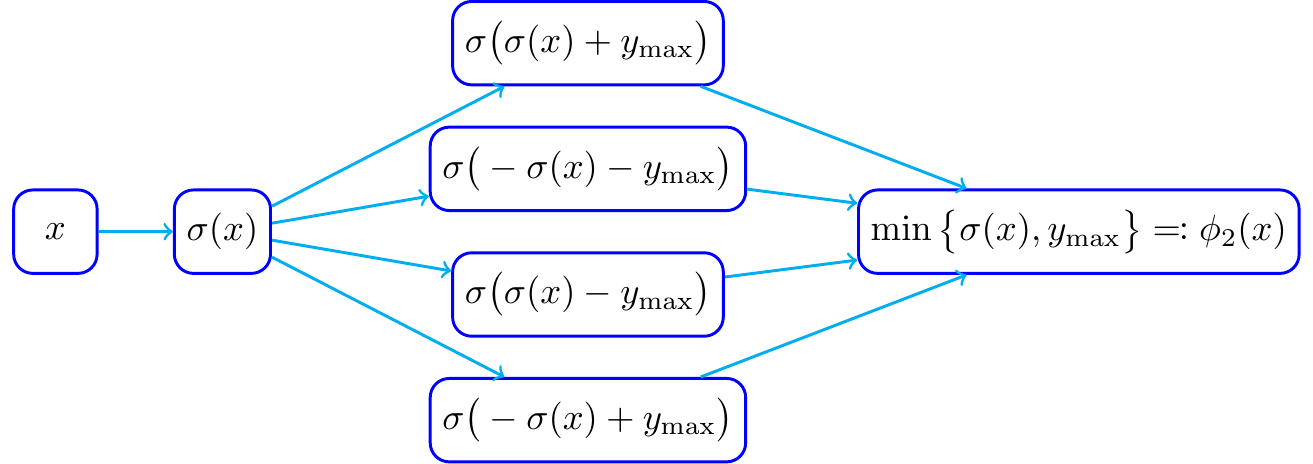}
		\subcaption{$\phi_2$}
	\end{subfigure}		
	\caption[Illustrations of two  sub-networks implementing the desired function $\phi=\phi_2\circ\phi_1$ based Equation \eqref{eq:returnaml}]{Illustrations of two  sub-networks implementing the desired function $\phi=\phi_2\circ\phi_1$ based Equation \eqref{eq:returnaml} and  the fact $\min\{x_1,x_2\}=\tfrac{x_1+x_2-|x_1-x_2|}{2}=\tfrac{\sigma(x_1+x_2)-\sigma(-x_1-x_2)-\sigma(x_1-x_2)-\sigma(-x_1+x_2)}{2}$. %The index $\beta\in \{0,1,\cdots,ML-1\}$ is unique represented by $\beta=mL+\ell$ for $m=0,1,\cdots,M-1$ and $\ell=0,1,\cdots,L-1$. 
		$y_{\tn{max}}$ is given by $\max\{ y_{m,\ell}: m=0,1,\cdots,M-1 \tn{ and } \ell=0,1,\cdots,L-1\}$. 
		``\textcolor{red}{$\psi_1$}'',``\textcolor{red}{$\psi_2$}'', and ``\textcolor{red}{$\psi_3$}'' near ``\textcolor{cyan}{$\longrightarrow$}'' represent the respective ReLU FNN implementing itself. We omit the activation function ReLU if the input of a neuron is non-negative.
		}
	\label{fig:contFuncPointFitting}
\end{figure}

By Lemma \ref{lem:wideToDeep},  $\psi_1\in \NNF(\NNinput=1\NNspace\NNwidthvec=[2N,2NL-1])\subseteq \NNF(\NNinput=1\NNspace\NNwidth\le 4N+2;\NNdepth\le L+1)$.
Note that $\psi_{2}, \psi_{3}\in \NNF(\NNwidth\le 4N+3\NNspace\NNdepth\le 3L+3)$.	
Thus,   $\phi_1\in\NNF(\NNwidth\le (4N+2)+2(4N+3)= 12N+8\NNspace\NNdepth\le (3L+3)+1=3L+4)$ as shown in Figure \ref{fig:contFuncPointFitting}.  And it is clear that $\phi_2\in \NNF(\NNwidth\le 4\NNspace\NNdepth\le 2)$, implying 
$\phi=\phi_2\circ\phi_1\in\NNF(\NNwidth\le 12N+8\NNspace\NNdepth\le (3L+4)+2=3L+6)$.

Clearly,  $0\le\phi(x_1,x_2)\le y_{\tn{max}}$ for any $x_1,x_2\in \R$, since $\phi(x_1,x_2)= \phi_2\circ\phi_1(x_1,x_2)=\max\{\sigma(\phi_1(x_1,x_2)),y_{\tn{max}}\}$. 

Note that $0\le \varepsilon a_{m,\ell}=\varepsilon \lfloor y_{m,\ell}/\varepsilon\rfloor \le y_{\tn{max}}$. Then we have $\phi(m,\ell)=\phi_2\circ\phi_1(m,\ell)=\phi_2(\varepsilon a_{m,\ell})=\max\{\sigma(\varepsilon a_{m,\ell}),y_{\tn{max}}\}=\varepsilon a_{m,\ell}$. Therefore,
\begin{equation*}
\begin{split}
|\phi(m,\ell)-y_{m,\ell}|=\left|a_{m,\ell}\varepsilon-y_{m,\ell}\right|=\big| \lfloor y_{m,\ell}/\varepsilon\rfloor\varepsilon-y_{m,\ell}\big|
\le \varepsilon,
\end{split}
\end{equation*}
for $m=0,1,\cdots,M-1$ and $\ell=0,1,\cdots,L-1$. Hence, we finish the proof.
\end{proof}

Finally, we apply Lemma \ref{lem:vcdimPoints1} to prove Proposition \ref{prop:pointFitting}.

\begin{proof}[Proof of Proposition \ref{prop:pointFitting}]
	Let $M=N^2L$, then we may assume $J=ML$ since we can set $y_{J-1}=y_{J}=y_{J+1}=\cdots=y_{ML-1}$ if $J<ML$. 

For the sample set
\[\{(mL,m):m=0,1,\cdots,M\}\cup \{(mL+L-1,m):m=0,1,\cdots,M-1\},\]
whose size is $2M+1=N\cdot\big((2NL-1)+1\big)+1$,
by Lemma \ref{lem:widthPower} (set $N_1=N$ and $N_2=NL-1$ therein), there exist $\phi_1\in \NNF(\NNinput=1\NNspace\NNwidthvec=[2N,2(2NL-1)+1])=\NNF(\NNinput=1\NNspace\NNwidthvec=[2N,4NL-1])$ such that
\begin{itemize}
	\item $\phi_1(ML)=M$ and $\phi_1(mL)=\phi_1(mL+L-1)=m$ for $m=0,1,\cdots,M-1$;
	\item $\phi_1$ is linear on each interval $[mL,mL+L-1]$ for $m=0,1,\cdots,M-1$.
\end{itemize}
It follows that 
\begin{equation}
\label{eq:phi1returnml}
\phi_1(j)=m,\quad \tn{and}\quad  j-L\phi_1(j)=\ell,\quad \tn{where $j=mL+\ell$},
\end{equation}
for $m=0,1,\cdots,M-1$ and $\ell=0,1,\cdots,L-1$.

Note that any number  $j$ in $\{0,1,\dots,J-1\}$ can be uniquely indexed as $j=mL+\ell$ for $m=0,1,\cdots,M-1$ and  $\ell=0,1,\cdots,L-1$.
So we can denote $y_j=y_{mL+\ell}$ as  $y_{m,\ell}$.
Then by Lemma \ref{lem:vcdimPoints1}, there exists $\phi_2\in \NNF(\NNwidth\le 12N+8\NNspace\NNdepth\le 3L+6)$ such that 
\begin{equation}
\label{eq:phi2Mniusyml}
|\phi_2(m,\ell)-y_{m,\ell}|\le \varepsilon,\quad \tn{for $m=0,1,\cdots,M-1 \tn{ and } \ell=0,1,\cdots,L-1$},
\end{equation}
and 
\begin{equation}
\label{eq:phi2UB}
0\le \phi_2(x_1,x_2)\le  y_{\tn{max}},\quad \tn{for any $x_1,x_2\in \R$,}
\end{equation}
where $y_{\tn{max}}\coloneqq\max\{y_{m,\ell}:m=0,1,\cdots,M-1 \tn{ and } \ell=0,1,\cdots,L-1\}=\max\{y_{j}:j=0,1,\cdots,ML-1\}$.

\begin{figure}[!htp]
	\centering
	\includegraphics[width=0.8\textwidth]{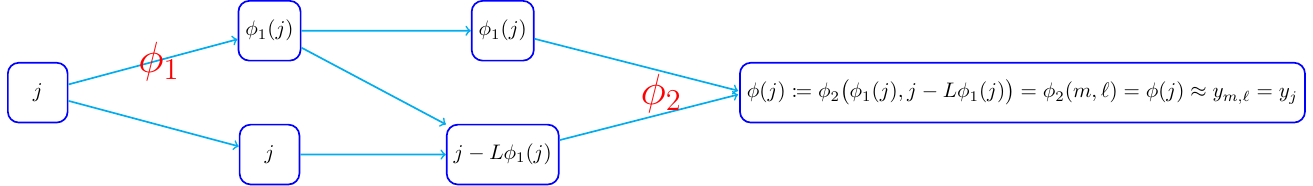}	
	\caption{A illustration of the ReLU FNN implementing the desired function $\phi$ based Equation \eqref{eq:phi1returnml}. The index $j\in \{0,1,\cdots,ML-1\}$ is unique represented by $j=mL+\ell$ for $m=0,1,\cdots,M-1$ and $\ell=0,1,\cdots,L-1$. 
		``\textcolor{red}{$\phi_1$}'' and ``\textcolor{red}{$\phi_2$}'' near ``\textcolor{cyan}{$\longrightarrow$}'' represent the respective ReLU FNN implementing itself. We omit the activation function ReLU if the input of a neuron is non-negative.}
	\label{fig:pointFittingThree}
\end{figure}

Note that $\phi_1\in \NNF(\NNinput=1\NNspace\NNwidthvec=[2N,4NL-1])\subseteq \NNF(\NNinput=1\NNspace\NNwidth\le 8N+2;\NNdepth\le L+1)$ by Lemma \ref{lem:wideToDeep} and $\phi_2\in \NNF(\NNwidth\le 12N+8\NNspace\NNdepth\le 3L+6)$. So $\phi\in \NNF(\NNwidth\le 12N+8\NNspace\NNdepth\le (L+1)+2+(3L+6)= 4L+9)$ as shown in Figure \ref{fig:pointFittingThree}.

Equation \eqref{eq:phi2UB} implies 
\begin{equation*}
0\le \phi(x)\le  y_{\tn{max}},\quad \tn{for any $x\in \R$,}
\end{equation*}
since $\phi$ is given by $\phi(x)=\phi_2\big(\phi_1(x),x-L\phi_1(x)\big)$.

Represent $j\in \{0,1,\cdots,ML-1\}$ via $j=mL+\ell$ for $m=0,1,\cdots,M-1$ and $\ell=0,1,\cdots,L-1$,  then we have, by Equation \eqref{eq:phi2Mniusyml},
\begin{equation*}
\begin{split}
|\phi(j)-y_j|=|\phi_2\big(\phi_1(j),j-L\phi_1(j)\big)-y_j|=|\phi_2(m,\ell)-y_{m,\ell}|\le \varepsilon.
\end{split}
\end{equation*}
 So we finish the proof.
\end{proof}

We would like to remark that the key idea in the proof of Proposition \ref{prop:pointFitting} is the bit extraction technique in Lemma \ref{lem:BitExtraction}, which allows us to store  $L$ bits in a binary number $\bin  0.\theta_1\theta_2\cdots \theta_L$ and extract each bit $\theta_i$. The extraction operator can be efficiently carried out via a deep ReLU neural network demonstrating the power of depth.

%%%%%%%%%%%%%%%%%%%%%%%%%%%%%%%%%%%%%%%%%%%%%
\section{Neural networks \sj{approximation and evaluation in practice}}
\label{sec:NNP}

This section is concerned with neural networks \sj{approximation and evaluation} in practice, e.g., approximating functions defined on irregular domains or domains with a low-dimensional structure, and neural network computation in parallel computing. \sj{In the practical training of FNNs, the approximation rate in this paper can only be observed if the global minimizers of neural network optimization can be identified. Since there is no existing optimization algorithm guaranteeing a global minimizer, it is challenging to observe the proved approximation rate currently. Developing optimization algorithms for global minimizers is another interesting research topic as a future work.}

\subsection{Approximation on irregular domain}

In this section, we consider approximating continuous functions defined on irregular domains by deep ReLU FNNs. The construction is through extending the target function to a cubic domain, applying Theorem \ref{thm:main}, and finally restricting the constructed FNN back to the irregular domain.

Given any uniformly continuous and real-valued function $f$ defined on a metric space $S$ with a metric $d_S(\cdot,\cdot)$,  we define the (optimal) modulus of continuity of $f$ on a subset $E\subseteq S$ as 
\begin{equation*}
\omega_f^{E}(r)\coloneqq \sup\{|f(\bmx_1)-f(\bmx_2)|:, d_S(\bmx_1,\bmx_2)\le r,\, \bmx_1,\bmx_2\in E\},\quad \tn{for any $r\ge 0$}.
\end{equation*}
For the purpose of consistency and simplicity, $\omega_f(\cdot)$ is short of $\omega_f^{[0,1]^d}(\cdot)$. 

First, let us present two lemmas for (approximately) extending (almost) continuous functions on $E$ to (almost) continuous functions on $S$. These lemmas are similar to the well-known results for extending Lipschitz or differentiable functions in \cite{mcshane1934,10.2307/1989708}. We generalize these results to a broader class of functions required in the proof of Theorem \ref{thm:mainIrregularDomain}.

\begin{lemma}[Approximate Extension of Almost-Continuous Functions]
    \label{lem:extensionContiuous2}
    Assume $S$ is a metric space with a metric $d_S(\cdot,\cdot)$ and $\omega:[0,\infty)\to [0,\infty)$ is an increasing function with  
    \begin{equation}
    \label{eq:ConvexOmega}
    \omega(r_1+r_2)\le \omega(r_1)+\omega(r_2),\quad \tn{for any $r_1,r_2\in [0,\infty)$}.
    \end{equation}
    Let   $f$ be a real-valued function defined on a subset $E\subseteq S$ and satisfy 
    \begin{equation}
    \label{eq:ContinuityE}
    |f(\bmx_1)-f(\bmx_2)|\le \omega (d_S(\bmx_1,\bmx_2)+\Delta), \quad \tn{for any $\bmx_1,\bmx_2\in E$},
    \end{equation} 
    where $\Delta$ is a positive constant independent of $f$.
    Then there exists a function $g$ defined on $S$ such that
    \begin{equation*}
    0\le f(\bmx)-g(\bmx)\le \omega (\Delta), \quad \tn{for any $\bmx\in E$}
    \end{equation*}
    and
    \begin{equation*}
    |g(\bmx_1)-g(\bmx_2)|\le \omega (d_S(\bmx_1,\bmx_2)), \quad \tn{for any $\bmx_1,\bmx_2\in S$.}
    \end{equation*}
\end{lemma}

\vspace{5pt}
In Lemma \ref{lem:extensionContiuous2}, $g$ is an approximate extension of $f$ defined on $E$ to a new domain $S$ with an approximation error $\omega(\Delta)$. In a special case when $\Delta=0$ and $\omega(0)=0$, $g$ is an exact extension of $f$. 
\begin{proof}[Proof of Lemma \ref{lem:extensionContiuous2}]
    Define 
    \begin{equation*}
    g(\bmx)\coloneqq \sup_{\bmz\in E} \big(f(\bmz)-\omega(d_S(\bmz,\bmx)+\Delta)\big).
    \end{equation*}
    By Equation \eqref{eq:ContinuityE}, we have $f(\bmx_1)-\omega (d_S(\bmx_1,\bmx_2)+\Delta) \le f(\bmx_2)$ for any $\bmx_1,\bmx_2\in E$. It holds that $g(\bmx)\le f(\bmx)$ for any $\bmx\in E$. Together with 
    \begin{equation*}
    g(\bmx)=\sup_{z\in E} \big(f(\bmz)-\omega(d_S(\bmz,\bmx)+\Delta)\big)\ge
    f(\bmx)-\omega(d_S(\bmx,\bmx)+\Delta) =f(\bmx)-\omega(\Delta) ,
    \end{equation*}
    \tn{for any $\bmx\in E$,}
    it follows that $0\le f(\bmx)-g(\bmx)\le\omega(\Delta)$ for any $\bmx\in E$. By Equation \eqref{eq:ConvexOmega} and the fact 
    \begin{equation*}
    \sup_{\bmz\in E}f_1(\bmz)-\sup_{\bmz\in E}f_2(\bmz)\le \sup_{\bmz\in E}\big(f_1(\bmz)-f_2(\bmz)\big),\quad \tn{for any functions $f_1,f_2$,}
    \end{equation*} we have
    \begin{equation*}
    \begin{split}
    g(\bmx_1)-g(\bmx_2)&=  \sup_{\bmz\in E} \big(f(\bmz)-\omega(d_S(\bmz,\bmx_1))\big)-\sup_{\bmz\in E} \big(f(\bmz)-\omega(d_S(\bmz,\bmx_2))\big)\\
    &\le
    \sup_{\bmz\in E} \big(\omega(d_S(\bmz,\bmx_1))-\omega(d_S(\bmz,\bmx_2))\big)\\
    &\le
    \sup_{\bmz\in E} \omega\big(d_S(\bmz,\bmx_1)-d_S(\bmz,\bmx_2)\big)\\
    &\le 
    \sup_{\bmz\in E} \omega(d_S(\bmx_1,\bmx_2))
    =\omega(d_S(\bmx_1,\bmx_2)),
    \end{split}
    \end{equation*}
    for any $\bmx_1,\bmx_2\in S$. Similarly, we have
    $
        g(\bmx_2)-g(\bmx_1)\le \omega(d_S(\bmx_1,\bmx_2)),
    $
    which implies 
    \begin{equation*}
        |g(\bmx_1)-g(\bmx_2)|\le \omega(d_S(\bmx_1,\bmx_2)).
    \end{equation*}
    So we finish the proof.
\end{proof}

Next, we introduce a lemma below for extending continuous functions defined on $E\subseteq S$ to continuous functions defined on $S$ preserving the modulus of continuity.

\begin{lemma}[Extension of Continuous Functions]
    \label{lem:extensionContiuous}
    Suppose $f$ is a uniformly continuous function defined on a subset $E\subseteq S$, where $S$ is a metric space with a metric $d_S(\cdot,\cdot)$, then there exists a uniformly continuous function $g$ on $S$ such that $f(\bmx)=g(\bmx)$ for $\bmx\in E$ and $\omega_f^E(r)=\omega_g^S(r)$ for any $r\ge0$.
\end{lemma}
\begin{proof}%[Proof of Lemma \ref{lem:extensionContiuous}]
 By the application of Lemma \ref{lem:extensionContiuous2} with $\omega(r)=\omega_f^E(r)$ for $r\ge 0$ and $\Delta=0$, we know that there exists $g:S\to \R$ such that
    \begin{equation*}
    0\le f(\bmx)-g(\bmx)\le \omega_f^E(\Delta)=0,\quad \tn{for any $\bmx\in E$,}
    \end{equation*}
    and 
    \begin{equation*}
    |g(\bmx_1)-g(\bmx_2)|\le \omega_f^{E}(d_S(\bmx_1,\bmx_2)),\quad \tn{for any $\bmx_1,\bmx_2\in S$.}
    \end{equation*}
    The equation above and the uniform continuity of $f$ imply that $g$ is uniformly continuous.  It also follows that 
    \begin{equation*}
     f(\bmx)=g(\bmx),\quad \tn{for any $\bmx\in E$,} \quad\tn{and}\quad \omega_g^S(r)\le\omega_f^E(r),\quad\tn{for any $r\ge 0$,}
    \end{equation*}
    since $\omega_g^S(\cdot)$ is the optimal modulus of continuity of $g$. Note that $\omega_f^E(\cdot)$ is the optimal moduls of continuity of $f$ and 
    \begin{equation*}
    |f(\bmx_1)-f(\bmx_2)|=|g(\bmx_1)-g(\bmx_2)|\le \omega_g^S(d_S(\bmx_1,\bmx_2)),\quad \tn{for any $\bmx_1,\bmx_2\in E$.}
    \end{equation*}
    Hence, $\omega_f^E(r)\leq \omega_g^S(r)$ for all $r\geq 0$, which implies $\omega_f^E(r)=\omega_g^S(r)$ since we have proved that $\omega_g^S(r)\leq \omega_f^E(r)$ for all $r\geq 0$.
%    Then we can extend $f$ by setting $f(\bmx)=g(\bmx)$ for $\bmx\in S$, which implies 
%    \begin{equation*}
%    \omega_f^E(r)\le \omega_f^S(r)=\omega_g^S(r)\le \omega_f^E(r),\quad \tn{for any $r\ge 0$.}
%    \end{equation*}
%    That is, $\omega_f^S(r)=\omega_f^E(r)$ for any $r\ge 0$. 
So we finish the proof.
\end{proof}

Now we are ready to introduce and prove the main theorem of this section, which extends Theorem \ref{thm:main} to an irregular domain as follows.

\begin{theorem}
    \label{thm:mainIrregularDomain}
    Let $f$ be a uniformly continuous function defined on  $E\subseteq [-R,R]^d$. For arbitrary $L\in \N^+$ and $N\in \N^+$, 
    there exists a function  $\phi$ implemented by a ReLU FNN with width $3^{d+3}\max\big\{d\lfloor N^{1/d}\rfloor,\, N+1\big\}$
    and depth $12L+14+2d$
    such that 
    \begin{equation*}
    \|f-\phi\|_{L^\infty(E)}\le 19\sqrt{d}\,\omega_f^E(2R N^{-2/d}L^{-2/d}).
    \end{equation*}
\end{theorem}
\begin{proof}%[Proof of Theorem \ref{thm:mainIrregularDomain}]
    By Lemma \ref{lem:extensionContiuous}, $f$ can be extended to $\R^d$ such that 
    \begin{equation*}
    \omega_f^{\R^d}(r)=\omega_f^E(r),\quad \tn{for any $r\ge0$.}
    \end{equation*}
    Define 
    \begin{equation*}
    \tildef(\bmx)\coloneqq f(2R\bmx-R),\quad \tn{for any $\bmx\in \R^d$.}
    \end{equation*}
    It follows that 
    \begin{equation}
    \label{eq:modulustfRd}
    \omega_\tildef^{\R^d}(r)=\omega_f^{\R^d}(2Rr)=\omega_f^E(2Rr),\quad \tn{for any $r\ge 0$.}
    \end{equation}
    By Theorem \ref{thm:main}, there exists a function $\tildephi$ implemented by a ReLU FNN with width $3^{d+3}\max\big\{d\lfloor N^{1/d}\rfloor,\, N+1\big\}$ and depth $12L+14+2d$ such that
    \begin{equation*}
\|\tildef-\tildephi\|_{L^\infty([0,1]^d)}\le 19\sqrt{d}\,\omega_\tildef^{[0,1]^d}(N^{-2/d}L^{-2/d})\le 19\sqrt{d}\,\omega_\tildef^{\R^d}(N^{-2/d}L^{-2/d}).
\end{equation*}    
 Define 
\begin{equation*}
\phi(\bmx)\coloneqq \tildephi(\tfrac{1}{2R}\bmx+\tfrac12),\quad \tn{for any $\bmx\in \R^d$.}
\end{equation*}
Then, by Equation \eqref{eq:modulustfRd}, for any $\bmx\in E\subseteq [-R,R]^d$, we have
\begin{equation*}
\begin{split}
|f(\bmx)-\phi(\bmx)|&=|\tildef(\tfrac{1}{2R}\bmx+\tfrac12)-\tildephi(\tfrac{1}{2R}\bmx+\tfrac12)|
\le \|\tildef-\tildephi\|_{L^\infty([0,1]^d)}\\
&\le 19\sqrt{d}\,\omega_\tildef^{\R^d}(N^{-2/d}L^{-2/d})
=19\sqrt{d}\,\omega_f^{E}(2RN^{-2/d}L^{-2/d}),
\end{split}
\end{equation*}
 which implies
\begin{equation*}
\|f-\phi\|_{L^\infty(E)}\le 19\sqrt{d}\,\omega_f^{E}(2RN^{-2/d}L^{-2/d}).
\end{equation*}
So we finish the proof.
\end{proof}

\subsection{Approximation in a neighborhood of a low-dimensional manifold}
In this section, we study neural network approximation of functions defined in a neighborhood of a low-dimensional manifold and prove Theorem \ref{thm:upDimReduction} in this setting. Let us first introduce Theorem \ref{thm:existenceA} which is required to prove Theorem \ref{thm:upDimReduction}. 

\begin{theorem}[Theorem $3.1$ of \cite{Baraniuk2009}]
    \label{thm:existenceA}
    Let $\calM$ be a compact $d_\calM$-dimensional Riemannian submanifold of $\R^d$ having condition number $1/\tau$, volume $V$, and geodesic covering regularity $\mathcal{R}$. Fix $\delta\in (0,1)$ and $\gamma\in (0,1)$. Let $\bmA=\sqrt{\tfrac{d}{d_\delta}}\Phi$, where $\Phi\in \R^{d_\delta\times d}$ is a random orthoprojector with 
    \begin{equation*}
    d_\delta=\calO\left(
    \tfrac{d_\calM \ln (dV\mathcal{R}\tau^{-1}\delta^{-1})\ln(1/\gamma)}
    {\delta^2}
    \right).
    \end{equation*}
    If $d_\delta\le d$, then with probability at least $1-\gamma$, the following statement holds: For every $\bmx_1,\bmx_2\in \calM$,
    \begin{equation*}
    (1-\delta)|\bmx_1-\bmx_2|\le|\bmA\bmx_1-\bmA\bmx_2|\le (1+\delta)|\bmx_1-\bmx_2|.
    \end{equation*}
\end{theorem}

Theorem \ref{thm:existenceA} shows the existence of a linear projector $\bmA\in\mathbb{R}^{d_\delta\times d}$ that maps a low-dimensional manifold in a high-dimensional space to a low-dimensional space nearly preserving distance. With this projection $\bmA$ available, we can prove Theorem \ref{thm:upDimReduction} via constructing a ReLU FNN defined in the low-dimensional space using Theorem \ref{thm:mainIrregularDomain} and hence the curse of dimensionality is lessened. The ideas of the proof are summarized in the following Table \ref{tab:idea}. 

In Table \ref{tab:idea} and the detailed proof later, we introduce a new notation $\Small(E)$ for any compact set $E\subseteq \R^d$ as the ``smallest'' element of $E$. Specifically, $\Small(E)$ is defined as the unique point in $\cap_{k=1}^{d}E_k$, where
\begin{equation*}
   E_k\coloneqq \{\bmx\in E_{k-1}:x_k=s_k\},\quad  s_k\coloneqq \inf \big\{x_k:[x_1,x_2,\cdots,x_d]^T\in E_{k-1}\big\}, \quad \tn{for $k=1,2,\cdots,d$},
\end{equation*}
and $E_0=E$. The compactness of $E$ ensures that $\cap_{k=1}^{d}E_k$ is in fact one point belonging to $E$. The introduction of $\Small(\cdot)$ uniquely formulates a low-dimensional function $\tilde{f}$ representing a high-dimensional function $f$ defined on $\calM_\epsilon$ by 
\begin{equation*}
    \tilde{f}(\bm{y}):=f(\bm{x}_{\bm{y}}),\quad \tn{where $\bmx_\bmy=\Small\big( \{\bmx\in\calM_\epsilon: \bmA \bmx=\bmy\}\big)$},\quad \tn{ for any $\bm{y}\in \bmA(\mathcal{M}_\epsilon)\subseteq \R^{d_\delta}$.}
\end{equation*}
As we shall see later, such a definition of $\tildef$ is reasonable because $\{\bmx\in \calM_\epsilon:\bmA\bmx=\bmy\}$ is contained in a small ball of radius $\calO(\epsilon)$ for any $\bmy\in \bmA(\calM_\epsilon)$. There are many other alternative ways to define $\Small(\cdot)$ as long as the definition ensures that $\Small(E)$ contains only one element. For example, $\Small(E)$ can be defined as any arbitrary point in $E$. For another example, $\bmy\in \bmA(\calM)$ cannot guarantee $\bmx_\bmy=\Small\big(
\{\bmx\in \calM_\epsilon:\bmA\bmx=\bmy\}\big)\in \calM$ in the current definition, but in practice we can choose $\Small\big(
\{\bmx\in \calM:\bmA\bmx=\bmy\}\big)$ as $\bmx_\bmy$ to ensure that $\bmx_\bmy\in \calM$, which might be beneficial for potential applications. 

\begin{table}[!htp]
    \caption{Main steps of the proof of Theorem \ref{thm:upDimReduction}. Step $1$: dimension reduction via the nearly isometric projection operator $\bmA$ provided by Theorem \ref{thm:existenceA} to obtain an ``equivalent" function $\tilde{f}$ of $f$ in a low-dimensional domain using $\bm{x}_{\bm{y}}=\Small\left( \{\bm{x}\in\mathcal{M}_\epsilon: \bmA \bm{x}=\bm{y}\}\right)$. Step $2$: construct a ReLU FNN to implement $\tilde{\phi}\approx \tilde{f}$ by Theorem \ref{thm:mainIrregularDomain}. Step $3$: define a ReLU FNN to implement $\phi$ in the original high-dimensional domain via the projection $\bmA$. Step $4$: verify that the approximation error of $\phi\approx f$ satisfies our requirement. } 
    \label{tab:idea}
    \centering  
    \resizebox{0.98\textwidth}{!}{\begin{tabular}{ccc} 
        \toprule
        $f(\bm{x})$ for $\bm{x}\in\mathcal{M}_\epsilon\subseteq [0,1]^d$ &\qquad $\stackrel{\mathclap{\normalfont\mbox{\text{Step }4}}}{
        \scalebox{3}[1.052]{$\approx$}}$  \qquad\quad& $\phi(\bm{x}):=\tilde{\phi}(\bmA\bm{x})$ for $\bm{x}\in\mathcal{M}_\epsilon\subseteq [0,1]^d$ \\
       \vspace{-4pt}\\
        \qquad\qquad Step $1$  $ \displaystyle\left\Updownarrow\vphantom{\int_A^B}\right. 
        \bmx_\bmy=\Small\big( \{\bmx\in\calM_\epsilon: \bmA \bmx=\bmy\}\big)$ & \qquad  \qquad& Step $3$  $ \displaystyle\left\Updownarrow\vphantom{\int_A^B}\right. \bm{y}=\bmA\bm{x}$\\
       \vspace{-8pt} \\
        $\tilde{f}(\bm{y}):=f(\bm{x}_{\bm{y}})$ for $\bm{y}\in \bmA(\mathcal{M}_\epsilon)\subseteq \R^{d_\delta}$ &\qquad $\stackrel{\mathclap{\normalfont\mbox{\text{Step }2}}}{
        \scalebox{3}[1.052]{$\approx$}}$ \qquad\quad&  $\tilde{\phi}(\bm{y})$ for $\bm{y}\in \bmA(\mathcal{M}_\epsilon)\subseteq \R^{d_\delta}$  \\
        \vspace{-8pt}\\
      \bottomrule% \bottomrule[1.2pt]  指定宽度
    \end{tabular} 
}%%%%%%%%%%%%%%5
\end{table}

Now we are ready to prove Theorem \ref{thm:upDimReduction}.

\begin{proof}[Proof of Theorem \ref{thm:upDimReduction}]

By Theorem \ref{thm:existenceA}, there exists a matrix $\bmA\in\R^{d_\delta\times d}$ such that
    \begin{equation}
    \label{eq:orthoA}
    \bmA \bmA^T=\tfrac{d}{d_\delta}\bm{I}_{d_\delta},
    \end{equation}
where $\bm{I}_{d_\delta}$ is an identity matrix of size $d_\delta\times d_\delta$, and
    \begin{equation}
    \label{eq:distortionM}
    (1-\delta)|\bmx_1-\bmx_2|\le |\bmA \bmx_1-\bmA \bmx_2|\le 
    (1+\delta)|\bmx_1-\bmx_2|,\quad \tn{for any } \bmx_1,\bmx_2\in \calM.
    \end{equation}

Given any $\bmy\in \bmA(\calM_\epsilon)$, then $\{\bmx\in \calM_\epsilon:\bmA\bmx=\bmy\}$ is a nonzero compact set. Let $\bmx_\bmy=\Small\big(
\{\bmx\in \calM_\epsilon:\bmA\bmx=\bmy\}\big)$, 
then we define $\tildef$  on $\bmA(\calM_\epsilon)$ as
$\tildef(\bmy)=f(\bmx_\bmy).$

For any $\bmy_1,\bmy_2\in \bmA(\calM_\epsilon)$, let $\bmx_i=\Small\big(
\{\bmx\in \calM_\epsilon:\bmA\bmx=\bmy_i\}\big)$, then $\bmx_i\in \calM_\epsilon$ for $i=1,2$. By the definition of $\mathcal{M}_\epsilon$, there exist $\txx_1,\txx_2\in \calM$ such that  $|\txx_i-\bmx_i|\le \epsilon$ for $i=1,2$. It follows that
\begin{equation*}
|\tildef(\bmy_1)-\tildef(\bmy_2)|=|f(\bmx_1)-f(\bmx_2)|\le \omega_f(|\bmx_1-\bmx_2|)\le \omega_f(|\txx_1-\txx_2|+2\epsilon)\le \omega_f(\tfrac{1}{1-\delta}|\bmA\txx_1-\bmA\txx_2|+2\epsilon),
\end{equation*}
where the last inequality comes from Equation \eqref{eq:distortionM}. By the triangular inequality, we have
\begin{equation*}
\begin{split}
|\tildef(\bmy_1)-\tildef(\bmy_2)|&\le \omega_f(\tfrac{1}{1-\delta}|\bmA\bmx_1-\bmA\bmx_2|+\tfrac{1}{1-\delta}|\bmA\bmx_1-\bmA\txx_1|+\tfrac{1}{1-\delta}|\bmA\bmx_2-\bmA\txx_2|+2\epsilon)\\
&\le \omega_f(\tfrac{1}{1-\delta}|\bmA\bmx_1-\bmA\bmx_2|+\tfrac{2\epsilon}{1-\delta}\sqrt{\tfrac{d}{d_\delta}}+2\epsilon)\\
&\le \omega_f(\tfrac{1}{1-\delta}|\bmy_1-\bmy_2|+\tfrac{2\epsilon}{1-\delta}\sqrt{\tfrac{d}{d_\delta}}+2\epsilon).
\end{split}
\end{equation*}

Set $\omega(r)=\omega_f(\tfrac{1}{1-\delta}r)$ for any $r\ge 0$ and $\Delta=2\epsilon\sqrt{\tfrac{d}{d_\delta}}+2\epsilon(1-\delta)$, then 
\begin{equation*}
\begin{split}
|\tildef(\bmy_1)-\tildef(\bmy_2)|\le  \omega(|\bmy_1-\bmy_2|+\Delta),\quad \tn{for any $\bmy_1,\bmy_2\in \bmA(\calM_\epsilon)\subseteq \R^{d_\delta}$}.
\end{split}
\end{equation*}
By Lemma \ref{lem:extensionContiuous2}, there exists $\tildeg$ defined on $\R^{d_\delta}$ such that 
\begin{equation}
\label{eq:difftfandtg}
|\tildeg(\bmy)-\tildef(\bmy)|\le \omega(\Delta)=\omega_f\big(\tfrac{2\epsilon}{1-\delta}\sqrt{\tfrac{d}{d_\delta}}+2\epsilon\big),\quad \tn{for any $\bmy\in \bmA(\calM_\epsilon)$,}
\end{equation}
and 
\begin{equation*}
|\tildeg(\bmy_1)-\tildeg(\bmy_2)|\le \omega(|\bmy_1-\bmy_2|)=\omega_f(\tfrac{1}{1-\delta}|\bmy_1-\bmy_2|),\quad \tn{for any $\bmy_1,\bmy_2\in \R^{d_\delta}$}.
\end{equation*}
It follows that 
\begin{equation}
\label{eq:omegaglef}
\omega_\tildeg^{\R^{d_\delta}}(r)\le \omega_f(\tfrac{r}{1-\delta}),\quad \tn{ for any $r\ge 0$. }
\end{equation}

By Equation \eqref{eq:orthoA} and the definition of $\calM_\epsilon$ in Equation \eqref{eqn:Me}, it is easy to check that
\begin{equation*}
\bmA (\calM_\epsilon)\subseteq \bmA ([0,1]^d)\subseteq [-\sqrt{\tfrac{d}{d_\delta}},\sqrt{\tfrac{d}{d_\delta}}]^{d_\delta}.
\end{equation*}

By the application of Theorem \ref{thm:mainIrregularDomain} with  $E=[-\sqrt{\tfrac{d}{d_\delta}},\sqrt{\tfrac{d}{d_\delta}}]^{d_\delta}$, there exists a function $\tildephi$ implemented by a ReLU FNN with width $3^{d_\delta+3}\max\big\{d_\delta\lfloor N^{1/d_\delta}\rfloor,\, N+1\big\}$ and depth $12L+14+2d_\delta$ such that
\begin{equation}
\label{eq:tgAndtphiErrorUB}
\|\tildeg-\tildephi\|_{L^\infty(E)}\le 19\sqrt{d}\,\omega_\tildeg^E(2\sqrt{\tfrac{d}{d_\delta}}N^{-2/d_\delta}L^{-2/d_\delta}).
\end{equation}

Define $\phi\coloneqq \tildephi\circ \bmA$, i.e., $\phi(\bmx)\coloneqq \tildephi(\bmA\bmx)$ for any $\bmx\in \R^d$. Then $\phi$ is also a ReLU FNN with width $3^{d_\delta+3}\max\big\{d_\delta\lfloor N^{1/d_\delta}\rfloor,\, N+1\big\}$ and depth $12L+14+2d_\delta$. 

For any $\bmx\in \calM_\epsilon$, set $\bmy=\bmA\bmx$ and $\bmx_\bmy=\Small\big(
\{\bmz\in \R^d:\bmA\bmz=\bmy\}\big)$, there exist $\txx,\txx_\bmy\in \calM$ such that
$|\txx-\bmx|\le \epsilon$ and $|\txx_\bmy-\bmx_\bmy|\le \epsilon$.
It follows from Equation \eqref{eq:distortionM} that
\begin{equation}
\label{eq:Distancex1x2}
\begin{split}
    |\bmx-\bmx_\bmy|&\le |\txx-\txx_\bmy|+2\epsilon
    \le \tfrac{1}{1-\delta}|\bmA\txx-\bmA\txx_\bmy|+2\epsilon\\
    &\le \tfrac{1}{1-\delta}\big(|\bmA\txx-\bmA\bmx|+|\bmA\bmx-\bmA\bmx_\bmy|+|\bmA\bmx_\bmy-\bmA\txx_\bmy|\big)+2\epsilon\\
    &= \tfrac{1}{1-\delta}\big(|\bmA\txx-\bmA\bmx|+|\bmA\bmx_\bmy-\bmA\txx_\bmy|\big)+2\epsilon
     \le \tfrac{2\epsilon}{1-\delta}\sqrt{\tfrac{d}{d_\delta}}+2\epsilon.
\end{split}
\end{equation}
In fact, the above equation implies that $\{\bmx\in \calM_\epsilon:\bmA\bmx=\bmy\}$ is contained in a small ball of radius $\calO(\epsilon)$ for $\bmy\in \bmA(\calM_\epsilon)$ as we mentioned previously.

Together with Equation \eqref{eq:difftfandtg}, \eqref{eq:omegaglef}, \eqref{eq:tgAndtphiErrorUB}, and \eqref{eq:Distancex1x2}, we have, for any $\bmx\in \calM_\varepsilon$,
\begin{equation*}
\begin{split}
|f(\bmx)-\phi(\bmx)|&\le |f(\bmx)-f(\bmx_\bmy)|+|f(\bmx_\bmy)-\phi(\bmx)|\\
&\le \omega_f\big(\tfrac{2\epsilon}{1-\delta}\sqrt{\tfrac{d}{d_\delta}}+2\epsilon\big)+ |\tildef(\bmy)-\tildephi(\bmy)|\\
&\le\omega_f\big(\tfrac{2\epsilon}{1-\delta}\sqrt{\tfrac{d}{d_\delta}}+2\epsilon\big)+ |\tildef(\bmy)-\tildeg(\bmy)|+|\tildeg(\bmy)-\tildephi(\bmy)|\\
&\le \omega_f\big(\tfrac{2\epsilon}{1-\delta}\sqrt{\tfrac{d}{d_\delta}}+2\epsilon\big)+\omega_f\big(\tfrac{2\epsilon}{1-\delta}\sqrt{\tfrac{d}{d_\delta}}+2\epsilon\big)+19\sqrt{d}\,\omega_\tildeg^E(2\sqrt{\tfrac{d}{d_\delta}}N^{-2/d_\delta}L^{-2/d_\delta})\\
&\le 2\omega_f\big(\tfrac{2\epsilon}{1-\delta}\sqrt{\tfrac{d}{d_\delta}}+2\epsilon\big)+19\sqrt{d}\,\omega_f(\tfrac{2\sqrt{d}}{(1-\delta)\sqrt{d_\delta}}N^{-2/d_\delta}L^{-2/d_\delta}).\\
\end{split}
\end{equation*}
Hence, we have finished the proof of this theorem.
\end{proof}

It is worth emphasizing that the approximation error \[\calO\Big(\omega_f\big(\calO(\epsilon)\big)+\omega_f\big(\calO(N^{-2/d_\delta}L^{-2/d_\delta})\big)\Big)\]
 in Theorem \ref{thm:upDimReduction} is equal to $\calO\Big(\omega_f\big(\calO(N^{-2/d_\delta}L^{-2/d_\delta})\big)\Big)$ when $\epsilon=\calO(N^{-2/d_\delta}L^{-2/d_\delta})$. 

The application of Theorem \ref{thm:existenceA} and the proof of Theorem \ref{thm:upDimReduction} in fact inspire an efficient two-step algorithm for high-dimensional learning problems: in the first step, high-dimensional data are projected to a low-dimensional space via a random projection; in the second step, a deep learning algorithm is applied to learn from the low-dimensional data. By Theorem \ref{thm:existenceA} and \ref{thm:upDimReduction}, the deep learning algorithm in the low-dimensional space can still provide good results with a high probability.

\subsection{Optimal ReLU FNN structure in parallel computing}
\label{sec:BestFNN}
In this section, we show how to select the best ReLU FNN to approximate functions in $\holder{\lambda}{\alpha}$ on a $d$-dimensional cube, if the approximation error $\epsilon$ and the number of parallel computing cores (processors) $p$ are given. We choose the best ReLU FNN by minimizing the time complexity in each training iteration. The analysis in this section is valid up to a constant prefactor.

Assume $\phi_{\bm{\theta}}\in \NN(\NNinput=d\NNspace\NNwidthvec=[N]^L\NNspace \NNoutput=1)$, $N,L\in \N^+$, where $\bm{\theta}$ is the vector including all parameters of $\phi_{\bm{\theta}}$. By the basic knowledge of parallel computing (see \cite{Kumar:2002:IPC:600009} for more details), we have the following Table \ref{tab:1}.
\begin{table}[H]
    \caption{Time complexity of one training iteration for an FNN of width $N$ and depth $L$.} 
    \label{tab:1}
    \centering  
    \begin{tabular}{ccc} 
        \toprule
        \multirow{2}{*}[-1pt]{Number of cores $p$}& \multicolumn{2}{c}{Time Complexity}\\
        \cmidrule{2-3}
        & Evaluating $\phi_{\bm{\theta}}(\bm{x})$ & Evaluating $\tfrac{\partial \phi_{\bm{\theta}}(\bm{x})}{\partial \bm{\theta}}$ \\
        \midrule% \midrule[0.6pt]  指定宽度
        $p\in [1,N]$ & $\mathcal{O}(N^{2}L/p)$  & $\mathcal{O}(N^{2}L/p)$\\
        $p\in (N,N^2]$ & $\mathcal{O}\big(L(N^{2}/p+\ln\tfrac{p}{N})\big)$  & $\mathcal{O}\big(L(N^{2}/p+\ln\tfrac{p}{N})\big)$\\
        $p\in (N^2,\infty)$ & $\mathcal{O}(L\ln N)$ & $\mathcal{O}(L\ln N)$ \\
        \bottomrule% \bottomrule[1.2pt]  指定宽度
    \end{tabular} 
\end{table}

For the sake of simplicity, we assume that the training batch size is $\calO(1)$. Denote the time complexity of each training iteration as $T(n,L)$, then 
\begin{equation*}
T(N,L)=\left\{
\begin{array}{ll}
\mathcal{O}(N^{2}L/p),    &   p\in [1,N]  , \\
\mathcal{O}\big(L(N^{2}/p+\ln\tfrac{p}{N})\big),   &   p\in (N,N^2] , \\
\mathcal{O}(L\ln N) ,  &   p\in (N^2,\infty).   \\
\end{array}
\right.
\end{equation*}

Theorem
\ref{thm:main} and \ref{thm:lowInfty} imply that the approximation error $\epsilon$ is essentially $\calO((NL)^{-2\alpha/d})$. 
Hence, we can get the optimal size of ReLU FNNs via the optimization problem below:
\begin{equation}
\label{eq:minProblem1}
\begin{split}
&\hspace{-3pt}(N_{\tn{opt}},L_{\tn{opt}})=\mathop{\tn{arg}\, \tn{min}}\limits_{N,\ L} \ T(N,L) \\
\tn{ subject to }\
&\left\{\begin{array}{l}
\ \vspace{-15pt}\\
 \epsilon = \calO\big((NL)^{-2\alpha/d}\big),\\
\ \vspace{-8pt}\\
N,L,p\in   N^+.\\
\ \vspace{-15pt}\\
\end{array}
\right.
\end{split}
\end{equation}
To simplify the discussion, we have the following assumptions:
\begin{itemize}
    \item Dropping the notation $\calO(\cdot)$ sometimes while assuming asymptotic analysis with the abuse of notations.
    \item $N$, $L$, and $p$ are allowed to be real numbers.
    \item We denote $\epsilon= (NL)^{-2\alpha/d}$ since the approximation rate $\calO\big((NL)^{-2\alpha/d}\big)$ is both attainable and nearly optimal.
\end{itemize}

With $\epsilon = (NL)^{-2\alpha/d}$, we have
\begin{equation}
\label{eq:overlineT}
\begin{split}
\overline{T}(N,L)&\coloneq \left\{
\begin{array}{ll}
N^{2}L/p    &   p\in [1,N]  , \\
L(N^{2}/p+\ln\tfrac{p}{N}),   &   p\in (N,N^2] , \\
L(1+\ln N),  &   p\in [N^2,\infty),   \\
\end{array}
\right.\\
&=
\left\{
\begin{array}{ll}
N\epsilon^{-d/(2\alpha)}/p ,   &   N\in [p,\infty)  , \\
N\epsilon^{-d/(2\alpha)}/p+\tfrac{1}{N}\epsilon^{-d/(2\alpha)}\ln\tfrac{p}{N},   &  N \in [\sqrt{p},p) , \\
\tfrac{1+\ln N}{N}\epsilon^{-d/(2\alpha)},  &   N\in [1,\sqrt{p}).   \\
\end{array}
\right.
\end{split}
\end{equation}
Then we get $T(N,L)=\mathcal{O}(\overline{T}(N,L))$.
Therefore, the optimization problem in  Equation \eqref{eq:minProblem1} can be simplified to
\begin{equation}
\label{eq:minProblem2}
\begin{split}
&\hspace{-3pt}(N_{\tn{opt}},L_{\tn{opt}})=\mathop{\tn{arg}\, \tn{min}}\limits_{N,\ L} \ \overline{T}(N,L) \\
&\hspace{-1pt}\tn{ subject to }\
\left\{\begin{array}{l}
\ \vspace{-15pt}\\
 \epsilon =(NL)^{-2\alpha/d},\\
\ \vspace{-8pt}\\
N,L,p\in  [1,\infty).\\
\ \vspace{-15pt}\\
\end{array}
\right.
\end{split}
\end{equation}
By Equation \eqref{eq:overlineT}, $\overline{T}(N,L)$ is independent of $L$ on the condition that $\epsilon = (NL)^{-2\alpha/d}$. Therefore, we may denote $\overline{T}(N,L)$ by $\overline{T}(N)$. Now we consider two cases: the case $p=\calO(1)$ and the case $p\gg \calO(1)$.

\mycase{1}{The case $p=\calO(1)$.}
It is clear that $\overline{T}(N)$ is increasing in $N$ when $N\in [p,\infty)$ by Equation \eqref{eq:overlineT}. Together with $p=\calO(1)$, then $\calO(\sqrt{p})=\calO(p)=\calO(1)$. Therefore,  $N_{\tn{opt}}=\calO(1)$ and $L_{\tn{opt}}=\calO(\epsilon^{-d/(2\alpha)})$. Note that we regard $d$ as a constant ($\calO(1)$) in above analysis, $N_{\tn{opt}}$ should be $\calO(d)$ in fact.

\mycase{2}{The case  $p\gg\calO(1)$.}
Since $\epsilon = (NL)^{-2\alpha/d}$, we have $N\le \epsilon^{-d/(2\alpha)}$. We only need to consider the monotonicity of $\overline{T}(N)$ on $[1,\epsilon^{-d/(2\alpha)}]$. Together with  Equation \eqref{eq:overlineT}, this case can be divided into two sub-cases:
the sub-case $\sqrt{p}\le \epsilon^{-d/(2\alpha)}$ and the sub-case $\sqrt{p}> \epsilon^{-d/(2\alpha)}$. 

\mycase{2.1}{The sub-case $\sqrt{p}> \epsilon^{-d/(2\alpha)}$.}
$\sqrt{p}> \epsilon^{-d/(2\alpha)}$ implies $[1,\epsilon^{-d/(2\alpha)}]\subseteq [1,\sqrt{p}]$. Hence, $\overline{T}(N)$ is decreasing in $N$ on $[1,\epsilon^{-d/(2\alpha)}]$. It follows that $N_{\tn{opt}}=\calO(\epsilon^{-d/(2\alpha)})$ and that $L_{\tn{opt}}=\calO(1)$.

\mycase{2.2}{The sub-case  $\sqrt{p}\le \epsilon^{-d/(2\alpha)}$.}
For this sub-case, $N_{\tn{opt}}$ and $N_{\tn{opt}}$ are hard to estimate. However, we can give a rough range of $N_{\tn{opt}}$. Since $\overline{T}(N)$ is decreasing in $N$ on $[1,\sqrt{p}]$ and increasing in $N$ on $[p,\infty)$, the minimum of $\overline{T}(N)$ is achieved on $[\sqrt{p},p]$. Hence, $N_{\tn{opt}}\in [\calO(\sqrt{p}),\calO(p)]\cap [\calO(\sqrt{p}),\calO(\epsilon^{-d/(2\alpha)})]$ and $L_{\tn{opt}}=\calO(\epsilon^{-d/(2\alpha)}/N_{\tn{opt}})$.

\section{Conclusion and future work}
\label{sec:conclusion}
This paper aims at a quantitative and optimal approximation rate of ReLU FNNs in terms of both width and depth simultaneously to approximate continuous functions. It was shown that ReLU FNNs with width $\calO(N)$ and depth $\calO(L)$ can approximate an arbitrary continuous function on a $d$-dimensional cube with an approximation rate $19\sqrt{d}\,\omega_f( N^{-2/d}L^{-2/d})$. In particular, when $f$ is a H\"older continuous function of order $\alpha$ with a H\"older constant $\lambda  $, the approximation rate is $19\sqrt{d}\,\lambda  N^{-2\alpha/d}L^{-2\alpha/d}$ and it is nearly asymptotically tight. We also extended our analysis to the case when the domain of $f$ is irregular and showed the same approximation rate. In practical applications, it is usually believed that real data are sampled from an $\epsilon$-neighborhood of a $d_{\mathcal{M}}$-dimensional smooth manifold $\mathcal{M}\subset [0,1]^d$ with $d_{\mathcal{M}}\ll d$. In the case of an essentially low-dimensional domain, we show an approximation rate
\[2\omega_f\big(\tfrac{2\epsilon}{1-\delta}\sqrt{\tfrac{d}{d_\delta}}+2\epsilon\big)+19\sqrt{d}\,\omega_f(\tfrac{2\sqrt{d}}{(1-\delta)\sqrt{d_\delta}}N^{-2/d_\delta}L^{-2/d_\delta})\]
%$2\omega_f\big(\tfrac{4\epsilon}{1-\delta}\sqrt{\tfrac{d}{d_\delta}}\big)+19\sqrt{d}\,\omega_f(\tfrac{2\sqrt{d}}{(1-\delta)\sqrt{d_\delta}}N^{-2/d_\delta}L^{-2/d_\delta})$ 
for ReLU FNNs to approximate $f$ in the $\epsilon$-neighborhood, $d_\delta=\calO\big(d_{\mathcal{M}}\tfrac{\ln (d/\delta)}{\delta^2}\big)$ for any given $\delta\in(0,1)$. 

Besides, we studied how to select the best ReLU FNN to approximate continuous function in parallel computing. In particular, ReLU FNNs with depth $\calO(1)$ are the best choices if the number of parallel computing cores $p$ is sufficiently large. ReLU FNNs with width $\calO(d)$ are best choices if $p=\calO(1)$. The width of best ReLU FNNs is between $\calO(\sqrt{p})$ and $\calO(p)$ if $p$ is moderate. 

We would like to remark that our analysis was based on the fully connected feed-forward neural networks and the ReLU activation function. It would be very interesting to generalize our conclusions to neural networks with other types of architectures (e.g., convolutional neural networks) and activation functions (e.g., tanh and sigmoid functions). %Another important direction is to extend our results to the $L^\infty$-norm on the whole domain without the trifling  region. 
Besides, if identity maps are allowed in the construction of neural networks as in the residual networks \cite{7780459}, the size of FNNs in our construction can be further optimized. Finally, the proposed analysis could be generalized to other function spaces with explicit formulas to characterize the approximation error. These will be left as future work.

\section*{Acknowledgments}
Z.~Shen is supported by Tan Chin Tuan Centennial Professorship.  H.~Yang was partially supported by the US National Science Foundation under award DMS-1945029.
%%%%%%%%%%%%%%%%%%%%%%%%%%%%%%%%%%%%%%%%%%%%%%%%%%%%%%%%%%%%%%%
%\clearpage
\bibliographystyle{siam}
\bibliography{references}
%%%%%%%%%%%%%%%%%%%%%%%%%%%%%%%%%%%
\end{document}